\documentclass[11pt, reqno]{amsart}
\usepackage{amsmath}
\usepackage{amssymb}
\usepackage{amstext}
\usepackage{mathrsfs}
\usepackage{a4wide}
\usepackage{graphicx}
\usepackage{bbm}
\usepackage{verbatim}
\usepackage{bm}
\usepackage{graphicx}
\usepackage{tikz}
\usepackage{pgfplots}
\usepackage{titletoc}

\allowdisplaybreaks[4]
\allowdisplaybreaks \numberwithin{equation}{section}
\usepackage{color}
\usepackage{cases}
\usepackage{hyperref}
\hypersetup{hypertex=true,
	colorlinks=true,
	linkcolor=black,
	anchorcolor=black,
	citecolor=black}

 \usepackage{geometry}
\numberwithin{equation}{section}

\newtheorem{theorem}{Theorem}[section]
\newtheorem{proposition}[theorem]{Proposition}
\newtheorem{corollary}[theorem]{Corollary}
\newtheorem{lemma}[theorem]{Lemma}

\theoremstyle{definition}

\theoremstyle{remark}
\newtheorem{remark}[theorem]{Remark}

\newcommand{\ep}{\varepsilon}
\newcommand{\Om}{\Omega}
\newcommand{\la}{\lambda}

\newcommand{\vertiii}[1]{{\left\vert\kern-0.25ex\left\vert\kern-0.25ex\left\vert #1
		\right\vert\kern-0.25ex\right\vert\kern-0.25ex\right\vert}}

\begin{document}

	\title	[Uniqueness and stability of vortex rings]{Uniqueness and stability of  steady vortex rings for 3D incompressible Euler equation}
	
	\author{Daomin Cao, Shanfa Lai, Guolin Qin, Weicheng Zhan, Changjun Zou}
	
	\address{Institute of Applied Mathematics, Chinese Academy of Sciences, Beijing 100190, and University of Chinese Academy of Sciences, Beijing 100049,  P.R. China}
	\email{dmcao@amt.ac.cn}
	
	\address{Institute of Applied Mathematics, Chinese Academy of Sciences, Beijing 100190, and University of Chinese Academy of Sciences, Beijing 100049,  P.R. China}
	\email{laishanfa@amss.ac.cn}
	
	\address{Institute of Applied Mathematics, Chinese Academy of Sciences, Beijing 100190, and University of Chinese Academy of Sciences, Beijing 100049,  P.R. China}
	\email{qinguolin18@mails.ucas.ac.cn}

	\address{School of Mathematical Sciences, Xiamen University, Xiamen, Fujian, 361005, P.R. China }
	\email{zhanweicheng@amss.ac.cn}	
	
	\address{Department of Mathematics, Sichuan University, Chengdu, Sichuan, 610064, P.R. China}
	\email{zouchangjun@amss.ac.cn}

	\begin{abstract}
		In this paper, we are concerned with  the uniqueness and nonlinear stability of vortex rings for the 3D Euler equation. By utilizing Arnold 's variational principle for steady states of Euler equations  and concentrated compactness method  introduced by  P. L. Lions, we first establish a  general stability criteria for vortex rings in rearrangement classes, which allows us to reduce  the stability  analysis  of certain vortex rings to the problem of their uniqueness. Subsequently, we prove the uniqueness of  a special  family of vortex rings with a small cross-section and polynomial type distribution function.  These vortex rings correspond to global  classical  solutions to the 3D Euler equation and have been shown to exist by many celebrate works.  The proof is achieved by studying carefully asymptotic behaviors of vortex rings as they tend to a circular filament and applying local Pohozaev identities.  Consequently, we provide the first family of nonlinear stable classical vortex ring solutions  to the 3D Euler equation.
	\end{abstract}
	
	\maketitle{\small{\bf Keywords:} The 3D Euler equation; Steady vortex rings;  Uniqueness; Nonlinear stability; Variational method; Semilinear elliptic equation.  \\
		
		}
	\setcounter{tocdepth}{3}

	   	
\section{Introduction and main results}\label{sec1}
The motion of an inviscid incompressible flow in $\mathbb{R}^3$ is described by the 	Euler equations
\begin{align}\label{1-1}
	\begin{cases}
		\partial_t\mathbf v+(\mathbf v\cdot \nabla)\mathbf v=-\nabla P,
		\\
		\nabla\cdot\mathbf v=0,
	\end{cases}
\end{align}
where  $\mathbf{v}(x,t)=(v_1, v_2, v_3)$ is the velocity field	and $P(  x,t)$ is the scalar pressure. Note that the incompressibility of the flow enables us to introduce a velocity vector potential  $\pmb\psi:\mathbb{R}^3\to \mathbb{R}^3$ such that $\mathbf{v}= \nabla\times \pmb\psi$. Let  $\pmb{\omega}:=\nabla\times\mathbf{v}$ be the  vorticity. Taking curl of the first equation in Euler equations \eqref{1-1}, we obtain the following equations for vorticity
\begin{align}\label{1-2}
	\begin{cases}
		\partial_t \pmb{\omega}+(\mathbf{v}\cdot\nabla)\pmb{\omega}=(\pmb{\omega}\cdot\nabla)\mathbf{v},
		\\
		\mathbf{v}=\nabla\times \pmb\psi,  \quad  -\Delta \pmb\psi=\pmb{\omega}.
	\end{cases}
\end{align}

	We are interested in  a special family of global solutions to \eqref{1-1}, the  vortex rings, which means axi-symmetric flows with   vorticity $\pmb{\omega}$  supported in toroidal regions. By choosing $x_3$-axis as the axis of symmetry, in the usual cylindrical coordinate frame $\{\mathbf{e}_r, \mathbf{e}_\theta, \mathbf{e}_z\}$ with $\mathbf{e}_z=\mathbf{e}_3$, the axi-symmetric velocity field $\mathbf{v}$  does not depend on the $\theta$ coordinate and can be expressed by
\begin{equation*}
	\mathbf{v}=v^r(r,z)\mathbf{e}_r+v^\theta(r,z)\mathbf{e}_\theta+v^z(r,z)\mathbf{e}_z.
\end{equation*}
The component $v^\theta$ in the $\mathbf{e}_\theta$ direction is  called the swirl velocity. If an axi-symmetric flow is non-swirling (i.e., $v^\theta \equiv 0$), then the vorticity admits its angular component $\omega^\theta$ only, namely, $\pmb{\omega}=\omega^\theta \mathbf{e}_\theta$.  Let $\zeta=\omega^\theta/r$ be the potential vorticity and $\psi=r\pmb\psi\cdot \mathbf{e}_\theta$. Then the vorticity equation \eqref{1-2} is reduced to an active scalar equation for $\zeta$ (see e.g. \cite{MB})
\begin{equation}\label{1-3}
	\begin{cases}
		\partial_t \zeta+\mathbf{v}\cdot\nabla\zeta=0,\\
		\mathbf{v}=\frac{1}{r}\left(-\frac{\partial\psi}{\partial z}\mathbf{e}_r+\frac{\partial\psi}{\partial r}\mathbf{e}_z\right)\\
		\mathcal{L}\psi=\zeta,
	\end{cases}
\end{equation}
where $\nabla\zeta :=\frac{\partial\zeta}{\partial r}\mathbf{e}_r+\frac{\partial\zeta}{\partial z}\mathbf{e}_z$ and
\begin{equation*}
	\mathcal{L}:=-\frac{1}{r}\frac{\partial}{\partial r}\Big(\frac{1}{r}\frac{\partial}{\partial r}\Big)-\frac{1}{r^2}\frac{\partial^2}{\partial z^2}.
\end{equation*}

\subsection{Vortex rings for the 3D Euler equation}
By a \emph{steady vortex ring} we mean a vortex ring that moves vertically at a constant speed forever without changing its shape or size. In other words, a steady vortex ring is of the form
\begin{equation}\label{1-4}
	\zeta(  x,t)=\zeta(  x+t\mathbf{v}_\infty),
\end{equation}
where $\mathbf{v}_\infty=-\mathcal W\mathbf{e}_z$ is a constant propagation speed. Substituting \eqref{1-4} into \eqref{1-3}, we arrive at a stationary equation
\begin{equation}\label{1-5}
	\begin{cases}
		\nabla^\perp \left(\psi -\frac{\mathcal W}{2} r^2\right) \cdot \nabla \zeta=0,\\
		\mathcal{L}\psi=\zeta,
	\end{cases}
\end{equation}
where $\nabla^\perp h:= -\frac{\partial h}{\partial z}\mathbf{e}_r+\frac{\partial h}{\partial r}\mathbf{e}_z$.
The key observation is that if $\zeta=\lambda f\left(\psi -\frac{\mathcal W}{2} r^2-\mu\right)$ for some $C^1$ function $f: \mathbb R \to \mathbb R$ and constants $\lambda, \mu\in \mathbb R$, then the first equation holds automatically. In this case, \eqref{1-5} is thus reduced to a study of the \emph{semilinear elliptic problem} on the half plane $\Pi:=\{(r,z)\mid r>0,\,z\in \mathbb R \}$
 \begin{equation}\label{1-6}
 		\mathcal{L}\psi=\lambda f\left(\psi -\frac{\mathcal W}{2} r^2-\mu\right).
 \end{equation}
The function $f$ in \eqref{1-6} is   referred to as the vorticity function or profile function in this paper.

Mathematical analysis for the vortex rings has been a subject of active studies for a long time and  traces back to the works of Helmholtz \cite{Hel} in 1858 and Lord Kelvin \cite{Tho} in 1867.  Kelvin and Hicks  found a formula to show that if the vortex ring with circulation $\kappa$ has radius $r_*$ and its cross-section $\varepsilon$ is small, then the vortex ring translates approximately with the velocity (see \cite{Lam, Tho})
\begin{equation}\label{1-7}
	\frac{\kappa}{4\pi r_*}\Big(\log \frac{8r_*}{\varepsilon}-\frac{1}{4}\Big).
\end{equation}
More generally, in addition to the circular vortex filaments, the law of motion of the general
evolved curve has also been extensively studied, which was formally shown to be binormal curvature flow (see e.g. \cite{DR, LC}).  Mathematical verification of the validity of describing the   motion of  general evolved curve by binormal curvature flow is  closely connected  to the well-known \emph{vortex filament conjecture}, for which we refer to \cite{JS}.

In 1894, Hill \cite{Hil} first constructed an explicit particular translating flow of the Euler equation\,(called Hill's spherical vortex) whose support is a ball.  Fraenkel \cite{Fra1, Fra2}  provided a first constructive proof for the existence of a vortex ring concentrated around a torus   with a small  cross-section $\ep>0$.  In 1972, Norbury \cite{Nor72}  constructed a family of  steady vortex rings with constant $\zeta$ that are close to Hill's vortex, which are usually referred to as Norbury's nearly spherical vortex rings. General  existence theory of vortex rings with a given vorticity function was first established by Fraenkel and Berger \cite{BF1}. Existence of steady vortex
rings are also studied in \cite{AS,Ni} by using the mountain pass theorem.  For the construction of vortex rings in Beltrami flows, we refer to \cite{Abe22}.

For special nonlinearities $f$, lots of existence results have been obtained, for which we refer to \cite{CWZ, CWWZ, FT,  VS, YJ} and references therein.  We would like to highlight the work \cite{VS}, in which de Valeriola and Van Schaftingen  constructed steady vortex rings with $f(s)=s_+^p$ for $p>1$ and $\lambda=\frac{1}{\ep^2}$ for every $\ep>0$  by studying the elliptic problem \eqref{1-6}, where $s_+:=\max\{0, s\}$. We remark that the vortex rings obtained in \cite{VS} are classical solutions to the 3D Euler equation since the velocity field is of $C^{1,\alpha}$ for any $\alpha\in(0,1)$. Recently, Cao et al. \cite{CWWZ} established the existence of steady vortex rings with very general $f$ when $\lambda$ is sufficiently large, including the special cases $f(s)=s_+^p$.

The method used in Cao et al. \cite{CWWZ} is an improved vorticity method, which is totally different from the method of de Valeriola and Van Schaftingen \cite{VS}. Therefore, it is natural to ask whether the solutions obtained by different methods are actually the same, which was put forward by Friedman and Turkington in \cite{FT}.  Moreover, as we will see later, uniqueness results are also essential in the study of nonlinear stability, which is another motivation for our investigation on the problem of uniqueness.  One of our main results in this paper is to give a positive answer to the question of uniqueness for a family of vortex rings.

\subsection{Stability  for steady flows}

The stability of steady flows is a classical object in fluid dynamics. In this paper we are interested in the stability of vortex rings for the 3D Euler equation. Since vortex rings can be easily observed experimentally, e.g. when smoke is ejected from a tube, a bubble rises in a liquid, or an ink is dropped in another fluid, and so on,  it is  reasonable to expect their stability in suitable settings.  We refer the reader to \cite{Akh, MGT} for some  historical reviews of the achievements in experimental, analytical, and numerical studies of vortex rings.

Numerous contributions  to stability of steady solutions to 2D Euler equations have been achieved, for which we refer the interested reader to \cite{Abe, Bu5,  Bu6, CQZZ2, CW,  CL,   CJ2,  Wan} and references therein. In recent years, significant breakthroughs have been made in the asymptotic stability of inviscid fluids, see e.g. \cite{BCV, BM, IJ, WZZ} and references therein.
On the other hand,  works on  stability of vortex rings seem to be relatively rare. Recently, Choi  \cite{Choi20} 
established the orbital stability of Hill's vortex by a combination of the variational framework due to \cite{FT}, the uniqueness result \cite{AF} and the concentrated compactness lemma of Lions \cite{Lions}.  Later, nonlinear stability of Norbury's nearly spherical vortex was obtained in \cite{CQZZ, Wan1} by using the local uniqueness result in \cite{AF88}. Recently, uniqueness and stability of vortex rings with small cross-section and uniform density were  proved in \cite {CGPY} and \cite{CQYZZ}. We  would  also  like to mention the work of Gallay and Smets \cite{GSm} on  spectral stability of columnar vortices for the 3D Euler equations.

We note that all the aforementioned works \cite{AF, AF88, CGPY, CQYZZ, CQZZ, Choi20, Wan1} are concerned with uniqueness or stability of vortex rings with the vorticity function $f$ in \eqref{1-6} being of the form $f(s)=1_{s>0}$. Therefore velocity fields of these vortex rings are of $C^\alpha$ with $\alpha\in (0,1)$ and hence are merely weak solutions to the Euler equations  \eqref{1-1}. Therefore, a natural open problem arises as to whether there exist stable classical vortex ring solutions to the 3D Euler equation.

\emph{In this paper, we will establish the first result on the uniqueness and nonlinear stability of a family of vortex rings, whose velocity fields are of  $C^{3}$ and hence   the vorticity fields belong to $C^2$ and are classical solutions to \eqref{1-2}}.

\subsection{Main results and ideas of proof}
Main results of the present paper can be divided into three parts: a general stability criteria for vortex rings in a variational framework (Theorem \ref{thmSs}), which reduces the stability of a given vortex ring to the problem of its uniqueness; the uniqueness of a family of vortex rings with small cross-section (Theorem \ref{thmU}); nonlinear stability of these vortex rings  (Theorem \ref{thmS}).

\subsubsection{A general stability theorem for vortex rings}
 To investigate the nonlinear stability of vortex rings, we first prove a general stability theorem.

  We follow the well-known strategy of  variational method on vorticity exploiting the ideas of Kelvin \cite{Tho} and Arnold \cite{ Ar2,   Ar4}, which was used and further developed by many authors (see e.g. \cite{Ben,  Bu,   Bu1,  BNL13, CWZ, CWWZ,   BF1, FT, GS} and references therein). We use the adaption of Arnold's variational principle due to Benjamin \cite{Ben} and Badiani-Burton \cite{Bad} for vortex rings.  To be precise, let $G_1(x,y)$ be the Green function of $\mathcal{L}$ given by \eqref{Green} below and  define the kinetic energy of the axi-symmetric fluid
\begin{equation*}
	E[\zeta] =\pi\int_{\Pi}\int_{\Pi} \zeta(r,z)G_1(r,z , r',z')\zeta(r',z') rr'dr'dz'drdz,
\end{equation*}
and the impulse
\begin{equation*}
	P[\zeta]=\frac{1}{2}\int_{\mathbb{R}^3}r^2\zeta(  x)d   x=\pi\int_\Pi r^3 \zeta drdz.
\end{equation*}
The works \cite{Bad, Ben} suggested that maximization of $E-\mathcal W P$ over all $\zeta$ in the set of rearrangements  of a given function $\zeta_0$ will give a steady vortex ring with traveling speed $\mathcal W$.  Recall that for a given nonnegative axi-symmetric function $\zeta_0\in L^1(\mathbb{R}^3)$, the set of axi-symmetric rearrangements  of $\zeta_0$ is defined by
	\begin{equation}\label{1-8}\begin{split}
			\mathcal{R}(\zeta_0)=\left\{\right.0\leq \zeta\in L^1(\mathbb{R}^3)\,\mid \,   &\zeta(x)=\zeta(r,z)   \\   &\text{and}\    |\{x: \zeta(x)>\tau\}|_{\mathbb{R}^3} =|\{x: \zeta_0(x)>\tau\}|_{\mathbb{R}^3} , \forall\, \tau>0 \left. \right\},
		\end{split}	
\end{equation}
where $|\cdot|_{\mathbb{R}^3} $ stands for the Lebesgue measure in $\mathbb{R}^3$. For $\zeta_0\in L^1\cap L^2(\mathbb{R}^3)$, let  $\overline{\mathcal{R}(\zeta_0)^w}$ be the closure of $\mathcal{R}(\zeta_0)$ in the weak topology of $L^2_{as}(\mathbb R^3)$, the axi-symmetric subspace of $L^2(\mathbb R^3)$. A key point proved by Douglas \cite{Dou} is that the set  $\overline{\mathcal{R}(\zeta_0)^w}$ is convex and weakly compact.

Let $\zeta_0\in L^1\cap L^q(\mathbb{R}^3)$ for some $q>2$ and define $\mathcal{E}_{\mathcal W}:=\frac{1}{2\pi}(E-\mathcal W P)$. We consider the following maximization problem
\begin{equation}\label{1-9}
	\sup_{\zeta\in \overline{\mathcal{R}(\zeta_0)^w}} \mathcal{E}_{\mathcal W}(\zeta).
\end{equation}
According to ideas in \cite{Bad, Ben}, each maximizer  of the above problem are  expected to correspond to a steady vortex ring.

To state our stability theorems, we need the following existence result for the Cauchy problem for the 3D Euler equation due to Lemma 3.4 in \cite{Choi20}.
\begin{proposition}\label{Pro1}
	Define the weighted space $L^1_\text{w}(\mathbb R^3)$ by
	$$L^1_\text{w}(\mathbb R^3)=\{\vartheta: \mathbb R^3 \to \mathbb R\ \text{measurable} \mid r^2\vartheta\in L^1(\mathbb R^3)\}.$$
	For any non-negative axi-symmetric function $\zeta_0\in L^1\cap L^\infty\cap L^1_\mathrm{w}(\mathbb R^3)$ satisfying $r\zeta_0\in L^\infty(\mathbb R^3)$, there exists a unique weak solution $\zeta\in BC([0,\infty);L^1\cap L^\infty\cap L^1_\mathrm{w}(\mathbb R^3))$ of \eqref{1-3} for the initial data $\zeta_0$ such that $\zeta(\cdot,t)\ge 0$ and is axi-symmetric,
	\begin{equation*}
		\begin{array}{ll}
			\|\zeta(\cdot,t)\|_{L^p(\mathbb{R}^3)} =\|\zeta_0\|_{L^p(\mathbb{R}^3)},\ \ 1\le p\le \infty, &\\
			\mathcal{P}[\zeta(\cdot,t)] = \mathcal{P}[\zeta_0], \,\,\,E[\zeta(\cdot,t)] =E[\zeta_0],\ \ \ \text{for all}\ t>0,&
		\end{array}
	\end{equation*}
	and, for any $0<\upsilon_1<\upsilon_2<\infty$ and for all $t>0$,
	\begin{equation*}
		\int_{\{ x\in\mathbb{R}^3\,\,\mid\,\, \upsilon_1<\zeta(x,t)<\upsilon_2\}}\zeta(x,t)d x=\int_{\{ x\in\mathbb{R}^3\,\,\mid\,\, \upsilon_1<\zeta_0(x)<\upsilon_2\}}\zeta_0(x)d x.
	\end{equation*}
\end{proposition}

Let  $\|\cdot\|_{L^1\cap L^2(\mathbb R^3)}:=\|\cdot\|_{L^1(\mathbb R^3)}+\|\cdot\|_{L^2(\mathbb R^3)}$. Our first main result is  the following  stability criteria on the set of maximizers.
\begin{theorem}\label{thmSs}
	Suppose that $0\leq \zeta_0\in L^q(\mathbb{R}^3)$ for some $q>2$ is a non-negative axi-symmetric function with a   support of finite volume $ |\mathrm{supp}(\zeta_0) |_{\mathbb{R}^3}<\infty $. Let $\mathcal W>0$ be a positive constant. Suppose that $\Sigma_{\zeta_0,\mathcal W} $,  the set of maximizers of $\mathcal{E}_{\mathcal W}$ over $\overline{\mathcal{R}(\zeta_0)^w}$, satisfies $$\emptyset\not= \Sigma_{\zeta_0,\mathcal W}\subset \mathcal{R}(\zeta_0).$$  Then $\Sigma_{\zeta_0,\mathcal W}$ is orbitally stable in the sense that for arbitrary $\epsilon>0$, there exists $\delta>0$ such that if the nonnegative axi-symmetric function $\omega_0  \in L^1\cap L^\infty\cap L^1_\mathrm{w}(\mathbb R^3)$ satisfies $r\omega_0\in L^\infty(\mathbb R^3)$  and
	\begin{equation*}
		\inf_{\zeta\in \Sigma_{\zeta_0,\mathcal W}}   \{\|  \omega_0-\zeta\|_{L^1\cap L^2(\mathbb{R}^3)}+|P(\omega_0-\zeta)| \} \leq \delta,
	\end{equation*}
	then for all $t>0$, we have
	\begin{equation*}
		\inf_{\zeta\in \Sigma_{\zeta_0,\mathcal W}}    \|  \omega(t)-\zeta\|_{L^1\cap L^2(\mathbb{R}^3)}  \leq \epsilon.
	\end{equation*}
    If in addition $$P(\zeta_1)=P(\zeta_2),\quad \forall\  \zeta_1, \zeta_2\in \Sigma_{\zeta_0,\mathcal W}, $$
   then for all $t>0$, we have
\begin{equation*}
	\inf_{\zeta\in \Sigma_{\zeta_0,\mathcal W}} \left\{\|  \omega(t)-\zeta\|_{L^1\cap L^2(\mathbb{R}^3)} +P(|\omega(t)-\zeta|)\right\}\leq \epsilon.
\end{equation*}
Here $\omega(t)$   means the corresponding solution of \eqref{1-3} with initial data $\omega_0$ given by Proposition \ref{Pro1}.
\end{theorem}	
\begin{remark}\label{rem1-7}
	Compared with the stability result in \cite{BNL13} for vortex pairs,  Theorem \ref{thmSs} admits perturbations with non-compact supports. This is achieved by   introducing    the $L^1$-norm in our theorem.
\end{remark}

The proof of Theorem \ref{thmSs} is a combination of the variational method and the concentrated compactness lemma of Lions \cite{Lions}. 

It can be seen that Theorem \ref{thmSs} is a very general result due to the weak requirements on $\zeta_0$. However, for a given vortex ring solution to the Euler equation, it is still a challenging problem to apply Theorem \ref{thmSs} to verify its stability, since the structure of the set of maximizers is in general hard to determine. Note that the problem is invariant under translations in   $z$-direction. The best situation we can expect, of course, is that   the set of maximizers $\Sigma_{\zeta_0,\mathcal W}$ consists of exactly one single orbit under translation, which apparently gives nonlinear orbital stability of the vortex ring. This can be reduced to the problem of uniqueness of vortex ring in suitable setting, which will be our next  topic.

\subsubsection{Uniqueness of a family of vortex rings}
 In comparison to the abound results on the existence of vortex rings as mentioned above, there are relatively fewer results on the uniqueness. One of the initial results  was presented by Amick and Fraenkel  \cite{AF}, who demonstrated that Hill's vortex is the unique solution when viewed in a natural weak formulation by using the method of moving planes.  Subsequently,  Amick and Fraenkel \cite{AF88} also established local uniqueness for Norbury's nearly spherical vortex. More recently, Cao et al. \cite{CQYZZ}  showed the uniqueness of steady vortex rings of small cross-section and uniform density.

 It can be seen that all the aforementioned works \cite{AF, AF88, CQYZZ} are concerned with uniqueness of vortex rings of ``patch" type. We will provide the first uniqueness result of vortex rings of non-patch type in the literature.

 To state our main results, we need to introduce some notations. We shall say that a scalar function $\vartheta:\mathbb R^3\to \mathbb R$ is axi-symmetric if it has the form of $\vartheta( x,y,z)=\vartheta(r,z)$, and a subset $\Omega\subset \mathbb R^3$ is axi-symmetric if its characteristic function $ 1_\Omega$ is axi-symmetric. The cross-section parameter $\sigma$ of an axi-symmetric set $\Omega\subset \mathbb R^3$ is defined by
\begin{equation*}
	\sigma(\Omega):=\frac{1}{2} \sup\left\{  \delta_{z}(  x, y)\,\,|\,\,  x,  y\in \Omega \right\},
\end{equation*}
where the axi-symmetric distance $ \delta_z$ is given by
\begin{equation*}
 \delta_z(  x,  y):=\inf\left\{| x-Q( y)|\,\,\,\,| \,\, \ Q \ \text{is a rotation around}\ \mathbf{e}_z\right\}.
\end{equation*}
Let $\mathcal{C}_{r,z}=\{  x=(x_1,x_2,x_3)\in\mathbb{R}^3~|x_1^2+x_2^2=r^2,x_3=z\}$ be a circle of radius $r$ on the plane perpendicular to $\mathbf{e}_z$. For an axi-symmetric set $\Omega\subset \mathbb{R}^3$, we define the axi-symmetric distance between $\Omega$ and $\mathcal C_{r,z}$ as follows
\begin{equation*}
	\text{dist}_{\mathcal{C}_{r,z}}(\Omega)=\sup_{ x\in \Omega}\inf_{  y\in{\mathcal{C}_{r,z}}}|  x-y|.
\end{equation*}
The circulation of a steady vortex ring $\zeta$ is given by
\begin{equation*}
	\frac{1}{2\pi}\int_{\mathbb{R}^3}\zeta(  x)d  x=\int_\Pi \zeta(r,z) r dr dz.
\end{equation*}
Let $G_1(r,z, r',z')$ be the Green function for the operator $\mathcal L$ defined below \eqref{1-3} and denote $\mathcal{G}_1 \zeta:=G_1\ast \zeta$. We first state   the following existence result due to \cite{CWWZ, VS}.
\begin{proposition}[Existence, \cite{CWWZ, VS}]\label{thmE}
	Let $\kappa>0$, $W>0$ and $p\geq 2$ be three positive numbers. Then there exists $\ep_1>0$ such that there exist a family of  steady vortex rings $\{\zeta_\varepsilon\}_{0<\ep<\ep_1}$  with the same circulation $\kappa$ and traveling speed $W\ln \frac{1}{\ep}$. Moreover $\zeta_\varepsilon$ satisfies the following properties:
	\begin{itemize}
		\item [(i)]$\zeta_\varepsilon=\varepsilon^{-2}\left(\mathcal{G}_1\zeta_\ep-\frac{W}{2}x_1^2\ln \frac{1}{\ep}-\mu_\ep \right)_+^p$ for some constant  $\mu_\ep $;
		\item [(ii)]As $\varepsilon\to 0$, $\limsup_{\ep\to0}\sigma\left(\Omega_\varepsilon\right)\ep^{-1}\leq C  $ for a constant $C >0$, where $\Om_\ep:=\{x\in \mathbb{R}^3\mid \zeta_\ep(x)>0\}$.
		\item [(iii)] It holds $\mathrm{dist}_{\mathcal{C}_{\frac{\kappa}{4\pi W}, 0}}(\Om_\ep)\to 0$ as $\ep\to 0$.
	\end{itemize}
\end{proposition}
\begin{remark}\label{rem1}
 Although  the circulation of the vortex rings obtained in \cite{VS} is  $\kappa_\varepsilon\to \kappa$, which may be  not fixed as $\ep$ varies, one may obtain a family of vortex rings with fixed circulation $\kappa$ by  a scaling argument.
\end{remark}

Our second main result shows that the vortex rings constructed in \cite{CWWZ, VS} are actually the same (after some scaling) when   $\ep$ is sufficiently small.

\begin{theorem}\label{thmU}
	Let $\kappa>0$, $W>0$ and $p\geq 2$ be three positive numbers. Suppose that two families of steady vortex rings $\{\zeta^{(1)}_\varepsilon\}_{\ep>0}$ and $\{\zeta^{(2)}_\varepsilon\}_{\ep>0}$ with the same circulation $\kappa$ satisfy the following conditions:
	\begin{itemize}
		\item [(i)]$\zeta^{(i)}_\varepsilon=\varepsilon^{-2}\left(\mathcal{G}_1\zeta_\ep^{(i)}-\frac{W}{2}x_1^2\ln \frac{1}{\ep}-\mu_\ep^{(i)}\right)_+^p$ for some constants $\mu_\ep^{(i)}$, $i=1,2$;
		\item [(ii)]As $\varepsilon\to 0$, $\limsup_{\ep\to0}\sigma\left(\Omega^{(i)}_\varepsilon\right)\ep^{-1}\leq C_0 $ for a constant $C_0>0$, where $\Om_\ep^{(i)}:=\{x=(x_1,x_2,x_3)\in \mathbb{R}^3\mid \zeta_\ep^{(i)}(x)>0\}$ for $i=1,2$.
		\item [(iii)]There exist   $(r_*^{(i)},z_*^{(i)})\in \Pi$ such that $\mathrm{dist}_{\mathcal{C}_{r_*^{(i)},z_*^{(i)}}}(\Om_\ep^{(i)})\to 0$ as $\ep\to 0$, $i=1,2$.
	\end{itemize}
Then there exists $\ep_0>0$ such that for any $  \varepsilon \in (0,\varepsilon_0) $ there is a constant $c_\ep\in \mathbb{R}$ so that $$\zeta^{(1)}_\varepsilon(r,z)\equiv\zeta^{(2)}_\varepsilon(r,z+c_\ep).$$

\end{theorem}
\begin{remark}\label{rem2}
	The velocity fields of $\zeta_\ep$ in Proposition \ref{thmE} and Theorem \ref{thmU} are  obviously of  $C^{3}$  due to standard regularity theory of elliptic equations and $p\geq2$,   hence the corresponding vorticity fields are classical solutions to the Euler equations \eqref{1-2}.
\end{remark}

The proof of Theorem \ref{thmU} is highly technical and delicate. To obtain the uniqueness, we first establish some rough estimates for the Stokes stream function by direct  blow-up analysis. Then we improve those estimates by approximating the Stokes stream function with careful choice of rescaling of ground state solution  to a corresponding elliptic equation. We will study the error function between the stream function and the approximate solution via  the linearized equation, where a combination of the non-degeneracy of the ground state solution and method of moving planes is used. Unlike the case of vortex rings with uniform density considered in \cite{CQYZZ},  approximation by rescaling of ground state solution is not accurate enough. Thus, we need to introduce \emph{ the second approximation procure} in this paper to obtain necessary estimates, which was done by a existence theory of the linearized equation with odd nonlinearities given in Appendix \ref{appD}.  With a delicate estimate in hand, we are able to apply a local Pohozaev identity  to derive a contradiction if there are two different vortex rings satisfying all the assumptions in Theorem \ref{thmU}. In addition to proving Theorem \ref{thmU},  our approach may be applied to the problem of  uniqueness  of `thin' vortex rings with other vorticity function $f$.

Theorem \ref{thmU} is essential to  the study of nonlinear stability, which will be our last main result.

\subsubsection{Stability  of a family of vortex rings}
Steady vortex rings $\zeta_\ep$ in Proposition \ref{thmE} were constructed in \cite{CWWZ} as a special case via the vorticity method, which are maximizers of the following problem
\begin{equation}\label{1-11}
	\sup_{\zeta \in \mathcal{A}}\left\{\frac{1}{2\pi}(E-W\ln\frac{1}{\ep} P)-\frac{1}{\varepsilon^2}\int_{D_0} G(\varepsilon^2\zeta)d\nu\right\},
\end{equation}
where
\begin{equation*}
	\mathcal{A}:=\left\{\zeta\in L^\infty(\Pi)~\bigg|~\int_{\Pi}r\zeta dr dz\le \kappa,\ \text{supp}(\zeta)\subseteq D_0\right\},
\end{equation*}
  $$D_0=\{(r,z)\in \Pi\mid \kappa/(8\pi W)\le r\le \kappa/(4\pi W)+1, -1\le z\le 1\},$$
  and
\begin{equation*}
	G(s)=\sup_{s'\in\mathbb{R}}\left[ss'-F(s')\right]
\end{equation*}
is the conjugate function to $F(s):=\int_{0}^{s}f(t)dt$. Here, $f$ is the vorticity function in \eqref{1-6}.

The maximization problem \eqref{1-11} is quite different from \eqref{1-9}. To  make a bridge between $\zeta_\ep$ and Theorem \ref{thmSs}, we first study the maximization problem \eqref{1-9} with $\zeta_0=\zeta_\ep$ and $\mathcal W=W\ln \frac{1}{\ep}$. We show that the set of maximizers of this problem, denoted by $\Sigma_\ep$,   satisfies the condition $\emptyset\not= \Sigma_\ep\subset \mathcal{R}(\zeta_\ep)$. Thus, we immediately obtain the stability of  $\Sigma_\ep$ by applying Theorem \ref{thmSs}.

In order to obtain the orbital stability  of $\zeta_\ep$, we need to verify
\begin{equation}\label{1-12}
	\Sigma_\ep=\{\zeta_\ep(\cdot+c\mathbf{e}_z)\mid c\in \mathbb R\},
\end{equation}
which can be reduced to a problem of uniqueness after we show that each element  in $\Sigma_\ep$ is a steady vortex ring with cross section of order $\ep$ and concentrate  to a circle in $\mathbb{R}^3$. These asymptotic behaviors of maximizers in $\Sigma_\ep$ are proved by an adapted vorticity method of \cite{CWWZ}. Then, we can reduce the stability of $\zeta_\ep$ to a problem of uniqueness (see Theorem \ref{U-S}) and hence by Theorem \ref{thmU}, in the special case $f(s)=s_+^p$ with $p\geq 2$, we obtain the nonlinear stability of the steady vortex rings $\zeta_\ep$ in Proposition \ref{thmE}. Our last main result reads as follows.
\begin{theorem}\label{thmS}
	For given $\kappa>0$, $W>0$ and $p\geq 2$, let $\zeta_\ep$ be the vortex ring obtained in Proposition \ref{thmE}.  Then there exists $\ep_0>0$ small such that for $0<\ep<\ep_0$,  $\zeta_\ep$ is orbitally stable in the following sense:
	
		For every $\epsilon>0$, there exists $\delta>0$ such that if   $\omega_0 \in L^1\cap L^\infty\cap L^1_\mathrm{w}(\mathbb R^3)$ satisfies $r\omega_0\in L^\infty(\mathbb R^3)$  and
		\begin{equation*}
			\inf_{c\in \mathbb{R}}\Bigg{\{}\|\omega_0-\zeta_{\ep}(\cdot+c\mathbf{e}_2)\|_{L^1\cap L^2(\mathbb{R}^3)}+\left|P(\omega_0-\zeta_{\ep}(\cdot+c\mathbf{e}_2))\right|\Bigg{\}} \leq \delta,
		\end{equation*}
		then  corresponding  solution $\omega(t)$  of \eqref{1-3}  with initial data $\omega_0$ satisfies
		\begin{equation*}
			\inf_{c\in \mathbb{R}} \Big{\{}\|\omega(t)-\zeta_{\ep}(\cdot+c\mathbf{e}_2)\|_{L^1\cap L^2(\mathbb{R}^3)}+P(|\omega(t)-\zeta_{\ep}(\cdot+c\mathbf{e}_2)|) \Big{\}}\leq \epsilon,\ \ \text{for all}\,\, t>0.
		\end{equation*}
\end{theorem}
\begin{remark}
We remark that in the proof of existence and the study of asymptotic behaviors of maximizers in \cite{CWWZ}, the restriction  $\text{supp}(\zeta)\subseteq D_0$ in the definition of $\mathcal A$ is not essential and can be relaxed to $\text{supp}(\zeta)\subseteq \Pi$. Moreover, by virtue of the uniqueness result in Theorem \ref{thmU},   one may also obtain the nonlinear stability of $\zeta_\ep$ using an argument similar to \cite{Abe, CQZZ1}. However, we prefer the approach used in the present paper, which  
has a stronger physical motivation and more extensive applications without prescribing the vorticity function $f$.
\end{remark}
	
	The paper is organized as follows. In Subsection \ref{sec2}, we  prove a general   orbital stability for vortex rings (i.e. Theorem \ref{thmSs}) based on a combination of the variational method and the concentrated compactness lemma due to Lions \cite{Lions}. Section \ref{sec4} is devoted to prove the uniqueness of a family of vortex rings, i.e. Theorem \ref{thmU}. We first study the asymptotic behavior of the steady vortex rings carefully by several steps in Subsections \ref{sec2-1} and \ref{sec2-2} and then prove the uniqueness result Theorem \ref{thmU} via reduction to absurdity and the local Pohozaev identity in Subsection \ref{sec2-3}. In Section \ref{sec3}, we study the  orbital stability of vortex rings in Proposition \ref{thmE}. By considering variational problem in rearrangement class and studying asymptotic behavior of maximizers, we reduce the stability of general concentrated steady vortex ring  into the  problem  of their uniqueness  and  obtain Theorem \ref{thmS} by using Theorem \ref{thmU} as corollary.    Some auxiliary lemmas and complicated calculations are collected in   Appendixes \ref{appB}-\ref{appD}.

In what follows, the symbol $C$ denotes a general positive constant that may change from line to line.

\section{A general stability theorem}\label{sec2}
This section is devoted to investigate the nonlinear stability of vortex rings in general setting and give a proof of Theorem \ref{thmSs}.

In order to simplify notations, we will use  $\mathbb R^2_+=\{  x=(x_1,x_2) \ | \ x_1>0\}$ to substitute the meridional half plane $\Pi$ by writing $(r,z)$ as $(x_1, x_2)$ hereafter. Then  the elliptic operator $\mathcal L:=-\frac{1}{r}\frac{\partial}{\partial r}\Big(\frac{1}{r}\frac{\partial}{\partial r}\Big)-\frac{1}{r^2}\frac{\partial^2}{\partial z^2}$ can be  abbreviated as
\begin{equation}\label{2-1}
	\mathcal L=-\frac{1}{x_1}\text{div}\left(\frac{1}{x_1}\nabla\right).
\end{equation}
 Let ${G_1}( x, y)$ be the Green  function for $\mathcal{L}$ in $\mathbb R^2_+$ given by
\begin{equation}\label{Green}
	{G_1}( x, y):=\frac{x_1 y_1}{4\pi}\int_{-\pi}^\pi\frac{\cos\theta d\theta}{\left[(x_2-y_2)^2+x_1^2+y_1^2-2x_1 y_1\cos\theta\right]^{\frac{1}{2}}}.
\end{equation}
Denoting
\begin{equation}\label{2-6}
	\rho( x, y)=\frac{(x_1-y_1)^2+(x_2-y_2)^2}{x_1 y_1},
\end{equation}
we have the following asymptotic estimates
\begin{equation}\label{2-7}
	{G_1}(x, y)=
	\frac{x_1^{1/2}y_1^{  1/2}}{4\pi}\left(\ln\left(\frac{1}{\rho}\right)
	+2\ln 8-4+O\left(\rho\ln\frac{1}{\rho}\right)\right),\quad \text{as} \ \rho\to 0,
\end{equation}
and
\begin{equation}\label{2-8}
	{G_1}( x, y) =\frac{x_1^{1/2} y_1^{ 1/2}}{4}\left(\frac{1}{\rho^{3/2}}+O(\rho^{-5/2})\right), \quad \text{as} \  \rho\to \infty,
\end{equation}
which can be found in e.g. \cite{Fra2,Lam,Sve}. Thus, by \eqref{2-7} and \eqref{2-8}, the following inequality holds
\begin{equation}\label{2-9}
	{G_1}(x, y)\leq C_\delta
	\frac{(x_1y_1)^{\delta+\frac{1}{2}}}{\left(|x_1-y_1|^2+|x_2-y_2|^2\right)^\delta},\quad \forall \ \delta\in \bigg(0,\frac{3}{2}\bigg], \quad \forall \ x,y\in \mathbb R^2_+.
\end{equation}

 For convenience, we denote the weighted measure with density $x_1$ on $\mathbb{R}^2_+$ by $\nu$. That is, $$d\nu=x_1 dx.$$ We also use $\vertiii{f}_p:=\left(\int_{\mathbb R^2_+}   |f(x)|^pd\nu\right)^{\frac{1}{p}} $ to denote the weighted $L^p$ norm.

Let $\xi$ be a non-negative Lebesgue integrable function on $\mathbb{R}^2_+$, we denote by $\mathcal{R}(\xi)$ the set of (equimeasurable) rearrangements of $\xi$ with respect to $\nu$ on $\mathbb{R}^2_+$ defined by
\begin{equation*}
	\mathcal{R}(\xi)=\Big\{0\leq \zeta\in L^1(\mathbb{R}^2_+,\nu)\Big| \nu (\{x: \zeta(x)>\tau\}) =\nu (\{x: \xi(x)>\tau\}) , \forall\, \tau>0  \Big\}.
\end{equation*}
Note that all functions in $\mathcal{R}(\xi)$ has the same $L^q$ norms with measure $\nu$. Following \cite{Bad, BNL13}, we also define
\begin{equation*}
	\mathcal{R}_+(\xi)=\Big\{\zeta 1_S \big|	\zeta\in\mathcal{R}(\xi),\  \   S\subset \mathbb{R}^2_+ \ \text{measurable}  \Big\},
\end{equation*}
and
$$\overline{\mathcal{R}(\xi)^w}=\left\{ \zeta\geq 0 \  \text{measurable} \Bigg | \  \int_{\mathbb{R}_+^2}(\zeta-\alpha)_+d\nu \leq \int_{\mathbb{R}_+^2}(\xi-\alpha)_+d\nu, \ \ \forall \alpha>0\right\}.   $$
It is easy to see that the inclusions $ 	\mathcal{R}(\xi)\subset \mathcal{R}_+(\xi)\subset \overline{\mathcal{R}(\xi)^w}$ hold. As pointed out by Badiani and Burton  \cite{Bad}, by the results in  \cite{Dou}, one can see that the following properties are true.
\begin{lemma}[Proposition 1.1 and Lemma 2.9 in \cite{Bad}]\label{lem3-1}
	It hold
	\begin{itemize}
		\item[(a).]If $0\leq \zeta_0\in  L^q(\mathbb{R}^2_+,\nu)$ for some $1<q<\infty$, then $\overline{\mathcal{R}(\zeta_0)^w}$ is a weakly compact convex set in $L^q(\mathbb{R}^2_+,\nu)$;
		\item[(b).]If $0\leq \zeta_0\in  L^q(\mathbb{R}^2_+,\nu)$ for some $1<q<\infty$, then $ \mathcal{R}(\zeta_0)$ is dense in $\overline{\mathcal{R}(\zeta_0)^w}$ in the weak topology of $L^q(\mathbb{R}^2_+,\nu)$;
		\item[(c).] If $0\leq \zeta_0\in  L^q(\mathbb{R}^2_+,\nu)$ for some $1\leq q<\infty$, then $$\vertiii{\zeta}_q\leq \vertiii{\zeta_0}_q, \quad \forall \ \  \zeta\in \overline{\mathcal{R}(\zeta_0)^w};$$
		\item[(d).] If $0\leq \zeta_0\in  L^1\cap L^q(\mathbb{R}^2_+,\nu)$ for some $1\leq q<\infty$, then $$\vertiii{\zeta}_s\leq \vertiii{\zeta_0}_s, \quad \forall \ \ 1\leq s\leq q,\quad\forall\ \ \zeta\in \overline{\mathcal{R}(\zeta_0)^w}.$$
	\end{itemize}
\end{lemma}

Denote $\mathcal{G}_1\xi(x)=\int_{\mathbb R^2_+} G_1(x,y)\xi(x) d\nu$.  Recall the definitions (after dividing by $2\pi$) of the  kinetic energy	 
\begin{equation*}
	E(\xi)=\frac{1}{2}\int_{\mathbb R^2_+}\xi(x)\mathcal{G}_1\xi(x) d\nu,
\end{equation*}	
and the impulse
\begin{equation*}
	P(\xi)=\frac{1}{2}\int_{\mathbb R^2_+}\xi(x) x_1^2 d\nu.
\end{equation*}	
For a constant $\mathcal W$, set the energy as $$\mathcal E_{\mathcal W}:=E-\mathcal W P.$$
For a given function $\zeta_0$ and constant $\mathcal W>0$, we will consider the maximization problem
$$\sup_{\zeta\in\overline{\mathcal{R}(\zeta_0)^w}}  \mathcal E_{\mathcal W}(\zeta).$$

The first key step of proving Theorem \ref{thmSs} is to verify the compactness of  maximizing sequences.

\subsection{Compactness of maximizing sequences}
We collect some basic estimates that will be used later.
\begin{lemma}[Lemma 2.8 \cite{Bad}]\label{lem3-2}
	Suppose 	$ \zeta \in L^1\cap L^q(\mathbb{R}_+^2, \nu)$ for some $2< q<+\infty$, then there exists a constant $C_q$ such that
	\begin{equation}\label{3-1}
		|\mathcal{G}_1\zeta(x)|\leq C_q\left(\vertiii{\zeta}_1+\vertiii{\zeta}_q\right)\times
		\begin{cases}
			x_1^2, \quad &\text{if}\ \ 0<x_1\leq 2,\\
			x_1(\ln x_1)^2,&\text{if}\ \  x_1\geq 2.
		\end{cases}
	\end{equation}
\end{lemma}

\begin{lemma}[Lemmas 2.2 and 2.3 in \cite{Choi20}]\label{lem3-3}
	For $ \zeta \in L^1\cap L^2(\mathbb{R}_+^2, \nu)$ with $P(|\zeta|)<+\infty$, it holds
	\begin{equation}\label{3-2}
		|\mathcal{G}_1\zeta(x)|\leq C \left(\vertiii{\zeta}_1+\vertiii{\zeta}_2+P(|\zeta|)\right)\min\{x_1, x_1^{-\frac{1}{2}}\}.
	\end{equation}
	For $ \zeta_i \in L^1\cap L^2(\mathbb{R}_+^2, \nu)$ with $P(|\zeta_i|)<+\infty$ and $i=1,2$, it holds
	\begin{equation}\label{3-3}
		|E(\zeta_1)-E(\zeta_2)|\leq \left( \vertiii{\zeta_1+\zeta_2}_1+\vertiii{\zeta_1+\zeta_2}_2+P(|\zeta_1+\zeta_2|)\right)P( |\zeta_1-\zeta_2| )^{\frac{1}{2}}\vertiii{\zeta_1-\zeta_2}_1^{\frac{1}{2}}.
	\end{equation}
	In particular, taking $\zeta_2\equiv0$ in \eqref{3-3}, one finds that  if $\zeta\in  L^1\cap L^q(\mathbb{R}_+^2, \nu)$ for some $q\geq 2$ and $P(|\zeta|)<+\infty$ then $$E(\zeta)\leq \left( \vertiii{\zeta}_1+\vertiii{\zeta }_2+P(|\zeta |)\right)P( |\zeta | )^{\frac{1}{2}}\vertiii{\zeta }_1^{\frac{1}{2}}<+\infty.$$
\end{lemma}

We shall say that $\zeta$ is Steiner symmetric about the $x_1$-axis  if for any fixed $x_1$, $\zeta$  is the unique even function of $x_2$  such that
\begin{equation*}
	\zeta(x_1, x_2)>\tau\ \ \ \text{if and only if}\ \ \ |x_2|<\frac{1}{2}\,|\left\{y_2\in \mathbb{R} \mid\ \zeta(x_1,y_2)>\tau \right\}|_{\mathbb R},
\end{equation*}
where $|\cdot|_{\mathbb R}$ denotes the Lebesgue measure on $\mathbb R$.

For a function $0\leq \zeta\in L^1\cap L^q(\mathbb{R}_+^2, \nu)$ with $q>2$, we denote by $\zeta^\star$ the Steiner symmetrization of $\zeta$, which is the unique  function in $\mathcal{R}(\zeta)$ that is Steiner symmetric about the $x_1$-axis. A key fact  is the rearrangement inequality (see e.g. Theorems 3.7 and 3.9 in  \cite{Lieb})
\begin{equation}\label{rieq}
	E(\zeta^\star)\geq E(\zeta),
\end{equation}
with strict inequality unless $\zeta(\cdot)\equiv\zeta^\star(\cdot+c\mathbf{e}_2)$ for some $c\in\mathbb{R}$.

\begin{lemma}\label{lem3-4}
	Let $q>2$ be a positive number. If	$0\leq \zeta\in L^1\cap L^q(\mathbb{R}_+^2, \nu)$ is Steiner symmetric about the $x_1$-axis and $P(|\zeta|)<+\infty$, then
	\begin{equation}\label{3-4}
		|\mathcal{G}_1\zeta(x)|\leq C_q \left(|x_2|^{-\frac{1}{2}}\vertiii{\zeta}_1+|x_2|^{-\frac{1}{2q}}\vertiii{\zeta}_q+|x_2|^{-\frac{3}{2}} P(|\zeta|)\right)x_1^2, \quad \text{if}\ \ |x_2|>1,
	\end{equation}
	for some constant $C_q$ depending only on $q$.
	
\end{lemma}		
\begin{proof}
	For function $0\leq w\in L^1{(\mathbb{R}_+^2)}$ that is Steiner symmetric about the $x_1$-axis, the following inequality holds trivially
	\begin{equation}\label{3-5}
		\int_{|y_2-x_2|<l} w(y) d\nu\leq \frac{l}{|x_2|}\int_{\mathbb{R}_+^2} w(y)d\nu,
	\end{equation}
	for $0<l\leq |x_2|$.

	For any fixed $x\in \mathbb{R}_+^2$, let
	\begin{equation*}
		\zeta_1(y):=\left\{
		\begin{array}{lll}
			\zeta(y), \ \  & \text{if} \ \ |y_2-x_2|<\sqrt{|x_2|},\\
			0, & \text{if} \ \ |y_2-x_2|\geq\sqrt{|x_2|}.
		\end{array}
		\right.
	\end{equation*}
	Using  \eqref{3-5} with $w(y)= \zeta^s$ and $l=|x_2|^{\frac{1}{2}}$ for $|x_2|\geq 1$, it is easy to see that
	\begin{equation*}
		\vertiii{\zeta_1}_s\leq \left(\frac{|x_2|^{\frac{1}{2}}}{|x_2|}\right)^{\frac{1}{s}}\vertiii{\zeta}_s=|x_2|^{-\frac{1}{2s}}\vertiii{\zeta}_s, \quad \forall\ s\in[1,q].
	\end{equation*}
	Hence, by \eqref{3-1}, we have
	\begin{equation}\label{3-6}
		\mathcal{G}_1\zeta_1(x)\leq C_q\left(\vertiii{\zeta_1}_1+\vertiii{\zeta}_q\right)x_1^2\leq C_q \left(|x_2|^{-\frac{1}{2}}\vertiii{\zeta}_1+ |x_2|^{-\frac{1}{2q}}\vertiii{\zeta}_q\right)x_1^2.
	\end{equation}
	
	Letting $\zeta_2=\zeta-\zeta_1$, by the inequality \eqref{2-9} with $\delta=\frac{3}{2}$, we find
	\begin{equation}\label{3-7}
		\begin{split}
			\mathcal{G}_1\zeta_2(x)&=\int_{|y_2-x_2|\geq \sqrt{|x_2|}} G_1(x,y) \zeta(y) d \nu\\
			&\leq C \int_{|y_2-x_2|\geq \sqrt{|x_2|}} \frac{x_1^2 y_1^2}{|x_2|^{ \frac{3}{2}}} \zeta(y) d \nu\leq C|x_2|^{-\frac{3}{2}}P(\zeta) x_1^2.
		\end{split}
	\end{equation}
	Then   \eqref{3-4} follows by combining \eqref{3-6} and \eqref{3-7} and hence the proof is complete.
\end{proof}

\begin{lemma}\label{lem3-5}
	Suppose that $\zeta_0\in L^1\cap L^q(\mathbb{R}_+^2, \nu)$ with some $2<q<\infty$   and $\mathcal W>0$ is a given constant. Then the following properties hold.
	\begin{itemize}
		\item[($\mathrm i$).] There exists  $R_0>0$ depending on $\vertiii{\zeta_0}_1+\vertiii{\zeta_0}_q $ and $\mathcal W $ such that for all $\zeta\in\overline{\mathcal{R}(\zeta_0)^w}$, $$\mathcal{G}_1\zeta(x)-\frac{\mathcal W}{2} x_1^2<0,\quad \forall x=(x_1,x_2)\ \text {with}\  x_1>R_0.$$
		\item[($\mathrm{ii}$)] Let $h=\zeta 1_V$ for some function $\zeta$ and some set $V\subset \{\mathcal{G}_1\zeta-\frac{\mathcal W}{2} x_1^2\leq 0\}$, then $\mathcal E_{\mathcal W}(\zeta-h)\geq \mathcal E_{\mathcal W}(\zeta)$ with strict inequality unless $h\equiv0$.
	\end{itemize}
\end{lemma}	
\begin{proof}
	The	assertion ($\mathrm i$) follows directly from \eqref{3-1}. As for ($\mathrm{ii}$), one sees that
	\begin{equation*}
		\begin{split}
			\mathcal E_{\mathcal W}(\zeta-h)&=\frac{1}{2}\int_{\mathbb R^2_+}(\zeta-h)(x)\mathcal{G}_1(\zeta-h)(x) d\nu-\frac{\mathcal W}{2}\int_{\mathbb R^2_+}(\zeta-h)(x) x_1^2 d\nu\\
			&=\mathcal E_{\mathcal W}(\zeta)+E(h)-\int_{\mathbb R^2_+}h(x)\left(\mathcal{G}_1\zeta(x) -\frac{\mathcal W}{2} x_1^2\right)d\nu\\
			&\geq \mathcal E_{\mathcal W}(\zeta)+E(h),
		\end{split}
	\end{equation*}
	which implies ($\mathrm{ii}$) since $E(h)\geq 0$ and $E(h)=0$ if and only if $h\equiv0$. The proof is therefore complete.
\end{proof}

\begin{lemma}\label{lem3-6}
	Let  $0\leq \zeta_0\in L^1\cap L^q(\mathbb{R}_+^2, \nu)$ with some $2<q<\infty$   and $\mathcal W>0$ is a given constant. Then $$\sup_{\zeta\in\overline{\mathcal{R}(\zeta_0)^w}}  \mathcal E_{\mathcal W}(\zeta)<+\infty,$$ and any maximizer (if exists) is supported in $[0, R_0]\times \mathbb R$, where $R_0$ is the  constant given by Lemma \ref{lem3-5}.
\end{lemma}
\begin{proof}
	Let $\{\zeta_{k}\}_{k=1}^\infty\subset \overline{\mathcal{R}(\zeta_0)^w}$ be  a maximizing sequence. In view of  Lemma \ref{lem3-5}, we may assume that $\zeta_{k}$ is supported in $[0, R_0]\times \mathbb R$ by replacing $\zeta_{k}$ with $\zeta_{k} 1_{(0, R_0) \times \mathbb R}$ if otherwise. Then, we obtain $P(\zeta_k)\leq R_0^2 \vertiii{\zeta_k}_1\leq R_0^2 \vertiii{\zeta_0}_1 $ by Lemma \ref{lem3-1} $(d)$.  Thus we obtain $  \mathcal E_{\mathcal W}(\zeta_k)  $ is uniformly bounded by a constant depending only on $\zeta_0$ by Lemma \ref{lem3-1} $(d)$ and Lemma \ref{lem3-3}. Suppose that $\zeta^0 \in \overline{\mathcal{R}(\zeta_0)^w}$ is a maximizer, let $h:=\zeta^0 1_{(R_0, \infty)\times \mathbb R}$, then we infer from Lemma \ref{lem3-5} ($\mathrm{ii}$) that $h\equiv0$ since $(R_0, \infty)\times \mathbb R\subset\{\mathcal{G}_1\zeta^0-\frac{\mathcal W}{2} x_1^2\leq 0\}$ by Lemma \ref{lem3-5} ($\mathrm i$). The proof of this lemma is thus finished.
\end{proof}	
\begin{lemma}\label{lem3-7}
	For  $2<q<\infty$, let  $0\leq \zeta_0\in  L^1 \cap L^q(\mathbb{R}_+^2, \nu)$ be a function  with $\nu(\mathrm{supp}(\zeta_0))<\infty$ and $\mathcal W>0$ is a given constant. Then we have $$\sup_{\zeta\in \mathcal{R}(\zeta_0) }  \mathcal E_{\mathcal W}(\zeta)=\sup_{\zeta\in \mathcal{R}_+(\zeta_0) }  \mathcal E_{\mathcal W}(\zeta)=\sup_{\zeta\in\overline{\mathcal{R}(\zeta_0)^w}}  \mathcal E_{\mathcal W}(\zeta)<+\infty.$$
\end{lemma}
\begin{proof}
	It is sufficient to show $\sup_{\zeta\in \mathcal{R}(\zeta_0) }  \mathcal E_{\mathcal W}(\zeta) =\sup_{\zeta\in\overline{\mathcal{R}(\zeta_0)^w}}  \mathcal E_{\mathcal W}(\zeta) .$ Take arbitrary $\zeta\in \overline{\mathcal{R}(\zeta_0)^w}$, 
	set $\zeta^*:=\zeta 1_{(0,R_0)\times \mathbb R}$ with $R_0$ the constant given by Lemma \ref{lem3-5} $(\mathrm{i})$ for $\zeta_0$. Then we find $$\mathcal E_{\mathcal W}(\zeta^*)\geq \mathcal E_{\mathcal W}(\zeta)$$ by Lemma \ref{lem3-5} ($\mathrm{ii}$). By the monotone convergence theorem, for arbitrary $\ep>0$, we can choose a constant $Z_\ep>0$ such that the function $\zeta^{**}:=\zeta 1_{(0,R_0)\times (-Z_\ep, Z_\ep)}\in \overline{\mathcal{R}(\zeta_0)^w}$ satisfies
	$$\mathcal E_{\mathcal W}(\zeta^{**})\geq \mathcal E_{\mathcal W}(\zeta^*)-\ep.$$
	By the weak density Lemma \ref{lem3-1} $(b)$ and the fact $G_1(x,y)\in L^s((0,R_0)\times (-Z_\ep, Z_\ep))$ for any $1\leq s<\infty $ due to \eqref{2-9}, we can find a function $\zeta^{***}\in \mathcal{R}(\zeta_0)$ such that
	$$\mathcal E_{\mathcal W}(\zeta^{***}1_{(0,R_0)\times (-Z_\ep, Z_\ep)})\geq \mathcal E_{\mathcal W}(\zeta^{**})-\ep.$$
	Let $h:=\zeta^{***} 1_{(R_0, \infty)\times ((-\infty, Z_\ep)\cup (Z_\ep, \infty))}$. Since $\zeta_0$ is assumed to have a support with finite measure, we can rearrange $h$ on the thin strip $(0,\delta)\times (-\gamma_\delta, \gamma_\delta)$ with $\gamma_\delta=2\delta^{-1}\nu(\text{supp}(\zeta_0))<\infty$ into a function $h_\delta$. Then by the estimate \eqref{3-1}, one can easily verify that $\zeta_\delta:=\zeta^{***}1_{(0,R_0)\times (-Z_\ep, Z_\ep)}+h_\delta\in \mathcal{R}(\zeta_0)$ and $$\mathcal E_{\mathcal W}(\zeta_\delta)\to \mathcal E_{\mathcal W}(\zeta^{***}1_{(0,R_0)\times (-Z_\ep, Z_\ep)}), \quad \text{as}\ \  \delta\to 0.$$
	Thus, we obtain a function $\zeta_\delta\in \mathcal{R}(\zeta_0)$ such that
	$$\mathcal E_{\mathcal W}(\zeta_\delta)\geq \mathcal E_{\mathcal W}(\zeta)-3\ep,$$
	by taking $\delta$ sufficiently small, which implies $\sup_{\zeta\in \mathcal{R}(\zeta_0) }  \mathcal E_{\mathcal W}(\zeta) =\sup_{\zeta\in\overline{\mathcal{R}(\zeta_0)^w}}  \mathcal E_{\mathcal W}(\zeta)  $ by the arbitrariness of $\ep$ and completes the proof.
\end{proof}

To obtain the compactness, we need the following concentration compactness lemma  due to Lions \cite{Lions}.
\begin{lemma}\label{ccp}
	Let $\{u_n\}_{n=1}^\infty$ be a sequence of nonnegative functions in $L^1(\mathbb{R}_+^2)$ satisfying
	$$\limsup_{n\rightarrow \infty} \int_{\mathbb{R}_+^2} u_n dx\rightarrow \mu,$$ for some $0\leq \mu<\infty$.
	Then, after passing to a subsequence, one of the following holds:\\
	(i) (Compactness) There exists a sequence $\{y_n\}_{n=1}^\infty$ in $\overline{\mathbb{R}_+^2}$ such that for arbitrary $\varepsilon>0$, there exists $R>0$ satisfying
	\begin{equation*}
		\int_{\mathbb{R}_+^2\cap B_R(y_n)}u_n dx\geq \mu-\varepsilon, \quad \forall n\geq 1.
	\end{equation*}\\
	(ii) (Vanishing) For each $R>0$,
	\begin{equation*}
		\lim_{n\rightarrow \infty}\sup_{y\in \mathbb{R}_+^2}  \int_{\mathbb{R}_+^2\cap B_R(y)} u_n dx =0.
	\end{equation*} \\
	(iii) (Dichotomy) There exists a constant $0<\alpha<\mu$ such that for any $\varepsilon>0$, there exist $N=N(\varepsilon)\geq 1$ and $0\leq u_{i,n}\leq u_n, \,i=1,2$ satisfying
	\begin{equation*}
		\begin{cases}
			\|u_n-u_{1,n}-u_{2,n}\|_{L^1(\mathbb{R}_+^2)}+|\alpha-\int_{\mathbb{R}_+^2} u_{1,n} dx|+|\mu-\alpha-\int_{\mathbb{R}_+^2} u_{2,n} dx|<\varepsilon,\quad \text{for}\,\,n\geq N,\\
			d_n:=dist(supp(u_{1,n}), supp(u_{2,n}))\rightarrow \infty, \quad \text{as}\,\,n\rightarrow \infty.
		\end{cases}	
	\end{equation*}
	Moreover, if $\mu=0$ then  only vanishing will occur.
\end{lemma}
\begin{proof}
	This lemma is a slight reformulation of Lemma 1.1 in \cite{Lions}, so we omit the proof.
\end{proof}

For a function $0\leq \zeta_0\in L^1\cap L^q(\mathbb{R}_+^2, \nu)$ with some $q>2$ and a constant $\mathcal{W}>0$, we denote $S_{\zeta_0,\mathcal{W}}:=\sup_{\zeta\in\overline{\mathcal{R}(\zeta_0)^w}}  \mathcal E_{\mathcal W}(\zeta)$ as the maximum value and $\Sigma_{\zeta_0,\mathcal{W}}:=\{\zeta\in\overline{\mathcal{R}(\zeta_0)^w}\mid \mathcal E_{\mathcal W}(\zeta)=  S_{\zeta_0,\mathcal{W}}\}$ as the set of all the maximizers.

Now, we are in a position to show the following compactness theorem of a  maximizing sequence up to translations for the $x_2$-variable by using a concentration compactness principle, which is a essential tool in the  proof  of   stability.
\begin{theorem}\label{cp}
	For  $2<q<\infty$, let  $0\leq \zeta_0\in   L^q(\mathbb{R}_+^2, \nu)$ be a function  with $0<\nu(\mathrm{supp}(\zeta_0))<\infty$ and $\mathcal W>0$ be a given constant. Assume that the set of maximizers satisfies $$\emptyset\not= \Sigma_{\zeta_0,\mathcal W}\subset \mathcal{R}(\zeta_0).$$ Suppose that $\{\zeta_n\}_{n=1}^\infty\subset \mathcal{R}_+(\zeta_0)$ is a maximizing sequence in the sense that
	\begin{equation}\label{3-8}
		\mathcal{E}_{\mathcal{W}} (\zeta_n)\rightarrow S_{\zeta_0,\mathcal{W}},\quad \text{as}\ \ n\rightarrow \infty.
	\end{equation}
	Then, there exist $\zeta^0\in \Sigma_{\zeta_0,\mathcal{W}}$, a subsequence $\{\zeta_{n_k}\}_{k=1}^\infty$ and a sequence of real numbers $\{c_k\}_{k=1}^\infty$ such that as $k\rightarrow \infty$,
	\begin{equation}\label{4-4}
		\vertiii{\zeta_{n_k}(\cdot+c_k\mathbf{e}_2)- \zeta^0}_2\to 0.
	\end{equation}
\end{theorem}
\begin{proof}
	We first remark that by Lemma \ref{lem3-7}, the condition $$\emptyset\not= \Sigma_{\zeta_0,\mathcal W}\subset \mathcal{R}(\zeta_0) $$ implies $\sup_{\zeta\in \mathcal{R}(\zeta_0) }  \mathcal E_{\mathcal W}(\zeta)=\sup_{\zeta\in \mathcal{R}_+(\zeta_0) }  \mathcal E_{\mathcal W}(\zeta)=\sup_{\zeta\in\overline{\mathcal{R}(\zeta_0)^w}}  \mathcal E_{\mathcal W}(\zeta)=S_{\zeta_0,\mathcal{W}}>0$ since $0\in \overline{\mathcal{R}(\zeta_0)^w}\setminus  \mathcal{R}(\zeta_0)$.
	
	Let $u_n=x_1\zeta_n^2$.  Then, by Lemma \ref{lem3-1} $(d)$, we find $0\leq \int_{\mathbb{R}_+^2} u_n dx\leq \vertiii{\zeta_0}_2^2<\infty$. Therefore, up to a subsequence (still denoted by $\{u_n\}_{n=1}^\infty$), we may assume that $$\int_{\mathbb{R}_+^2} u_n dx\to \mu$$ for some $0\leq \mu\leq  \vertiii{\zeta_0}_2^2$.  Applying Lemma \ref{ccp},  we find that for a certain subsequence, still denoted by $\{u_n\}_{n=1}^\infty$, one of the three cases in Lemma \ref{ccp} should occur. To deal with the three cases, we divide the proof into three steps.
	
	\emph{Step 1. Vanishing excluded:} By Lemma \ref{lem3-5}, we may assume that $\zeta_n$ is supported in $[0, R_0]\times \mathbb{R}$.
	Suppose that for each fixed $R>0$,
	\begin{equation}\label{3-10}
		\lim_{n\rightarrow \infty}\sup_{y\in \mathbb{R}^2_+}  \int_{B_R(y)\cap\mathbb{R}^2_+}  \zeta^2_n d\nu =0.
	\end{equation}
	We will show $\lim_{n\rightarrow \infty} E(\zeta_n)=0$, which contradicts the fact $S_{\zeta_0, \mathcal W}>0$.  By the property of rearrangement and H\"older's inequality, we have for any $R>0$ and $1\leq s \leq 2$
	$$\int_{B_R(y)\cap\mathbb{R}^2_+}  \zeta^s_n d\nu \to 0$$ as $n\to+\infty$ uniformly over $y\in \mathbb{R}_+^2$. Since $supp(\zeta_n)\subset [0, R_0]\times \mathbb{R}$, it holds $$\int_{B_R(y)\cap\mathbb{R}^2_+}  x_1^2\zeta_n d\nu \leq R_0^2 \int_{B_R(y)\cap\mathbb{R}^2_+}   \zeta_n d\nu\to 0$$ as $n\to+\infty$ uniformly over $y\in \mathbb{R}_+^2$.
	Thus, $\mathcal{G}_1(\zeta_n 1_{B_R(y)})(y)=o_n(1)$ uniformly over $  y\in (0, R_0)\times \mathbb R$ by \eqref{3-2}.
	By the inequality \eqref{2-9} with $\delta=1$, we find $\mathcal{G}_1(\zeta_n (1-1_{B_R(y)}))(y)\leq C\frac{R_0^3}{R^2}\vertiii{\zeta_n}_1$.
	Therefore, $$E(\zeta_n)\leq \vertiii{\zeta_n}_1 \left( C\frac{R_0^3}{R^2}\vertiii{\zeta_n}_1+o_n(1)\right),$$ for any $R>0$ and hence $\lim_{n\to\infty} E(\zeta_n)=0$. This is a contradiction to $S_{\zeta_0,\mathcal{W}}>0$ and thus vanishing can not occur.

	\emph{Step 2. Dichotomy excluded:}
	We may also assume that $\zeta_n$ is supported in $[0, R_0]\times \mathbb{R}$. Suppose that there is a constant $\alpha\in (0, \mu)$ such that for any $\varepsilon>0$, there exist $N(\varepsilon)\geq 1$ and $0\leq \zeta_{i,n}\leq \zeta_n, \,i=1,2,3$ satisfying
	\begin{equation*}
		\begin{cases}
			\zeta_n=\zeta_{1,n}+\zeta_{2,n}+\zeta_{3,n},\\
			\int_{\mathbb{R}^2_+}  \zeta^2_{3,n}d\nu +|\alpha-\alpha_n|+|\mu-\alpha-\beta_n|<\varepsilon,\quad \text{for}\,\,n\geq N(\varepsilon),\\
			d_n:=dist(supp(\zeta_{1,n}), supp(\zeta_{2,n}))\rightarrow \infty, \quad \text{as}\,\,n\rightarrow \infty,
		\end{cases}	
	\end{equation*}
	where $\alpha_n=\int_{\mathbb{R}^2_+}  \zeta^2_{1,n}d\nu$ and $ \beta_n=\int_{\mathbb{R}^2_+}  \zeta^2_{2,n}d\nu$. Using a diagonal argument, we obtain that there exists a subsequence, still denoted by $\{\zeta_n\}_{n=1}^\infty$, such that
	\begin{equation*}
		\begin{cases}
			\zeta_n=\zeta_{1,n}+\zeta_{2,n}+\zeta_{3,n}, \quad 0\leq \zeta_{i,n}\leq \zeta_n, \,i=1,2,3\\
			\int_{\mathbb{R}^2_+}  \zeta^2_{3,n}d\nu +|\alpha-\alpha_n|+|\mu-\alpha-\beta_n|\rightarrow 0,\quad \text{as}\,\,n\rightarrow \infty,\\
			d_n:=dist(supp(\zeta_{1,n}), supp(\zeta_{2,n}))\rightarrow \infty, \quad \text{as}\,\,n\rightarrow \infty.
		\end{cases}	
	\end{equation*}
	
	By the symmetry of $E$, we have
		\begin{align*}
			&2E(\zeta_n)=E(\zeta_{1,n}+\zeta_{2,n}+\zeta_{3,n})\\
			&=\int_{\mathbb{R}^2_+}\int_{\mathbb{R}^2_+} \zeta_{1,n}(x)G_1(x,y)\zeta_{1,n}(y)x_1 dxdy+\int_{\mathbb{R}^2_+}\int_{\mathbb{R}^2_+} \zeta_{2,n}(x)G_1(x,y)\zeta_{2,n}(y)  d\nu(x)d\nu(y)\\
			& +2\int_{\mathbb{R}^2_+}\int_{\mathbb{R}^2_+} \zeta_{1,n}(x)G_1(x,y)\zeta_{2,n}(y)d\nu(x)d\nu(y)+\int_{\mathbb{R}^2_+}\int_{\mathbb{R}^2_+} (2\zeta_n-\zeta_{3,n}(x))G_1(x,y)\zeta_{3,n}(y) d\nu(x)d\nu(y).
	\end{align*}
	By \eqref{3-1} and H\"older's inequality, we derive
	\begin{equation*}
		\begin{split}
			&\int_{\mathbb{R}^2_+}\int_{\mathbb{R}^2_+} (2\zeta_n-\zeta_{3,n}(x))G_1(x,y)\zeta_{3,n}(y) d\nu(x)d\nu(y)\\
			\leq &C  \left(\vertiii{\zeta_0}_1+\vertiii{\zeta_0}_q\right) \int_{\mathbb{R}^2_+}y_1^2\zeta_{3,n}(y)  d\nu(y)\leq C R_0^2 \nu(\text{supp}(\zeta_0)) \left(\int_{\mathbb{R}^2_+}  \zeta^2_{3,n}d\nu\right)^{\frac{1}{2}}\to 0.
		\end{split}
	\end{equation*}
	By using \eqref{2-9} with $\delta=\frac{3}{2}$, it is obvious that
	\begin{align*}
		\int_{\mathbb{R}^2_+}\int_{\mathbb{R}^2_+} \zeta_{1,n}(x)G_1(x,y)\zeta_{2,n}(y)d\nu(x)d\nu(y)\leq \frac{C R_0^4 \vertiii{\zeta_0}_1^2}{d_n^{3}}.
	\end{align*}
	Hence, we arrive at
	$$\mathcal{E}_{\mathcal W}(\zeta_n)=E(\zeta_n)-\frac{\mathcal W}{2}\int_{\mathbb{R}^2_+} x_1^2\zeta_n d\nu(x) \leq \mathcal{E}_{\mathcal W}(\zeta_{1,n})+\mathcal{E}_{\mathcal W}(\zeta_{2,n})+o_n(1).$$
	
	Taking Steiner symmetrization about the $x_1$-axis $\zeta^*_{i,n}$ of $\zeta_{i,n}$ for $i=1,2$, we have
	\begin{equation*}
		\mathcal{E}_{\mathcal W}(\zeta_n)\leq \mathcal{E}_{\mathcal W}(\zeta^*_{1,n})+\mathcal{E}_{\mathcal W}(\zeta^*_{2,n})+o_n(1).
	\end{equation*}
	By Lemma \ref{lem3-4}, there exists $Z_0>0$ depending on $\vertiii{\zeta_0}_1+ \vertiii{\zeta_0}_q$, $\mathcal W $ and $R_0$ such that for all $\zeta\in\mathcal{R}_+(\zeta_0)$ with $supp(\zeta)\subset [0,R_0]\times \mathbb{R}$, $$\mathcal{G}_1\zeta(x)-\frac{\mathcal W}{2} x_1^2<0,\quad \forall x\ \text {with}\  |x_2|>Z_0$$
	Let $$\zeta^{**}_{i,n}(x)=\zeta^*_{i,n}1_{[0,R_0]\times[-Z_0,Z_0]} (x+(-1)^i Z_0 \mathbf{e}_2),\quad i=1,2.$$
	Then, we infer from Lemma \ref{lem3-5} that
	\begin{equation*}
		\mathcal{E}(\zeta_n)\leq \mathcal{E}(\zeta^{**}_{1,n})+\mathcal{E}(\zeta^{**}_{2,n})+o_n(1).
	\end{equation*}
	and
	$$\text{supp}(\zeta^{**}_{1,n})\subset [0,R_0]\times[0, 2Z_0],\  \  \ \text{supp}(\zeta^{**}_{2,n})\subset [0,R_0]\times[- 2Z_0,0].$$
	We may assume that $\zeta^{**}_{i,n}\rightarrow \zeta^{**}_{i}$   weakly in $L^2(\mathbb{R}_+^2, \nu)$ for $i=1,2$. Then, $\zeta^{**}:=\zeta^{**}_{1}+\zeta^{**}_{2}\in \overline{\mathcal{R}(\zeta_0)^w}$. Moreover, we have the convergence of the kinetic energy and impulse by the weak continuity
	$$\lim_{n\rightarrow \infty} E(\zeta^{**}_{i,n})=E(\zeta^{**}_{i}),\quad \lim_{n\rightarrow \infty} P(\zeta^{**}_{i,n})=P(\zeta^{**}_{i}),\,\,\text{for}\, i=1,2,$$
	and therefore we arrive at
	$$\mathcal{E}_{\mathcal W}(\zeta^{**}_{1})+\mathcal{E}_{\mathcal W}(\zeta^{**}_{2})\geq \limsup_{n\to\infty}\mathcal{E}_{\mathcal W}(\zeta_n)=S_{\zeta_0, \mathcal W}.$$
	It can be seen that
	\begin{equation*}
		\begin{split}
			&\quad S_{\zeta_0, \mathcal W}\geq \mathcal{E}_{\mathcal W}(\zeta^{**})=\mathcal{E}_{\mathcal W}(\zeta^{**}_{1}+\zeta^{**}_{2})\\
			&=\mathcal{E}_{\mathcal W}(\zeta_1^{**})+\mathcal{E}_{\mathcal W}(\zeta_2^{**})+\int_{\mathbb{R}^2_+}\int_{\mathbb{R}^2_+} \zeta_1^{**}(x)G_2(x,y)\zeta_1^{**}(y)d\nu(x)d\nu(y)\\
			&\geq S_{\zeta_0, \mathcal W}+\int_{\mathbb{R}^2_+}\int_{\mathbb{R}^2_+} \zeta_1^{**}(x)G_1(x,y)\zeta_2^{**}(y)d\nu(x)d\nu(y),
		\end{split}
	\end{equation*}
	from which we must have
	\begin{equation}\label{3-11}
		\mathcal{E}_{\mathcal W}(\zeta^{**})=S_{\zeta_0, \mathcal W}\ \ \  \text{and}\ \ \  \int_{\mathbb{R}^2_+}\int_{\mathbb{R}^2_+} \zeta_1^{**}(x)G_1(x,y)\zeta_2^{**}(y)d\nu(x)d\nu(y)=0.
	\end{equation}
	Since $\Sigma_{\zeta_0,\mathcal W}\subset \mathcal{R}(\zeta_0)$ by the assumption,  we conclude that  $\zeta^{**}\in  \mathcal{R}(\zeta_0)$ and $\mu=\vertiii{\zeta^{**}}^2_2=\vertiii{\zeta_1^{**}}^2_2+\vertiii{\zeta_2^{**}}^2_2$.
	
	On the other hand, since  $\vertiii{\zeta_1^{**}}^2_2\leq \alpha$ and $\vertiii{\zeta_2^{**}}^2_2\leq \mu-\alpha$, we deduce $\vertiii{\zeta_1^{**}}^2_2= \alpha>0$ and $\vertiii{\zeta_2^{**}}^2_2=\mu-\alpha>0$. This implies that both $\zeta_1^{**}$ and $\zeta_2^{**}$ are non-zero and hence $\int_{\mathbb{R}^2_+}\int_{\mathbb{R}^2_+} \zeta_1^{**}(x)G_1(x,y)\zeta_2^{**}(y)d\nu(x)d\nu(y)>0$, which is a contradiction to \eqref{3-11}.

	\emph{Step 3. Compactness:} Assume that there is a sequence $\{y_n\}_{n=1}^\infty$ in $\overline{\mathbb{R}_+^2}$ such that for arbitrary $\varepsilon>0$, there exists $R>0$ satisfying
	\begin{equation}\label{3-12}
		\int_{\mathbb{R}_+^2\cap B_R(y_n)}  \zeta_n^2 d\nu\geq \mu-\varepsilon, \quad \forall\, n\geq 1.
	\end{equation}
	We may assume that $y_n=(  y_{n,1}, 0)$ after a suitable translation in $x_2$-variable.
	Define $\zeta_n^0:=\zeta_n 1_{(0, R_0)\times \mathbb R}$ and $\zeta_n^R:=\zeta_n 1_{(0, R_0)\times (-R, R)}$. Then $\{\zeta_n^0\}_{n=1}^\infty$ is also a maximizing sequence in $\mathcal{R}_+(\zeta_0)$ by Lemma \ref{lem3-5}. Moreover, we infer from \eqref{3-12} that for arbitrary $\ep>0$, there exists $R>0$ such that $$\vertiii{\zeta_n^0-\zeta_n^R}_2^2\leq \ep,\quad \forall\  n\geq 1.$$
	That is,
	\begin{equation}\label{3-13}
		\vertiii{\zeta_n^0-\zeta_n^R}_2\to 0,\quad \text{as}\ R\to \infty, \quad \text{uniformly over} \  n.
	\end{equation}
	
	We may assume that $\zeta_n^0\to \zeta^0$ weakly in $L^2(\mathbb{R}_+^2, \nu)$ and hence $\zeta_n^R\to \zeta^01_{(0, R_0)\times (-R, R)}$ weakly in $L^2(\mathbb{R}_+^2, \nu)$.
	By the weak convergence and $\zeta_n\in \mathcal{R}_+(\zeta_0)$, we find
	\begin{equation}\label{3-14}
		\vertiii{\zeta^0}_2\leq \liminf_{n\to\infty}\vertiii{\zeta_n^0}_2\leq \vertiii{\zeta_0}_2.
	\end{equation}
	
	Set $\zeta^R:=\zeta^0 1_{(0, R_0)\times (-R, R)}$. Using \eqref{3-3} and H\"older's inequality, we conclude
	$|E(\zeta_n^0)-E(\zeta_n^R)|\leq C\vertiii{\zeta_n^0-\zeta_n^R}_2^{\frac{1}{2}}$. So, $E(\zeta_n^R)\to E(\zeta_n^0)$ as $R\to \infty$ uniformly over $n$ by \eqref{3-12}. On the other hand $E(\zeta_n^R)\to E(\zeta^R)$ as $n\to \infty$ for fixed $R$ by weak continuity of $E$ in functions supported on  bounded domains and $E(\zeta^R)\to E(\zeta^0)$  as $R\to+\infty$ by the monotone convergence theorem. Therefore, we obtain $$E(\zeta_n^0)\to E(\zeta^0).$$
	
	As for the impulse, we split
	\begin{equation*}
		|P(\zeta_n^0)-P(\zeta^0)|\leq |P(\zeta_n^0)-P(\zeta_n^R)|+|P(\zeta_n^R)-P(\zeta^R)|+|P(\zeta^R)-P(\zeta^0)|.
	\end{equation*}
	For the first term, by H\"older's inequality, we deduce
	$$|P(\zeta_n^0)-P(\zeta_n^R)|\leq R_0^{2}|\nu(\text{supp}(\zeta_n^0))|^{\frac{1}{2}}\vertiii{\zeta_n^0-\zeta_n^R}_2\to 0, $$
	as $R\to \infty$ uniformly over $n$.
	For fixed $R$, we have $|P(\zeta_n^R)-P(\zeta^R)|\to 0$ as $n\to \infty$ by the weak convergence. Since $|P(\zeta^R)-P(\zeta^0)|\to0$ as $R\to \infty$ by  the monotone convergence theorem, we have $|P(\zeta_n^0)-P(\zeta^0)|\to 0$ by first letting $R\to \infty$ and then $n\to\infty$.
	
	Therefore, we have proved $\mathcal{E}_{\mathcal W}(\zeta_n^0)\to \mathcal{E}_{\mathcal W}(\zeta^0)$ and hence $\mathcal{E}_{\mathcal W}(\zeta^0)=S_{\zeta_0,\mathcal W}$ and $\zeta^0\in \mathcal{R}(\zeta_0)$ by our assumption $\Sigma_{\zeta_0,\mathcal{W}}\subset \mathcal{R}(\zeta_0)$. Then we deduce that  $	\vertiii{\zeta^0}_2=\vertiii{\zeta_0}_2$ by the property of rearrangement, which implies $\lim_{n\to\infty}\vertiii{\zeta_n^0}_2=\vertiii{\zeta^0}_2$. So, we obtain the strong convergence $\vertiii{\zeta_n^0-\zeta^0}_2\to 0$ by the uniform convexity of $L^2(\mathbb{R}_+^2, \nu)$.
	
	Now we want to show that $\zeta_n\to \zeta^0$ strongly. Indeed, since the supports of $\zeta_n^0$ and $\zeta_n- \zeta_n^0$ are disjoint and $\zeta^0\in \mathbb{R}(\zeta_0)$, we conclude
	$$\vertiii{\zeta_n-\zeta_n^0}_2^2=\vertiii{\zeta_n}_2^2-\vertiii{ \zeta_n^0}_2^2\leq \vertiii{\zeta_0}_2^2-\vertiii{ \zeta_n^0}_2^2\to \vertiii{\zeta_0}_2^2-\vertiii{ \zeta^0}_2^2=0\ \    \text{as}\, n\to+\infty.$$
	Therefore, we obtain $\vertiii{\zeta_n-\zeta^0}_2\to 0$ and finish the proof.
	
\end{proof}

\subsection{Proof of the stability}
With Theorem \ref{cp} in hand, we are able to obtain the following stability theorem on the set of maximizers via  reduction to absurdity.

\begin{theorem}\label{Sset}
	For  $2<q<\infty$, let  $0\leq \zeta_0\in L^q(\mathbb{R}_+^2, \nu)$ be a function  with $0<\nu(\mathrm{supp}(\zeta_0)) <\infty$ and $\mathcal W>0$ is a given constant.
	Suppose that $\Sigma_{\zeta_0,\mathcal W}$,  the set of maximizers of $\mathcal{E}_{\mathcal W}:=E- \mathcal W  P$ over $\overline{\mathcal{R}(\zeta_0)^w}$, satisfies $$\emptyset\not= \Sigma_{\zeta_0,\mathcal W}\subset \mathcal{R}(\zeta_0).$$   Then $\Sigma_{\zeta_0,\mathcal W}$ is orbitally stable in the following sense:
	
	For arbitrary $\ep>0$, there exists $\delta>0$ such that   if  $\omega_0 $ is nonnegative and satisfies $\omega_0, x_1\omega_0\in L^\infty(\mathbb R^2_+)$,  $P(|\omega_0|)<+\infty$  and
	\begin{equation*}
		\inf_{\zeta\in \Sigma_{\zeta_0,\mathcal W}}\{ \vertiii{\omega_0-\zeta}_1+ \vertiii{\omega_0-\zeta}_2+|P(\omega_0-\zeta)| \} \leq \delta,
	\end{equation*}
	then for all $t\in \mathbb{R}$, we have
	\begin{equation}\label{s-1}
		\inf_{\zeta\in \Sigma_{\zeta_0,\mathcal W}} \left\{  \vertiii{\omega(t)-\zeta}_1+\vertiii{\omega(t)-\zeta}_2\right\}\leq \ep.
	\end{equation}
	If in addition $$P(\zeta_1)=P(\zeta_2),\quad \forall\  \zeta_1, \zeta_2\in \Sigma_{\zeta_0,\mathcal W}, $$
	then for all $t\in \mathbb{R}$, we have
	\begin{equation}\label{s-2}
		\inf_{\zeta\in \Sigma_{\zeta_0,\mathcal W}} \left\{\vertiii{\omega(t)-\zeta}_1+\vertiii{\omega(t)-\zeta}_2+P(|\omega(t)-\zeta|)\right\}\leq \ep.
	\end{equation}
	Here $\omega(t)$   means the corresponding solution of \eqref{1-3} with initial data $\omega_0$ given by Proposition \ref{Pro1}.
	
\end{theorem}	
\begin{proof}
	By Lemma \ref{lem3-6}, any maximizer (if exists) is supported in $[0, R_0]\times \mathbb R$.  So, we deduce that $\sup_{\zeta_1\in \Sigma_{\zeta_0,\mathcal W}} P(\zeta_1)\leq R_0^2 \vertiii{\zeta_0}_1<\infty.$
	By \eqref{3-2}, through an  argument similar to the proof of Lemma \ref{lem3-5}, for $\zeta$ with the following properties:
	\begin{equation*}
		\zeta\geq 0,\ \ \ \vertiii{\zeta}_1 \leq  \vertiii{\zeta_0}_1+1, \ \ \  \vertiii{\zeta}_2\leq  \vertiii{\zeta_0}_2+1,\ \ \  P(\zeta)\leq \sup_{\zeta_1\in \Sigma_{\zeta_0,\mathcal W}} P(\zeta_1) +1,
	\end{equation*}
	there exists a constant $R_1\geq R_0$ independent of $\zeta$ such that
	$$\mathcal{G}_1\zeta(x)-\frac{\mathcal W}{2} x_1^2<0,\quad \forall x\in\mathbb R^2_+\ \text {with}\  x_1>R_1.$$
	
	Now we prove this theorem by  contradiction. Suppose on the contrary that there exist a sequence of non-negative functions $\{\omega_0^n\}_{n=1}^\infty$ satisfying the assumptions of this theorem  and	as $n\to \infty$,
	\begin{equation*}
		\inf_{\zeta\in \Sigma_{\zeta_0,\mathcal W}} \{\vertiii{\omega_0^n-\zeta}_1+\vertiii{\omega_0^n-\zeta}_2 +|P(\omega_0^n-\zeta)|\}\to 0,
	\end{equation*}
	while
	\begin{equation*}
		\sup_{t\geq 0}\inf_{\zeta\in \Sigma_{\zeta_0,\mathcal W}} \left\{  \vertiii{\omega^n(t)-\zeta}_1+\vertiii{\omega^n(t)-\zeta}_2\right\}\geq c_0>0,
	\end{equation*}
	for some positive constant $c_0$, where $\omega^n(t)$ is the corresponding solution of \eqref{1-3} with initial data $\omega_0^n$ provided by Proposition \ref{Pro1}. By the choose of $\omega_0^n$, we infer from \eqref{3-3}  that	as $n\to \infty$,
	\begin{equation}\label{3-14-1}
		\mathcal{E}_{\mathcal W}(\omega_0^n)\to S_{\zeta_0,\mathcal W}.
	\end{equation}
	We  choose $t_n>0$ such that for each $n$,
	\begin{equation}\label{3-15}
		\inf_{\zeta\in \Sigma_{\zeta_0,\mathcal W}} \left\{ \vertiii{\omega^n(t_n)-\zeta}_1+\vertiii{\omega^n(t_n)-\zeta}_2 \right\}\geq \frac{c_0}{2}>0.
	\end{equation}

	Take a sequence of functions $\zeta_0^n \in \Sigma_{\zeta_0,\mathcal W} \subset \mathcal{R}(\zeta_0)$ for $n\geq 1$ such that 	as $n\to \infty$,
	\begin{equation}\label{3-20-0}
		\vertiii{\omega_0^n-\zeta_0^n}_1+\vertiii{\omega_0^n-\zeta_0^n}_2+|P(\omega_0^n-\zeta_0^n)|\to 0.
	\end{equation}
	
	For a function $\omega$, we use $\bar{\omega}:=\omega 1_{(0,R_1)\times \mathbb{R}}$ to denote the restriction of $\omega$. By Lemma \ref{lem3-5}, the conservation of energy and impulse, we conclude
	\begin{equation}\label{3-17}
		\mathcal{E}_{\mathcal W}(\overline{\omega}^n(t_n))\geq 	\mathcal{E}_{\mathcal W}( \omega^n(t_n))=	\mathcal{E}_{\mathcal W}( \omega^n_0)\to S_{\zeta_0,\mathcal W}.
	\end{equation}
	Recall that we use the notation $(x_1,x_2)=(r,z)$. Recall the notation $(x_1,x_2)=(r,z)$. Let $\zeta^n(t)$ be the  solution of the following linear transport equation with initial data $\zeta_0^n$:
	\begin{equation*}
		\begin{cases}
			\partial_t \zeta+\mathbf{u}_n\cdot\nabla\zeta=0,\\		
			\zeta( 0)=\zeta_0^n,
		\end{cases}
	\end{equation*}
	where $\nabla\zeta :=\frac{\partial\zeta}{\partial r}\mathbf{e}_r+\frac{\partial\zeta}{\partial z}\mathbf{e}_z$, $\mathbf{u}_n=\frac{1}{r}\left(-\frac{\partial\psi_n}{\partial z}\mathbf{e}_r+\frac{\partial\psi_n}{\partial r}\mathbf{e}_z\right)$ and $\psi_n=\mathcal{L}^{-1} \omega^n$.  The existence and some basic properties of $\zeta^n(t)$ are proved later in Lemma \ref{trp}. Then, it can be seen that $\zeta^n-\omega^n$ satisfies
	\begin{equation}\label{diff}
		\begin{cases}
			\partial_t (\zeta^n-\omega^n)+\mathbf{u}_n\cdot\nabla(\zeta^n-\omega^n)=0,\\		
			\zeta^n(0)-\omega^n(0)=\zeta_0^n-\omega_0^n.
		\end{cases}
	\end{equation}

	Then, we infer from \eqref{3-3},  the conservation of  $L^1$-norm due to \eqref{diff} and H\"older's inequality that
	\begin{equation}\label{3-18}
		\begin{split}
			|P(\bar\zeta^n(t_n) )-P(\bar\omega^n(t_n))|&\leq R_1^2 \vertiii{\bar\zeta^n(t_n)-\bar\omega^n(t_n)}_1\leq R_1^2 \vertiii{ \zeta^n(t_n)- \omega^n(t_n)}_1\\
			&\leq  R_1^2 \vertiii{ \zeta^n_0- \omega^n_0}_1,
		\end{split}
	\end{equation}
	and
	\begin{equation}\label{3-19}
		\begin{split}
			|E(\bar\zeta^n(t_n) )-E(\bar\omega^n(t_n))|&\leq C P( |\bar\zeta^n(t_n)-\bar\omega^n(t_n)| )^{\frac{1}{2}}\vertiii{\bar\zeta^n(t_n)-\bar\omega^n(t_n)}_1^{\frac{1}{2}}  \\
			&\leq C\vertiii{ \zeta^n_0- \omega^n_0}_1.
		\end{split}
	\end{equation}
	So, we deduce from \eqref{3-20-0}--\eqref{3-19} that $\{\bar\zeta^n(t_n)\}_{n=1}^\infty\subset \mathcal{R}_+(\zeta_0)$ is a maximizing sequence and hence by Theorem \ref{cp}, after extracting subsequences and translating properly in $x_2$,  we have
	\begin{equation}\label{3-20}
		\bar\zeta^n(t_n)\to \zeta_{**} \ \  \text{strongly in}\ \  L^2({\mathbb{R}_+^2, \nu}),
	\end{equation}
	as $n\to \infty$ for some function $ \zeta_{**} \in \Sigma_{\zeta_0,\mathcal W}$, which implies $\bar\zeta^n(t_n)\to \zeta^n(t_n)$ strongly. Indeed, since the supports of $\bar\zeta^n(t_n)$ and $\zeta^n(t_n)- \bar\zeta^n(t_n)$ are disjoint and $\zeta_{**}\in \mathcal{R}(\zeta_0)$, we conclude
	$$\vertiii{\zeta^n(t_n)-\bar\zeta^n(t_n)}_2^2=\vertiii{\zeta^n(t_n)}_2^2-\vertiii{\bar\zeta^n(t_n)}_2^2\leq \vertiii{\zeta_0}_2^2-\vertiii{ \bar\zeta^n(t_n)}_2^2\to \vertiii{\zeta_0}_2^2-\vertiii{\zeta_{**}}_2^2=0,$$
	from which we have
	$$\vertiii{\zeta^n(t_n)-\zeta_{**}}_2\leq \vertiii{\zeta^n(t_n)-\bar\zeta^n(t_n)}_2+\vertiii{\zeta_{**}-\bar\zeta^n(t_n)}_2\to 0.$$
	Hence, by  \eqref{3-20-0} and the  conservation of   the $L^2$ norm due to \eqref{diff} and Lemma \ref{trp}, we deduce
	\begin{equation*}
		\begin{split}
			\quad\vertiii{\omega^n(t_n)-\zeta_{**}}_2 &\leq  \vertiii{\zeta^n(t_n)-\omega^n(t_n)}_2+\vertiii{\zeta^n(t_n)-\zeta_{**}}_2\\
			&=\vertiii{\zeta^n_0-\omega^n_0}_2+\vertiii{\zeta^n(t_n)-\zeta_{**}}_2=o_n(1).\end{split}
	\end{equation*}
	
	Next, we will estimate $\vertiii{\omega^n(t_n)-\zeta_{**}}_1$.
	By the conservation of the $L^1$-norm and \eqref{3-20-0}, we have
	\begin{align*}
		&\ \ \  \vertiii{\omega^n(t_n)-\zeta_{**}}_1 \leq  \int_{\text{supp}(\zeta_{**})}  |\omega^n(t_n)-\zeta_{**}|d\nu+\int_{\mathbb{R}_+^2\setminus\text{supp}(\zeta_{**})}   \omega^n(t_n) d\nu\\
		&\leq \int_{\text{supp}(\zeta_{**})}  |\omega^n(t_n)-\zeta_{**}|d\nu+\int_{\mathbb{R}_+^2 }   \omega^n(t_n) d\nu-\int_{\mathbb{R}_+^2}  \zeta_0^n d\nu+\int_{\mathbb{R}_+^2}   \zeta_{**} d\nu-\int_{\text{supp}(\zeta_{**})  }  \omega^n(t_n) d\nu\\
		&\leq 2\int_{\text{supp}(\zeta_{**})}  |\omega^n(t_n)-\zeta_{**}|d\nu+\int_{\mathbb{R}_+^2 }|\omega_0^n-\zeta_0^n| d\nu \\
		&\leq 2  \nu(\text{supp}(\zeta_0) )^{\frac{1}{2}}  \vertiii{ \omega^n(t_n)-\zeta_{**}}_2+ \int_{\mathbb{R}_+^2 }|\omega_0^n-\zeta_0^n| d\nu=o_n(1),
	\end{align*}
	which contradicts the choice of $t_n$ in \eqref{3-15} and completes the proof of \eqref{s-1}. Here we have used $\int_{\mathbb{R}_+^2}  \zeta_0^n d\nu=\int_{\mathbb{R}_+^2}   \zeta_{**} d\nu$ and $ \nu(\text{supp}(\zeta_0) )= \nu(\text{supp}(\zeta_{**}) )$ due to the properties of rearrangement.
	
	If in addition $$P(\zeta_1)=P(\zeta_2),\quad \forall\  \zeta_1, \zeta_2\in \Sigma_{\zeta_0,\mathcal W}, $$ then by the conservation of the impulse, we find
		\begin{align*}
			&P( |\omega(t)-\zeta_1|) \leq  \int_{\text{supp}(\zeta_1)} x_1^2|\omega(t)-\zeta_1|d\nu+\int_{\mathbb{R}_+^2\setminus\text{supp}(\zeta_1)}  x_1^2\omega(t) d\nu\\
			&\leq \int_{\text{supp}(\zeta_1)} x_1^2 |\zeta_1  -  \omega(t)| d\nu+\int_{\mathbb{R}_+^2}x_1^2  \omega(t) d\nu-\int_{\mathbb{R}_+^2} x_1^2 \zeta_2 d\nu+\int_{\mathbb{R}_+^2} x_1^2 \zeta_1 d\nu-\int_{ \text{supp}(\zeta_1)}x_1^2  \omega(t) d\nu\\
			&\leq 2\int_{\text{supp}(\zeta_1)} x_1^2 |\zeta_1  -  \omega(t)| d\nu+\int_{\mathbb{R}_+^2}x_1^2  \omega_0 d\nu-\int_{\mathbb{R}_+^2} x_1^2 \zeta_2 d\nu\\
			&\leq 2 R_0^2\nu(\text{supp}(\zeta_0))^{\frac{1}{2}} 	\vertiii{ \omega(t)-\zeta_1}_2+ |P( \omega_0-\zeta_2)|, \quad \forall \ \zeta_1, \zeta_2 \in \Sigma_{\zeta_0,\mathcal W},
		\end{align*}
	which combined with  \eqref{s-1} implies \eqref{s-2} by  taking $\delta$ smaller if necessary. Therefore the proof is complete.

\end{proof}

\begin{lemma}\label{trp}
	Let $\omega\in L^1\cap L^\infty(\mathbb R^3)$ be a axi-symmetric function and let $\mathbf{u}=\frac{1}{r}\left(-\frac{\partial\psi}{\partial z}\mathbf{e}_r+\frac{\partial\psi}{\partial r}\mathbf{e}_z\right)$ with $\psi=\mathcal{L}^{-1} \omega$ be the velocity field introduced by $\omega$. Suppose the axi-symmetric function $\zeta_0\in L^q(\mathbb R^3)$ for some $2<q<+\infty$. Then the following initial value problem for the linear transport equation
	\begin{equation*}
		\begin{cases}
			\partial_t \zeta+\mathbf{u}\cdot\nabla\zeta=0,\ \ \ \   x\in \mathbb R^3,\ \ \  t>0,\\		
			\zeta( 0)=\zeta_0,
		\end{cases}
	\end{equation*}
	has a unique axi-symmetric weak solution $\zeta\in L^\infty_{loc}([0, +\infty), L^q(\mathbb R^3))$. Here $\nabla\zeta :=\frac{\partial\zeta}{\partial r}\mathbf{e}_r+\frac{\partial\zeta}{\partial z}\mathbf{e}_z$. Moreover the renormalisation property in the form $\zeta(t)\in \mathcal{R}(\zeta_0)$ holds for almost all $t>0$. As a corollary, one has the conservation of $L^r$ norm provided that $\zeta_0\in L^r(\mathbb R^3)$ for some $r\geq 1$. That is,
	$$\|\zeta(t)\|_{L^r(\mathbb R^3)}\equiv \|\zeta_0\|_{L^r(\mathbb R^3)}.$$
\end{lemma}
\begin{proof}
	The proof is almost the same as that of  Lemma 11 in \cite{BNL13}. Indeed, using Lemma 2 in \cite{Nobi} and Lemma 2.5 in \cite{Bad} and the assumption $\omega\in L^1\cap L^\infty(\mathbb R^3)$, we find that  $\mathbf{u}$ satisfies the requirement of   DiPerna-Lions theory of transport equations in \cite{DL} (see also Theorem 4 in \cite{Nobi}). The rest proof of this lemma follows the same procedure as the proof of Lemma  11 in \cite{BNL13} and  hence  we omit it.
\end{proof}

  Theorem  \ref{thmSs} is an immediate corollary of  Theorem \ref{Sset}. So, we have also finished the proof Theorem  \ref{thmSs}.

\section{Uniqueness of vortex rings}\label{sec4}
As explained in introduction, in order to deduce the stability of a given vortex ring solution by using the general result Theorem \ref{thmSs}, one need to verify \eqref{1-12}, which can be reduced into the problem of uniqueness up to translations.
In this section, we will show the uniqueness of a family of  vortex rings obtained in \cite{CWWZ, VS} (see Proposition \ref{thmE}) and prove Theorem \ref{thmU}.

Let $\zeta_\ep$ be a steady vortex ring that satisfies all the assumptions in Theorem \ref{thmU} and let $\psi_\ep=\mathcal{G}_1\zeta_\ep$ be the corresponding stream function. Recall that the steady vortex rings considered in Theorem \ref{thmU} are assumed to be of the form
$$\zeta_\ep=\frac{1}{\ep^2}\left(\psi_\ep-\frac{W}{2}x_1^2\ln\frac{1}{\ep}-\mu_\ep\right)_+^p.$$
By the discussions in Section \ref{sec1}, the proof of Theorem \ref{thmU} can be converted into the uniqueness of function $\psi_\ep$ satisfying the following equations
\begin{equation}\label{2-2}
	\begin{cases}
		\mathcal{L}\psi_\ep=\frac{1}{\ep^2}\left(\psi_\ep-\frac{W}{2}x_1^2\ln\frac{1}{\ep}-\mu_\ep\right)_+^p, & \text{in} \ \mathbb R^2_+,
		\\
		\psi_\varepsilon=0, & \text{on} \ \{x_1=0\},
		\\
		\psi_\varepsilon  \to0, &\text{as} \ |  x |\to \infty,
	\end{cases}
\end{equation}
combined with  the following conditions
\begin{align}
	\label{2-3} &\frac{1}{\varepsilon^2}\int_{\mathbb R^2_+}x_1 \left(\psi_\ep-\frac{W}{2}x_1^2\ln\frac{1}{\ep}-\mu_\ep\right)_+^p dx=\kappa,  &  \\
	\label{2-4}  &\text{diam}(  \Om_\ep )\leq C_0 \ep \ \ \text{for some constant}\  \  C_0>0, &\\
	\label{2-5} &\sup_{x\in \Om_\ep } |x-x_0| \to 0\ \ \text{for some}\  \ x_0\in \mathbb R^2_+, \ \ \text{as}\ \ \ep\to0.
\end{align}
Here $\Om_\ep$ is the support of the vorticity $\Om_\ep=\{x\in \mathbb R^2_+ \mid \zeta_\ep(x)>0 \}$ and $\text{diam}(  \Om_\ep ):= \sup_{x,y\in \Om_\ep} |x-y|$ defines the diameter of the set $ \Om_\ep$.

The equations \eqref{2-2} can be rewritten as the following integral equation
\begin{equation*}
	\psi_\varepsilon(  x)=\frac{1}{\varepsilon^2}\int_{\mathbb R^2_+}{G_1}(  x, y) \left(\psi_\ep(y)-\frac{W}{2}y_1^2\ln\frac{1}{\ep}-\mu_\ep\right)_+^p y_1dy,
\end{equation*}
where ${G_1}( x, y)$ is the Green  function for $\mathcal{L}$ defined by \eqref{2-1}.

Using the method of moving planes in integral form by Chen, Li and Ou \cite{CLO} (see also the appendix in \cite{CQZZ1}), one can show that $\psi_\ep$ is even symmetric with respect to some line $\{x\mid x_2=c\}$. The proof is standard since the nonlinearity is Lipschitz and the kernel $G_1(x,y)$ is decreasing in $|x_2-y_2|$, so we omit it. Therefore, after suitable translations in the $x_2$-direction, we may assume that $\psi_\ep$  is even symmetric and decreasing with respect to the $x_1$-axis hereafter. That is, $\psi_\ep$  is even in the variable $x_2$ in the sense $\psi_\ep(x_1,-x_2)=\psi_\ep(x_1,x_2)$.

In the next several subsections, we will show that   $\psi_\ep$ satisfying system \eqref{2-2}-\eqref{2-5} is unique for small $\ep$. The main result of this section is as follows, which is  equivalent to Theorem \ref{thmU}.
\begin{theorem}\label{thm2-1}
	The function $\psi_\ep$ that satisfies \eqref{2-2}-\eqref{2-5} and $\psi_\ep(x_1,-x_2)=\psi_\ep(x_1,x_2)$  is unique as long as $\ep$ is sufficiently small.
\end{theorem}

The proof of Theorem \ref{thm2-1} is extremely delicate and can be divided into three steps:
\begin{itemize}
	\item [Step 1. ] We first study the limiting shape of the vortex ring  using scaling and the technique of blow-up analysis in PDEs. The obtained limiting shape inspires us to construct suitable approximate solutions.
	\item [Step 2. ] By constructing suitable  approximate solutions and estimating the error of these approximations, we obtain a more refined understanding of the asymptotic behavior of the vortex ring solution.
	\item [Step 3. ] In the last step, with sufficiently detailed estimates on  asymptotic behaviors,  we prove uniqueness through a proof by contradiction. Assuming the existence of two distinct vortex ring solutions, we utilize local Pohozaev identities and asymptotic behavior estimates obtained in Step 2 to derive a contradiction, thus establishing uniqueness.
\end{itemize}
Therefore, we divide the rest proof into three subsections accordingly.

\subsection{Rough estimate by blowing up}\label{sec2-1}
In this subsection, we first study the asymptotic behavior of the solutions by  blow-up argument.

Let $r_{\ep}:=\frac{1}{2}\text{diam} (\Om_{\ep})$  be the radius of $\Om_\ep$. To determine the location of  $\Om_\ep$, we choose $z_{\ep}\in \Om_{\ep}$ with
$$
\psi_\ep(z_{\ep})=
\max_{x\in \Om_{\ep}}\left(\psi_\ep(x)-\frac{W}{2}x_1^2\ln\frac{1}{\ep}\right),
$$
whose existence is ensured by the maximum principle (see e.g. Theorem 3.5 in \cite{GT} and the claim below it).
Then $r_{\ep}\leq C_0 \ep$ and $z_\ep \to x_0\in \mathbb{R}_+^2$ as $\ep\to 0$ by our assumptions \eqref{2-4} and \eqref{2-5}. Moreover, it holds $z_{\ep,2}=x_{0,2}=0$ by the assumption that $\psi_\ep$  is even and decreasing in the variable $x_2$.

Let $U$ be the unique radial solution of
\begin{equation}\label{2-13}
	\begin{cases}
		-\Delta u=u^p,\,\,\, u>0\,  \, \, &\text{in}\, B_1(0),\\
		u=0,&\text{on}\, \partial B_1(0).
	\end{cases}
\end{equation}
Denote  $\Lambda_p:=\int_{B_1(0)} U^p(x) dx $. By integration by parts and the Pohozaev identity, we have the following identities
\begin{equation}\label{2-14}
	\Lambda_p=-2\pi U'(1),\ \ \ \int_{B_1(0)} U^{p+1}(x) dx= \frac{\pi (p+1)}{2}|U'(1)|^2 .
\end{equation}
We define the harmonic extension of $U$ in $\mathbb{R}^2$ (still denoted by $U$ with abuse of notation) as
\begin{equation}\label{2-15}
	U(x):= \begin{cases} U(x), \quad &|x|\leq 1,\\  -U'(1) \ln \frac{1}{|x|}, &|x|>1.\end{cases}
\end{equation}

Throughout this section, we use the notation $o_\ep(1)$ to represent a term that approaches to 0 as $\ep\to0$, and $O_\ep(1)$ to represent a term that is  bounded by a constant independent of $\ep$ for $\ep$ sufficiently small.

Using blowing up analysis, we will show the following asymptotic behavior for $\psi_\ep$ and $\mu_\ep$.
\begin{proposition}\label{prop2-5}
	Let $L>0$ be an arbitrary large constant. The following asymptotic behaviors hold  as $\varepsilon\to 0$,
	\begin{equation}\label{2-16}
		\psi_{\varepsilon}( x)=z_{\varepsilon,1}^{-\frac{2}{p-1}} \left(\frac{\varepsilon }{r_\varepsilon }\right)^{\frac{2}{p-1}}\left(U\left(\frac{ x-  z_\varepsilon}{r_\varepsilon}\right)+o_\varepsilon(1)\right)+\frac{W}{2}z_{\varepsilon,1}^2\ln\frac{1}{\varepsilon}+\mu_\varepsilon,\ \ \    x\in B_{L r_\varepsilon}(  z_\varepsilon),
	\end{equation}
	\begin{equation}\label{2-17}
		z_{\varepsilon,1}^{\frac{2}{p-1}} \left(\frac{r_{\ep} }{\ep}\right)^{\frac{2}{p-1}}\left(\frac{\kappa  z_{\ep,1} }{2\pi} \ln \frac{1}{r_\ep  }+ \frac{\kappa  z_{\ep,1} }{2\pi}\left(\ln( 8z_{\ep,1})  -2\right)-\frac{W}{2}z_{\ep,1}^2\ln\frac{1}{\varepsilon}-\mu_\ep\right)=o_\varepsilon(1),
	\end{equation}
	and
	\begin{equation}\label{2-18}
		\kappa\left( \frac{r_\ep}{\ep}\right)^{\frac{2}{p-1}} z_{\ep,1}^{\frac{p+1}{p-1}} \to  -  2\pi U'(1).
	\end{equation}
	Moreover, \eqref{2-16} holds in the topology of $C^1(B_{L r_\varepsilon}(  z_\varepsilon))$.
\end{proposition}

Take $\delta_0$ to be a small constant such that $0<\delta_0<\min\{1/100, x_{0,1}/2\}$, where $x_{0,1}$ is first   coordinate of the point $x_0$ in \eqref{2-5}. The proof of Proposition \ref{prop2-5} relies on the utilization of several preliminary lemmas.  We first give the estimates for the stream function $\psi_{\ep}$ away from the vorticity set $\Omega_{\ep}$.
	\begin{lemma}\label{lem2-1}
	For every $ x\in \mathbb R^2_+$ with $2r_\ep<\mathrm{dist}(x, \Om_\ep)<\delta_0$, we have
	\begin{equation*}
		\begin{split}
		\psi_{ \varepsilon}( x)&=\frac{\kappa x_1^{\frac{1}{2}}z_{\ep,1}^{\frac{1}{2}}}{2\pi} \ln \frac{1}{|  x-  z_\varepsilon|}+\frac{\kappa x_1^{\frac{1}{2}}z_{\ep,1}^{\frac{1}{2}}}{4\pi}\left(\ln(x_1 z_{\ep,1})+2\ln8 -4\right)\\
		&+O\left(\frac{ r_\varepsilon}{|  x-  z_\varepsilon|}+|x-z_\ep|^2\ln\frac{1}{|x-z_\ep|}\right).
		\end{split}
	\end{equation*}
\end{lemma}
\begin{proof}
For every $ x\in \mathbb R^2_+$ with $2r_\ep<\text{dist}(x, \Om_\ep)<\delta_0$, the expansion  holds
	\begin{equation*}
		|x-y|=|x-z_{\ep}|-\langle\frac{x-z_{\ep}}{|x-z_{\ep}| },y-z_{\ep}\rangle+O\left(\frac{|y-z_{\ep}|^2}{|x-z_{\ep}| }\right) \ \ \ \forall y\in\Omega_{\ep}.
	\end{equation*}
	Then by the above expansion, \eqref{2-7} and  \eqref{2-3}, we get
	\begin{align*}
			&\psi_{ \ep}(x)=  \int_{\Omega_{\ep}}G_1(x,y)\zeta_\ep(y)y_1 dy  \\
			=&\int_{\Omega_{\ep}}\frac{x_1^{1/2}y_1^{  1/2}}{4\pi}\left(2\ln\left(\frac{1}{|x-y|}\right)+\ln(x_1y_1)
			+2\ln 8-4\right)\zeta_\ep(y)y_1 dy+O\left(|x-z_\ep|^2\ln\frac{1}{|x-z_\ep|}\right)\\
			=&\int_{\Omega_{\ep}}\frac{x_1^{1/2}y_1^{  1/2}}{2\pi}  \ln\left(\frac{1}{|x-y|}\right) \zeta_\ep(y)y_1 dy+\frac{\kappa x_1^{1/2}z_{\ep,1}^{  1/2}}{4\pi} \left(\ln(x_1 z_{\ep,1})
			+2\ln 8-4\right)\\
			&+O\left(r_\ep+|x-z_\ep|^2\ln\frac{1}{|x-z_\ep|}\right)\\
			=&\int_{\Omega_{\ep}}\frac{x_1^{1/2}y_1^{  1/2}}{2\pi}  \ln\left(\frac{1}{|x-z_\ep|}\right) \zeta_\ep(y)y_1 dy+\int_{\Omega_{\ep}}\frac{x_1^{1/2}y_1^{  1/2}}{2\pi}  \ln\left(\frac{|x-z_\ep|}{|x-y|}\right) \zeta_\ep(y)y_1 dy\\
			&+\frac{\kappa x_1^{1/2}z_{\ep,1}^{  1/2}}{4\pi} \left(\ln(x_1 z_{\ep,1})
			+2\ln 8-4\right)+O\left(r_\ep+|x-z_\ep|^2\ln\frac{1}{|x-z_\ep|}\right)\\
			=&\frac{\kappa x_1^{\frac{1}{2}}z_{\ep,1}^{\frac{1}{2}}}{2\pi} \ln \frac{1}{|  x-  z_\varepsilon|}+\frac{\kappa x_1^{\frac{1}{2}}z_{\ep,1}^{\frac{1}{2}}}{4\pi}\left(\ln(x_1 z_{\ep,1})+2\ln8 -4\right)+O\left(\frac{ r_\varepsilon}{|  x-  z_\varepsilon|}+|x-z_\ep|^2\ln\frac{1}{|x-z_\ep|}\right),
		\end{align*}
	which completes the proof.
\end{proof}

Next we turn to study the local behavior of $\psi_{ \varepsilon}$ near $  z_\varepsilon$ by blow-up argument. We first prove the following lemma, which means the kinetic energy of the flow in vortex core is bounded.
\begin{lemma}\label{lem2-2}
	As $\varepsilon\to 0$, it holds
	\begin{equation*}
		\frac{1}{\varepsilon^2}\int_{\mathbb{R}_+^2} \left(\psi_\varepsilon(x)-\frac{W}{2}\ln\frac{1}{\varepsilon}x_1^2-\mu_\varepsilon\right)_+^{p+1}d  x=O_\varepsilon(1).
	\end{equation*}
\end{lemma}
\begin{proof}
	Take $\psi_+=\left(\psi_\varepsilon-\frac{W}{2}x_1^2\ln\frac{1}{\varepsilon}-\mu_\varepsilon\right)_+$ as the upper truncation of $\psi_\varepsilon$. From equation \eqref{2-2}, we conclude that $\psi_+$ satisfies
	\begin{equation*}
		\begin{cases}
			\mathcal{L}\psi_+=\frac{1}{\ep^2}\left(\psi_\ep-\frac{W}{2}x_1^2\ln\frac{1}{\ep}-\mu_\ep\right)_+^p,\  \ \  \text{in}\ \mathbb R^2_+,\\
			\psi_+(  x)=0, \ \ \ \text{on} \ \partial \Om_\varepsilon.
		\end{cases}
	\end{equation*}
	Multiplying the above equation by $x_1\psi_+$ and  integrating by part, we  obtain
	\begin{equation*}
		\int_{\Om_\varepsilon}\frac{1}{x_1}|\nabla\psi_+|^2d  x=\frac{1}{\varepsilon^2}\int_{\Om_\varepsilon}x_1\psi_+^{p+1}d x\le C \ep^{-2}  \int_{\Om_\varepsilon}|\psi_+|^{p+1}d   x,
	\end{equation*}
	where we have used the fact  $0<x_{0,1}/2<x_1<2x_{0,1}$ for $   x\in \Om_\ep$ when $\ep$ is small due to \eqref{2-5}. By the Sobolev embedding and the H\"older inequality, we find
	\begin{equation*}
		\int_{\Om_\varepsilon}|\psi_+|^{p+1}d   x\le C\|\nabla\psi_+\|_{L^{\frac{2(p+1)}{p+3}}(\Om_\ep)}^{p+1}\leq C \ep^2 \|\nabla\psi_+\|_{L^{2}(\Om_\ep)}^{p+1},
	\end{equation*}
where we have used $|\Om_\ep|\leq C\ep^2$ due to \eqref{2-4}.
	Hence we deduce
	\begin{equation*}
		\|\nabla\psi_+\|_{L^{2}(\Om_\ep)}^2\leq C\int_{\Om_\varepsilon}\frac{1}{x_1}|\nabla\psi_+|^2d  x \leq C\ep^{-2}\int_{\Om_\varepsilon}|\psi_+|^{p+1}d   x \le C\|\nabla\psi_+\|_{L^{2}(\Om_\ep)}^{p+1}.
	\end{equation*}
	Since $p>1$, we   obtain
	\begin{equation*}
		\frac{1}{\varepsilon^2}\int_{\Om_\varepsilon}x_1\psi_+^{p+1}d  x=\int_{\Om_\varepsilon}\frac{1}{x_1}|\nabla\psi_+|^2d x=O_\varepsilon(1),
	\end{equation*}
	which proves the desired estimate   by the definition of $\psi_+$.
\end{proof}

To study the local behaviors of $\psi_\ep$ near $z_{\ep}$,   we introduce the  scaling  of $\psi_{ \varepsilon}$
\begin{equation}\label{2-10}
	w_\varepsilon(  y)=z_{\varepsilon,1}^{\frac{2}{p-1}} \left(\frac{r_\varepsilon }{\varepsilon }\right)^{\frac{2}{p-1}}\left(\psi_{ \varepsilon}(r_\varepsilon   y+  z_\varepsilon)  -\frac{W}{2}(r_\ep y_1+z_{\varepsilon,1})^2\ln\frac{1}{\varepsilon}-\mu_\varepsilon\right).
\end{equation}
Then we derive from the equation of $\psi_\ep$ \eqref{2-2}  that $w_\varepsilon$ satisfies
\begin{equation}\label{2-11}
	-\frac{1}{r_\ep y_1+z_{\ep,1}}\text{div}\left(\frac{1}{r_\ep y_1+z_{\ep,1}}\nabla w_\ep(y)\right) =z_{\varepsilon,1}^{-2} \left(w_\ep \right)_+^p  (y).
\end{equation}

 To show the convergence of $w_\ep$ to some limiting function, we first prove that $w_\ep$ is locally bounded independent of $\ep$.
\begin{lemma}\label{lem2-3}
	For any $R>0$, there exists a constant $C_R>0$  independent of $\varepsilon$ such that $$||w_\varepsilon||_{L^\infty(B_R( 0))}\leq C_R.$$
\end{lemma}
\begin{proof}
	It follows from Lemma  \ref{lem2-2}  that
	\begin{align*}
		O_\varepsilon(1)&=\frac{1}{\varepsilon^2}\int_{\Om_\varepsilon} \left(\psi_\varepsilon-\frac{W}{2}\ln\frac{1}{\varepsilon}x_1^2-\mu_\varepsilon\right)_+^{p+1}d  x\\
		&=\left(z_{\ep,1}\right)^{-\frac{2(p+1)}{p-1}}\left(\frac{\ep}{r_\ep}\right)^{\frac{4}{p-1}}\int_{\mathbb R^2} \left(w_\varepsilon \right)_+^{p+1}(y) dy.
	\end{align*}
	Noticing that $r_\ep\le C_0\ep$ and $x_{0,1}/2<z_{\varepsilon,1}<2x_{0,1}$, we deduce
	$$ \int_{\mathbb R^2} (w_\varepsilon)_+^{p+1}d  y\le C.$$
	Applying Morse iteration on \eqref{2-11} (Theorem 4.1 in \cite{Han}), we  obtain
	$$||(w_\varepsilon)_+||_{L^\infty(B_{2R}(  0))}\leq C.$$
	To prove that the $L^\infty$ norm of $w_\varepsilon$ is bounded independent of $\ep$, we consider the following problem.
	\begin{equation*}
		\begin{cases}	-\frac{1}{r_\ep y_1+z_{\ep,1}}\text{div}\left(\frac{1}{r_\ep y_1+z_{\ep,1}}\nabla w_1(y)\right)   = z_{\varepsilon,1}^{-2} \left(w_\ep \right)_+^p  (y),\quad  &\text{in} \ B_{2R}(  0),\\
			w_1=0,&\text{on} \ \partial B_{2R}( 0).\end{cases}
	\end{equation*}
	It is obvious that $|w_1|\le C$ by the theory of regularity for elliptic equations and Sobolev embedding. Let $w_2:=w_\ep-w_1$, then $w_2$  satisfies
	$$\sup_{B_{2R}(0)} w_2\geq \sup_{B_{2R}(0)} w_\ep-C\geq -C,$$
	since $\sup_{B_{2R}(0)} w_\ep\geq 0$. On the other hand, we infer from $||(w_\ep)_+||_{L^\infty(B_{2R}(0))}\leq C$ that
	$$\sup_{B_{2R}(0)} w_2\leq \sup_{B_{2R}(0)} w_\ep+C\leq M,$$
	for some constant $M$. Thus,    by the Harnack inequality (Theorem 8.20 in \cite{GT}), we have
	$$\sup_{B_R(0)}( M- w_2)\leq C\inf_{B_R(0)}(M- w_2)\leq C(M+\sup_{B_R(0)} w_2)\leq C.$$
	Since $\sup_{B_R(0)}( M- w_2)=M-\inf_{B_R(0)} w_2$, we obtain $$\inf_{B_R(0)} w_2\geq C,$$ which implies the boundedness of $w_\ep$ and finishes the proof.
\end{proof}

Now, we are ready to prove Proposition \ref{prop2-5}.

\begin{proof}[\textbf{Proof of Proposition \ref{prop2-5}}]
	Let $0<L<R$ be two large fixed constants. For $  y\in B_R(  0)\setminus B_L(0)$, we infer from Lemma \ref{lem2-1} that
	\begin{align*}
		&\quad w_\varepsilon(  y)=z_{\varepsilon,1}^{\frac{2}{p-1}} \left(\frac{r_\varepsilon }{\varepsilon }\right)^{\frac{2}{p-1}}\left(\psi_{ \varepsilon}(r_\varepsilon   y +z_\ep) -\frac{W}{2}(r_\ep y_1+z_{\varepsilon,1})^2\ln\frac{1}{\varepsilon}-\mu_\varepsilon\right)\\
		&=z_{\varepsilon,1}^{\frac{2}{p-1}} \left(\frac{r_{\ep} }{\ep}\right)^{\frac{2}{p-1}}\left(\frac{\kappa  z_{\ep,1} }{2\pi} \ln \frac{1}{r_\ep |y|}+\frac{\kappa  z_{\ep,1} }{2\pi}\left(\ln( 8z_{\ep,1})  -2\right)-\frac{W}{2}z_{\ep,1}^2\ln\frac{1}{\varepsilon}-\mu_\ep+O\left(\frac{1}{L}+r_\ep \right)\right)\\
		&= z_{\varepsilon,1}^{\frac{p+1}{p-1}}\left(\frac{r_{\ep} }{\ep}\right)^{\frac{2}{p-1}}\frac{\kappa}{2\pi  }\ln \frac{1}{| y|}\\
		&\quad+ z_{\varepsilon,1}^{\frac{2}{p-1}} \left(\frac{r_{\ep} }{\ep}\right)^{\frac{2}{p-1}}\left(\frac{\kappa  z_{\ep,1} }{2\pi} \ln \frac{1}{r_\ep  }+ \frac{\kappa  z_{\ep,1} }{2\pi}\left(\ln( 8z_{\ep,1})  -2\right)-\frac{W}{2}z_{\ep,1}^2\ln\frac{1}{\varepsilon}-\mu_\ep\right)+O\left(\frac{1}{L}+r_\ep \right).
	\end{align*}
	Since $z_{\ep,1}\to x_{0,1}$, $|r_\varepsilon| \le C_0 \ep $ and  $||w_\varepsilon||_{L^\infty(B_R( 0))}\leq C_R$ by Lemma \ref{lem2-3}, we may assume that up to a subsequence, as $\ep\to0$,
	$$  \frac{r_{\ep} }{\ep} \to  \alpha\in [0, +\infty),$$
	and
	\begin{equation*}
		\begin{split}
			z_{\varepsilon,1}^{\frac{2}{p-1}} \left(\frac{r_{\ep} }{\ep}\right)^{\frac{2}{p-1}}\left(\frac{\kappa  z_{\ep,1} }{2\pi} \ln \frac{1}{r_\ep  }+ \frac{\kappa  z_{\ep,1} }{2\pi}\left(\ln( 8z_{\ep,1})  -2\right)-\frac{W}{2}z_{\ep,1}^2\ln\frac{1}{\varepsilon}-\mu_\ep\right)\to \beta\in [0, +\infty),
		\end{split}
	\end{equation*}
	for some constants $\alpha, \beta$.
	By \eqref{2-11} and theory of regularity for elliptic equations, we may further assume that up to a subsequence, $w_\varepsilon\to w$ in $C_{\text{loc}}^1(\mathbb{R}^2)$ for some function $w$ and $w$ satisfies
	\begin{equation}\label{2-12}
		\begin{cases}
			-\Delta w=(w)_+^p,\quad &\text{in}\,B_R( 0),\\
			w=\frac{\kappa \alpha^{\frac{2}{p-1}}  x_{0,1}^{\frac{p+1}{p-1}}}{2\pi }\ln \frac{1}{| y|}+\beta+O\left(\frac{1}{L}\right), &\text{in} \, B_R(  0)\setminus B_L(  0).
		\end{cases}	
	\end{equation}
	Moreover, $w$ will satisfy the integral equation, which can be verified by direct calculation,
	$$w( y)=\frac{1}{2\pi}\int_{\mathbb{R}^2} \ln\left(\frac{1}{|  y-  y'|}\right) \boldsymbol (w)_+^p(y')d  y'+\beta.$$
	Then the method of moving planes shows that $w$ is radially symmetric and decreasing (See e.g. \cite{CLO}).

	By the definition of  $r_\ep$, there exists $y_{\ep}$ with $|y_{\ep}|=1$ such that $r_{\ep} y_{\ep}+z_{\ep}\in \partial \Om_{\ep}$. Thus the radially symmetric function $w$   satisfies $\partial \{w>0\}= \partial B_1(0)$ and hence $w=U$ due to the uniqueness of solutions to \eqref{2-13}.
	Then, one can see from \eqref{2-12} and \eqref{2-15} that $\frac{ \kappa\alpha^{\frac{2}{p-1}}  x_{0,1}^{\frac{p+1}{p-1}}}{2\pi }= -  U'(1) $ and  $\beta =O\left(\frac{1}{L}\right)=0$ since $L$ is arbitrary.
	
	Note that by the above proof, each  sequence in $\{\omega_\ep\}$ has a subsequence that converges to the same limit $U$. A simple contradiction argument will show that $\{\omega_\ep\}$ itself must converge to $U$ as $\ep\to0$. The    proof is therefore finished.
\end{proof}

\subsection{Improvement for the estimates of asymptotic behavior}\label{sec2-2}
The estimates in Proposition \ref{prop2-5} are not accurate enough for the proof of uniqueness. In this subsection, we will  improve these  estimates by choosing properly approximate solutions according to the  asymptotic behavior obtained in previous subsection. By studying the error of our approximations, we are able to obtain refined asymptotic behaviors (i.e. Proposition \ref{prop2-13} below), which enables us to show the uniqueness of vortex rings.

\subsubsection{The first approximation}\label{sec2-2-1}
In view of \eqref{2-16}, we will first approximate $\psi_{\ep}$ by a proper scaling of the ground state $U$ near $z_\ep$.
For given $a>0$ and $z\in\mathbb{R}_+^2$, we define
\begin{equation}\label{2-19}
	U_{\ep, z, a}(x)=\begin{cases} a\ln \frac{1}{\ep}+z_{  1}^{-\frac{2}{p-1}}\left(\frac{\ep}{s  }\right)^{\frac{2}{p-1}}U\left(\frac{x-z}{s }\right),\quad &|x-z|\le s  ,\\ a\frac{\ln \ep}{\ln s  }\ln\left(\frac{s }{|x-z|}\right),&|x-z|\ge s  ,\end{cases}
\end{equation}
where   $s $ is a positive constant so chosen that $U_{\ep,z,  a}\in C^1(\mathbb{R}^2)$. Namely, we require
$$z_{ 1}^{-\frac{2}{p-1}}\left(\frac{\ep}{s  }\right)^{\frac{2}{p-1}}U'(1)=-\frac{a\ln \ep}{  \ln s  }.$$

We approximate $\psi_{\ep}$  by the function

\begin{equation}
	\Psi_{\ep, x_\ep,  a_{\ep}}(x):=\frac{1}{\ep^2}\int_{\mathbb{R}_+^2} G_1(x,y) \left(U_{\ep,x_\ep, a_\ep}(y)- a_\ep\ln \frac{1}{\ep}\right)_+^p y_1dy
\end{equation}
where  $ x_\ep$, $ a_{\ep} $ and $s_{\ep} $ are appropriately chosen such that $\Psi_{\ep, x_\ep,  a_{\ep}}$ becomes a batter approximation of $\psi_\ep$. In fact, we can choose $ x_\ep $, $  a_\ep$ and $s_{\ep} $ so that $(z, a,s)=(x_\ep,  a_{\ep}, s_{\ep})$  is a solution to the following system of equations
\begin{equation}\label{2-21}
	\begin{cases}
		\partial_{x_1} \Psi_{\ep, z,  a }(z_{\ep})=0,\\
		a \ln \frac{1}{\ep}=\frac{ W}{2}z_{ 1}^2\ln \frac{1}{\ep}+\mu_{\ep}-\frac{z_{ 1}^{-\frac{2}{p-1}}}{2\pi}\left(\frac{\ep}{s }\right)^{\frac{2}{p-1}}\Lambda_p  \left( \ln(8z_{ 1}) -2  \right),\\
		z_{ 1}^{-\frac{2}{p-1}}\left(\frac{\ep}{s  }\right)^{\frac{2}{p-1}}U'(1)=-\frac{a \ln \ep}{  \ln s  }.
	\end{cases}
\end{equation}
Choosing $z_2=z_{\ep,2}=0$ and using the expansion   for  $	\Psi_{\ep, x_\ep,  a_{\ep}}$  in Lemma \ref{lemB-1} in the Appendix \ref{appB}, we can see that the above system is uniquely solvable for $\ep$ sufficiently small by  contraction mapping  theorem.  Moreover,  by direct calculations,  one has the following estimates.
\begin{equation}\label{2-22}
	\begin{cases}
		|z_{\ep,1}-x_{\ep,1}|=O\left( \ep^2 \ln \frac{1}{\ep}\right),\\
		a_{\ep}\ln\frac{1}{  \ep }=\frac{ W}{2}x_{\ep,1}^2\ln \frac{1}{\ep}+\mu_{\ep}+O\left( 1\right),\\
		|r_{\ep}-s_{\ep}|=o\left( \ep \right).
	\end{cases}
\end{equation}

Denoting the difference between $\psi_\varepsilon$ and $\Psi_{\ep, x_\ep,  a_{\ep}}$ as the error term
\begin{equation}\label{2-23}
	\phi_\varepsilon( x):=\psi_\varepsilon( x)-\Psi_{\ep, x_\ep,  a_{\ep}}(x),
\end{equation}
then our task in this part is to  estimate  $\phi_\varepsilon$. A direct consequence of Proposition \ref{prop2-5} is that $\phi_\varepsilon$ tends to $0$ as $\ep\to0$ as the following lemma.
\begin{lemma}\label{lem2-6}
	As $\varepsilon\to 0$, it holds
	$$\|\phi_\varepsilon\|_{L^\infty(\mathbb{R}^2_+)}+\ep||\nabla \phi_\ep||_{L^\infty(\mathbb{R}^2_+)}=o_\varepsilon(1).$$
\end{lemma}
\begin{proof}
	By \eqref{2-16}, \eqref{2-22} and \eqref{B-2} in the Appendix \ref{appB}, it can be verified by direct calculations that
	\begin{equation*}
		\phi_\ep\rightarrow 0,\, \text{as}\  \ep\to 0\ \text{uniformly in } B_{Lr_{\ep}}(z_{\ep}).
	\end{equation*}

	Since $r_\ep\leq C_0\ep$,  the functions $\left(\psi_\ep(x)-\frac{W}{2}x_1^2\ln\frac{1}{\ep}-\mu_\ep\right)_+^p$ and $\left(U_{\ep,x_\ep, a_\ep}(x)- a_\ep\ln \frac{1}{\ep}\right)_+^p$ are both supported in $B_{Ls_\ep}(x_\ep)$   when $L$ is sufficiently large.
	Therefore, $\mathcal L \phi_\ep\equiv 0$  in  $ \mathbb{R}_+^2\setminus B_{Lr_{\ep}}(z_{\ep})$. Since $\phi_\ep=0$ on $\partial \mathbb{R}^2_+$ and $\phi_\ep\rightarrow 0$ as $|x|\rightarrow \infty$ uniformly for $\ep$ small, we infer from the maximum principle that $||\phi_\ep||_{L^\infty(\mathbb{R}^2_+)} \rightarrow 0\,\,\, \text{as}\, \ep\rightarrow 0 $.
	
	To estimate $||\nabla \phi_\ep||_{L^\infty(\mathbb{R}^2_+)}$,  we use the integral equation.  Notice that by \eqref{2-7}, one has $|\nabla G_1(x,y)|\leq \frac{C}{|x-y|}$ for $x,y\in B_{Ls_\ep}(x_\ep)$. Then for $x\in \mathbb{R}^2_+$, by \eqref{2-16}, \eqref{2-19}, \eqref{2-22}, $\phi_\ep=o_\ep(1)$ and  rearrangement inequality, we have
	\begin{align*}
			&\quad| \ep \nabla \phi_\ep (x)| \leq  \frac{1}{2\pi\ep }\int_{\mathbb{R}^2_+}   \frac{C}{|x-y|}  \Bigg|(\psi_\ep-Wx_1-\mu_\ep)_+^p- \left(U_{\ep, x_\ep, a_{\ep} }-a_\ep \ln \frac{1}{\ep}\right)_+^p  \Bigg|y_1 dy\\
			&\leq C\ep^{-1}\int_{\mathbb{R}^2_+}  \frac{1}{|x-y|} \left|  \left(\left(\frac{\ep}{s_{\ep} }\right)^{\frac{2}{p-1}}U\left(\frac{x-x_\ep}{s_\ep}\right)+o_\ep(1)\right)_+^p- \left(\left(\frac{\ep}{s_{\ep} }\right)^{\frac{2}{p-1}}U\left(\frac{x-x_\ep}{s_\ep}\right)\right)_+^p\right|\\
			&= \frac{o_\ep(1)}{  \ep }\int_{B_{Ls_\ep(x_\ep)}}  \frac{p}{|x-y|} \left| \left(\left(\frac{\ep}{s_{\ep} }\right)^{\frac{2}{p-1}}U\left(\frac{x-x_\ep}{s_\ep}\right)\right)^{p-1}\right|\\
			&= \frac{o_\ep(1)}{  \ep }\int_{B_{Ls_\ep}(x_\ep)}  \frac{1}{| y|} dy =o_\ep(1).
		\end{align*}
The proof is therefore complete.
\end{proof}

The estimate  in Lemma \ref{lem2-6} enables us to prove the Kelvin--Hicks formula for the location $x_\ep$ via a local Pohozaev identity.

Let $\bar x=(-x_1,x_2)$ be the reflection of $x$ with respect to the $x_2$-axis. To continue, we shall decompose the Green function ${G_1}$ into a combination of singular   and regular parts according to $x_\ep$ as follows.
\begin{equation}\label{2-24}
	y_1{G_1}( x, y)=x_{\ep,1}^2G( x, y)+ H( x, y),
\end{equation}
where
\begin{equation*}
	G( x,  y)=\frac{1}{4\pi}\ln\frac{(x_1+y_1)^2+(x_2-y_2)^2}{(x_1-y_1)^2+(x_2-y_2)^2},
\end{equation*}
is the Green function for $-\Delta$ in right half plane $\mathbb{R}_+^2$.  By \eqref{2-7}, $H( x, y)$ is a relatively regular function given by
\begin{equation}\label{2-25}
	\begin{split}
		H( x,y)&= \frac{x_1^{1/2}y_1^{3/2}-x_{\ep,1}^2}{2\pi}  \ln\frac{1}{|  x-y|}+\frac{x_{\ep,1}^2}{2\pi}\ln\frac{1}{| x- {\bar y}|}\\
		&\quad+\frac{x_1^{1/2}y_1^{3/2}}{4\pi}\left(\ln(x_1 y_1)+2\ln 8-4+  O(|\rho\ln \rho|)\right),
	\end{split}
\end{equation}
for $\rho$ small. Note that we omit the dependence of $\ep$ in by denoting the right hand side of \eqref{2-25} as $H(x,y)$ for simplicity of notations.

According to the decomposition of ${G_1}(x, y)$, we can split the stream function $\psi_\varepsilon$ into two parts $\psi_\varepsilon=\psi_{1,\varepsilon}+\psi_{2,\varepsilon}$, where
\begin{equation*}
	\psi_{1,\varepsilon}( x)=\frac{x_{\ep,1}^2}{\varepsilon^2}\int_{\mathbb R^2_+}G(  x,y)\left(\psi_\ep(y)-\frac{W}{2}y_1^2\ln\frac{1}{\ep}-\mu_\ep\right)_+^p dy ,
\end{equation*}
\begin{equation*}
	\psi_{2,\varepsilon}( x)=\frac{1}{\varepsilon^2}\int_{\mathbb R^2_+}H(x, y)\left(\psi_\ep(y)-\frac{W}{2}y_1^2\ln\frac{1}{\ep}-\mu_\ep\right)_+^p dy.
\end{equation*}
Moreover, it can be seen that $\psi_{1,\varepsilon}(x)$ solves the following problem
\begin{equation}\label{2-26}
	\begin{cases}
		-\Delta \psi_{1,\varepsilon}(x)=\frac{x_{\ep,1}^2}{\varepsilon^2}  \left(\psi_{1,\varepsilon}(x)+\psi_{2,\varepsilon}(x)-\frac{W}{2}x_1^2\ln\frac{1}{\ep}-\mu_\ep\right)_+^p,  \ \ \ &\text{in} \ \mathbb R^2_+,\\
		\psi_{1,\varepsilon}=0, &\text{on} \ \{x_1=0\},\\
		\psi_{1,\varepsilon}\to 0, &\text{as} \ |x|\to \infty.
	\end{cases}
\end{equation}

The Pohozaev identity (see e.g. (6.2.3)  in \cite{CPYb}) for \eqref{2-26} in a small ball $B_\delta(x_\ep)$ reads as follows.
\begin{align}\label{2-27}
		&\quad-\int_{\partial B_\delta(x_\ep)} \frac{\partial\psi_{1,\varepsilon}}{\partial \nu}\frac{\partial\psi_{1,\varepsilon}}{\partial x_1}ds+ \frac{1}{2}\int_{\partial B_\delta(x_\ep)} |\nabla\psi_{1,\varepsilon}|^2 \nu_1ds\nonumber\\
		&=-\frac{x_{\ep,1}^2}{\varepsilon^2}\int_{B_\delta(x_\ep)} \left(\psi_\ep(x)-\frac{W}{2}x_1^2\ln\frac{1}{\ep}-\mu_\ep\right)_+^p\partial_1\psi_{2,\varepsilon}(  x)   d x\\
		&\quad+\frac{x_{\ep,1}^2}{\varepsilon^2}\int_{B_\delta(x_\ep)} Wx_1\ln\frac{1}{\varepsilon}\left(\psi_\ep(x)-\frac{W}{2}x_1^2\ln\frac{1}{\ep}-\mu_\ep\right)_+^pd x.\nonumber
	\end{align}

Using \eqref{2-3}, \eqref{2-19},  \eqref{D-6} and the odd symmetry, we can calculate each term in \eqref{2-27}. Since the calculations are extremely complicated and technical, we collect them in the Appendix \ref{appC}. As a result of Lemmas \ref{lem2-6} and \ref{lemC-4}, we obtain the estimate of the location.
\begin{corollary}\label{cor2-7}
	It holds
	\begin{equation*}
		\begin{split}
			Wx_{\ep,1}\ln\frac{1}{\varepsilon}-\frac{\kappa}{4\pi}\left(\ln\frac{8x_{\ep,1}}{s_\ep}+\frac{p-1 }{4}\right)= o_\ep(1).
		\end{split}
	\end{equation*}
\end{corollary}
\begin{proof}
	By simply taking $\phi_\ep^{o}=0$ in the  decomposition \eqref{C-2}, we infer from the estimate of Lemma \ref{lem2-6} and the identity in Lemma \ref{lemC-4} that this corollary holds true.
\end{proof}

Next, we are going to improve the estimate in Lemma \ref{lem2-6} by studying the equation of $\phi_\ep$. By a linearization procedure, we see that $\phi_\varepsilon$  satisfies the equation
\begin{equation*}
	\mathbb L_\varepsilon\phi_\varepsilon=R_\varepsilon(\phi_\varepsilon),
\end{equation*}
where $\mathbb L_\varepsilon$ is the linear operator defined by
\begin{equation*}
	\mathbb L_\varepsilon\phi_\varepsilon= x_1\mathcal{L}\phi_\varepsilon-\frac{p x_{1}}{\ep^2}\left(U_{\ep, x_\ep, a_\ep}(x)- a_\ep\ln \frac{1}{\ep}\right)_+^{p-1} \phi_\varepsilon,
\end{equation*}
and the higher-order term
\begin{equation*}
	R_\varepsilon(\phi_\varepsilon)=\frac{x_1}{\varepsilon^2}\bigg(\left(\psi_\ep(x)-\frac{W}{2}x_1^2\ln\frac{1}{\ep}-\mu_\ep\right)_+^p -\left(U_{\ep, x_\ep, a_\ep}(x)- a_\ep\ln \frac{1}{\ep}\right)_+^{p}-p\left(U_{\ep, x_\ep, a_\ep}(x)- a_\ep\ln \frac{1}{\ep}\right)_+^{p-1} \phi_\varepsilon\bigg).
\end{equation*}
By \eqref{B-2}, \eqref{B-3}, Lemmas \ref{lem2-6} and \ref{lemB-2}, we obtain
\begin{equation*}
	R_\varepsilon(\phi_\varepsilon)\equiv0, \ \ \ \text{in} \ \mathbb R^2_+\setminus B_{2s_\ep}(x_\ep).
\end{equation*}

To derive a better estimate for $\phi_\varepsilon$, we need the following non-degeneracy of the linearized equation.
\begin{lemma}[Lemma 6 in \cite{FW}]\label{lem2-8}
	Let $U$ be the function   defined by \eqref{2-15}. Suppose that $v\in L^\infty(\mathbb{R}^2)\cap C(\mathbb{R}^2)$ solves
	\begin{equation*}
		-\Delta v-pU_+^{p-1}v=0,\quad \text{in}\,\,\,\mathbb{R}^2.
	\end{equation*}
	Then
	$$v\in span \left\{\frac{\partial U}{\partial x_1}, \frac{\partial U}{\partial x_2}\right\}.$$
\end{lemma}

With this non-degeneracy lemma, we are able to establish the following estimate for the linear  operator $\mathbb L_\varepsilon$ and $\phi_\ep$, which plays a central role in this section.
\begin{lemma}\label{lem2-9}
	Suppose that $\mathrm{supp} \,(\mathbb L_\varepsilon\phi_\varepsilon)\subset B_{2s_\ep}(  x_\ep)$. Then for any $q\in(2, +\infty]$, there exist  an $\epsilon_0>0$ small  and a constant $c_0>0$ such that for any $\varepsilon\in (0,\epsilon_0]$, it holds
	\begin{equation*}
		 \varepsilon^{1-\frac{2}{q}} ||\mathbb{L}_\varepsilon \phi_\varepsilon||_{W^{-1,q}(B_{Ls_\ep}(x_\ep))}+\varepsilon^2||\mathbb{L}_\varepsilon \phi_\varepsilon||_{L^\infty(B_{\frac{s_\ep}{2} }(x_\ep))}\ge c_0 \left(\varepsilon^{1-\frac{2}{q}} ||\nabla \phi_\varepsilon||_{L^{q}(B_{Ls_\ep}(x_\ep))}+||\phi_\varepsilon||_{L^\infty(\mathbb{R}_+^2)}\right)
	\end{equation*}
	with $L>0$ a large constant.
\end{lemma}
\begin{proof}
	We use the notation $o_n(1)$ to denote a term that approaches 0 as $n\to+\infty$. We will argue by contradiction. Suppose on the contrary that there exists $\varepsilon_n\to 0$ such that $\phi_n:=\phi_{\varepsilon_n}$ satisfies
	\begin{equation}\label{2-28}
		\varepsilon_n^{1-\frac{2}{q}} ||\mathbb{L}_{\varepsilon_n} \phi_n||_{W^{-1,q}(B_{Ls_{\ep_n}}(x_{\ep_n}))}+\varepsilon_n^2||\mathbb{L}_{\varepsilon_n} \phi_n||_{L^\infty(B_{s_{\ep_n}/2}(x_{\ep_n}))}\le \frac{1}{n},
	\end{equation}
	and
	\begin{equation}\label{2-29}
		\varepsilon_n^{1-\frac{2}{q}} ||\nabla \phi_n||_{L^{q}(B_{Ls_{\ep_n}}(B_{Ls_{\ep_n}}(x_{\ep_n}))}+||\phi_n||_{L^\infty(\mathbb{R}_+^2)}=1.
	\end{equation}
	By letting $f_n=\mathbb{L}_{\varepsilon_n}\phi_n$, we have
	\begin{equation*}
		-\text{div}\left(\frac{1}{x_1}\nabla\phi_n\right) =\frac{p x_{1}}{\ep^2}\left(U_{\ep_n,x_{\ep_n}, a_{\ep_n}} - a_{\ep_n}\ln \frac{1}{\ep_n}\right)_+^{p-1} \phi_n+f_n.
	\end{equation*}
	Henceforth we will denote $\tilde v( y):=v(s_{\ep_n} y+x_{\ep_n})$ for an arbitrary function $v$ and $\tilde{\mathbb R}^2_+:=\{y\in \mathbb{R}^2 \mid s_{\ep_n} y+x_{\ep_n}\in \mathbb R^2_+\}$. Then the above equation has a weak form
	\begin{equation*}
		\begin{split}
		& \int_{\tilde{\mathbb R}^2_+}\frac{1}{s_{\ep_n}y_1+x_{\ep_n,1}} \nabla\tilde\phi_n\cdot\nabla\varphi d  y=\frac{p s_{\ep_n}^2}{\ep_n^2}\int_{\tilde{\mathbb R}^2_+}(s_{\ep_n} y_1+x_{\ep_n,1})\left(\tilde U_{\ep_n, x_{\ep_n},a_{\ep_n}}(y) - a_{\ep_n}\ln \frac{1}{\ep_n}\right)_+^{p-1} \tilde\phi_n(y)\varphi(y)dy \\
		&+\langle   f_n(s_{\ep_n} y+x_{\ep_n}),\varphi((y-x_{\ep_n})/s_{\ep_n})\rangle, \ \ \ \ \ \ \forall \, \varphi\in C_0^\infty(\tilde{\mathbb R}^2_+).
		\end{split}
	\end{equation*}

	The right hand side of the above equation is bounded in $W_{\text{loc}}^{-1,q}(\mathbb{R}^2)$,  since the function \\$\left(\tilde U_{\ep_n, x_{\ep_n},a_{\ep_n}}(y) - a_{\ep_n}\ln \frac{1}{\ep_n}\right)_+^{p-1}$ is a bounded function and $f_n$ is bounded in $W^{-1,q}(B_{Ls_{\ep_n}}(x_{\ep_n}))$ and $\mathrm{supp} \,(f_n)\subset B_{2s_{\ep_n}}(  x_{\ep_n})$. Therefore $\tilde \phi_n$ is bounded in $W^{1,q}_{\text{loc}}(\mathbb{R}^2)$ by theory of regularity for elliptic equations and hence bounded in $C_{\text{loc}}^\alpha (\mathbb{R}^2)$ for some $\alpha>0$ by Sobolev embedding. Then we may assume that $\tilde \phi_n$ converges uniformly in any compact subset of $\mathbb{R}^2$ to $\phi^*\in L^\infty(\mathbb{R}^2)\cap C(\mathbb{R}^2)$. It can be seen that the limiting function $\phi^*$ satisfies
	\begin{equation*}
		-\Delta \phi^* =pU_+^{p-1}\phi^*, \quad \text{in} \ \mathbb{R}^2,
	\end{equation*}
which, by Lemma \ref{lem2-8} and the even symmetry in $x_2$, implies that
$$ \phi^*=c_1\frac{\partial U}{\partial x_1}.$$
	On the other hand, since $|f_n|\leq \frac{1}{n}$ in $B_{\frac{1}{2}s_{\ep}}(x_{\ep})$ by \eqref{2-28}, $|\tilde \phi_{n }|\leq 1$  and $\left(\tilde U_{\ep, x_\ep, a_\ep}(y) - a_\ep\ln \frac{1}{\ep}\right)_+  $ is bounded, we can deduce that $\tilde \phi_{n }$ is bounded in $W^{2,s}({B_{\frac{1}{4}}(0)})$ for any $s>1$. So, we may assume $\tilde \phi_{n } \to \phi^*$ in $C^1({B_{\frac{1}{4}}(0)})$. In view the choice of $z_\ep$ at the beginning of Section 2.1 and the first equation in \eqref{2-21}, one can verify that  $\partial_1 \tilde \phi_{n } (\frac{z_{\ep_n }-x_{\ep_n}}{s_{\ep_n}})=s_{\ep_n}\partial_1 \phi_n(z_{\ep_n})=s_{\ep_n}Wz_{\ep_n} \ln\frac{1}{\ep_n}\to 0$ and $\frac{z_{\ep_n}-x_{\ep_n}}{s_{\ep_n}}\to 0$ as $n\to+\infty$, which implies $\nabla \phi^*(0)=0$. Thus we have $c_1=0$ and hence $\phi^*\equiv 0$, which implies $\phi_n=o_n(1)$ in $B_{Ls_{\ep_n}}(x_{\ep_n})$ for any $L>0$ fixed. Note that $\mathcal{L}\phi_n\equiv 0$   in $\mathbb{R}^2_+\setminus B_{Ls_{\ep_n}}(x_{\ep_n})$, $\phi_n=0$ on $\partial\mathbb{R}^2_+$ and   $\phi_n(x)\to 0$ as $|x| \to \infty$ uniformly for $n$. By the maximum principle, we obtain $$||\phi_n||_{L^\infty(\mathbb{R}^2)}=o_n(1).$$	
	For any $\varphi\in C_0^\infty (B_{L}(0))$, we have
	\begin{align*}
		&  \int_{\mathbb{R}^2} \frac{1}{s_{\ep_n}y_1+x_{\ep_n,1}}\nabla \tilde \phi_{n}\nabla \varphi =\frac{p s_{\ep_n}^2}{\ep_n^2}\int_{\tilde{\mathbb R}^2_+}(s_{\ep_n} y_1+x_{\ep_n,1})\left(\tilde U_{\ep_n, x_{\ep_n},a_{\ep_n}}(y) - a_{\ep_n}\ln \frac{1}{\ep_n}\right)_+^{p-1} \tilde\phi_n(y)\varphi(y)dy\\
		&\quad +\int_{\mathbb{R}^2}  f_n(s_{\ep_n} y+x_{\ep_n})\varphi \left(\frac{y-x_{\ep_n}}{s_{\ep_n}}\right)dy\\
		&=o_n(1)\int_{\mathbb{R}^2}| \varphi|+o_n(1)||\varphi||_{W^{1,p'}(B_{ L}(0))}\\
		&=o_n(1)||\nabla\varphi||_{L^{q'}(B_{L}(0))},
	\end{align*}
	which implies $||\nabla \tilde \phi_{n}||_{L^{q}(B_{ L}(0))}=o_n(1).$
	Therefore, we obtain
	$$s_{\ep_n}^{1-\frac{2}{q}} ||\nabla \phi_n||_{L^{q}( B_{Ls_{\ep_n}}(x_{\ep_n}))}=o_n(1).$$
	This  is a contradiction with \eqref{2-29} since $s_{\ep_n}=O(\ep_n)$, and hence the proof of Lemma \ref{lem2-9} is finished.
\end{proof}

Now we are in the position to improve the estimate for error term $\phi_\varepsilon$ and obtain the main proposition in this subsection.
\begin{proposition}\label{prop2-10}
There exists $\ep_0>0$ such that	for $q\in (2,+\infty]$ and $\varepsilon\in (0,\varepsilon_0]$, it holds
	\begin{equation}\label{2-30}
		||\phi_\varepsilon||_{L^\infty(\mathbb{R}_+^2)}+\varepsilon^{1-\frac{2}{q}} ||\nabla \phi_\varepsilon||_{L^{q}(B_{Ls_\ep}(x_\ep))}=O \left(\ep\right).
	\end{equation}
\end{proposition}	
\begin{proof}	
	In view of Lemma \ref{lem2-9}, it suffices to estimate $$\varepsilon^{1-\frac{2}{q}} ||R_\varepsilon(\phi_\varepsilon)||_{W^{-1,q}(B_{Ls_\ep}(x_\ep))}+\varepsilon^2||R_\varepsilon(\phi_\varepsilon)||_{L^\infty(B_{\frac{s_\ep}{2} }(x_\ep))}.$$
	
	For $x\in B_{\frac{s_\ep}{2}}$, by Lemma \ref{lemB-1} and the   Taylor expansion
	\begin{equation}\label{2-31}
		t^p=t_0^p+pt_0^{p-1}(t-t_0)+O_M((t-t_0)^2),\  \forall\ t,t_0\in(0, M),
	\end{equation}
we have
\begin{align*}
		 \ep^2R_\varepsilon(\phi_\varepsilon)
			&= {x_1} \left(\left(U_{\ep, x_\ep,  a_{\ep}}(x)-a_\ep\ln\frac{1}{\ep} + \mathcal F(x)+\phi_\ep+O(\ep^2 |\ln\ep| )\right)_+^p -\left(U_{\ep, x_\ep, a_\ep}(x)- a_\ep\ln \frac{1}{\ep}\right)_+^{p}\right.\\
			&\quad\left.-p\left(U_{\ep, x_\ep, a_\ep}(x)- a_\ep\ln \frac{1}{\ep}\right)_+^{p-1} \phi_\varepsilon\right)\\
			&=O\left((\|\phi_\ep\|_{L^\infty(\Om_\ep)}+\|\mathcal F\|_{L^\infty(\Om_\ep)}+\ep^2 |\ln {\ep} |)^{2} \right)\\
			&=O\left( \ep^2+ o_\ep(1)\|\phi_\ep\|_{L^\infty(B_{2s_\ep}(x_\ep))}\right).
		\end{align*}
	
	Thus, it remains to estimate $\varepsilon^{1-\frac{2}{q}} ||R_\varepsilon(\phi_\varepsilon)||_{W^{-1,q}(B_{Ls_\ep}(x_\ep))}$.	Let
	\begin{equation}\label{2-32}
		d_\ep=\sup_{\theta\in [0,2\pi)} |t_\ep(\theta)|
	\end{equation}   with $t_\ep$ is the parametrization of the boundary $\partial \Om_\ep$ defined by Lemma \ref{lemB-2} in Appendix \ref{appB}.  By \eqref{B-6}, we know $d_\ep=O\left( \|\phi_\ep\|_{L^\infty(B_{2s_\ep}(x_\ep))}+\|\mathcal F\|_{L^\infty(B_{2s_\ep}(x_\ep))}+\ep^2\right)$. For each $\varphi\in C_0^1(B_{Ls_\ep}(x_\ep))$,   by the H\"older inequality and Poincar\'e inequality, one has
\begin{equation}\label{2-38}
\int_{\mathbb{R}^2_+}|\varphi(x)|dx\leq (\pi L^2 s_\ep^2)^{\frac{1}{q}}\|\varphi\|_{L^{q'}(B_{Ls_\ep}(x_\ep))}\leq C L^{2/q} \ep^{1+\frac{2}{q}}\|\nabla\varphi\|_{L^{q'}(B_{Ls_\ep}(x_\ep))},
\end{equation}
where $q'$ stands for the conjugate exponent of $q$ defined by $1/q'+1/q=1$.

Thus, direct calculation yields
	\begin{align*}
			&\  \langle R_\varepsilon(\phi_\varepsilon),\varphi \rangle =\int_{\mathbb{R}^2_+}\frac{x_1}{\varepsilon^2}\left(\left(\psi_\ep(x)-\frac{W}{2}x_1^2\ln\frac{1}{\ep}-\mu_\ep\right)_+^p -\left(U_{\ep, x_\ep, a_\ep}(x)- a_\ep\ln \frac{1}{\ep}\right)_+^{p}\right.\\
			&\left.-p\left(U_{\ep, x_\ep, a_\ep}(x)- a_\ep\ln \frac{1}{\ep}\right)_+^{p-1} \phi_\varepsilon\right)\varphi(x)dx\\
			=&\left(\int_{B_{(1-d_\ep)s_\ep}(x_\ep)}+\int_{B_{(1+d_\ep)s_\ep}(x_\ep)\setminus B_{(1-d_\ep)s_\ep}(x_\ep)}\right)\frac{x_1}{\varepsilon^2}\left(\left(\psi_\ep(x)-\frac{W}{2}x_1^2\ln\frac{1}{\ep}-\mu_\ep\right)_+^p \right.\\
			&\left.-\left(U_{\ep, x_\ep, a_\ep}(x)- a_\ep\ln \frac{1}{\ep}\right)_+^{p}-p\left(U_{\ep, x_\ep, a_\ep}(x)- a_\ep\ln \frac{1}{\ep}\right)_+^{p-1} \phi_\varepsilon\right)\varphi(x)dx\\
			=&\frac{1}{\ep^2}\int_{B_{(1-d_\ep)s_\ep}(x_\ep)}x_1\left(p\left(U_{\ep, x_\ep, a_\ep}(x)- a_\ep\ln \frac{1}{\ep}\right)_+^{p-1} \mathcal{F}(x)+O(\|\phi_\ep\|_{L^\infty}+\|\mathcal F\|_{L^\infty})^{2}\right)\varphi(x)dx\\
			&+O\left( \frac{\|\phi_\ep\|_{L^\infty}+\|\mathcal F\|_{L^\infty}}{\ep^2}\int_{B_{(1+d_\ep)s_\ep}(x_\ep)\setminus B_{(1-d_\ep)s_\ep}(x_\ep)} |\varphi(x)|dx\right)\\
			=& O(\|\phi_\ep\|_{L^\infty}^{2}+\|\mathcal F\|_{L^\infty})\ep^{-2}\int_{B_{(1-d_\ep)s_\ep}(x_\ep)} |\varphi(x)|dx\\
			&+O\left(  \|\phi_\ep\|_{L^\infty}+\|\mathcal F\|_{L^\infty} \right)\ep^{-2}\int_{B_{(1+d_\ep)s_\ep}(x_\ep)\setminus B_{(1-d_\ep)s_\ep}(x_\ep)} |\varphi(x)|dx\\
			= &O\left((\|\phi_\ep\|_{L^\infty}^{2}+\|\mathcal F\|_{L^\infty})+(\|\phi_\ep\|_{L^\infty} +\|\mathcal F\|_{L^\infty}) d_\ep^{\frac{1}{q}}\right)\ep^{-1+\frac{2}{q}}\|\nabla\varphi\|_{L^{q'}(B_{Ls_\ep}(x_\ep))}\\
			=&O\left(\ep +o_\ep(1)\|\phi_\ep\|_{L^\infty}\right)\ep^{-1+\frac{2}{q}}\|\nabla\varphi\|_{L^{q'}(B_{Ls_\ep}(x_\ep))}.			
		\end{align*}
Here we have used the expansion \eqref{2-31} in $B_{(1-d_\ep)s_\ep}(x_\ep)$ and the following facts
\begin{equation*}
	\begin{cases}
		&\|\phi_\ep\|_{L^\infty}=o_\ep(1),\ \  \|\mathcal F\|_{L^\infty(B_{2s_\ep}(x_\ep))}=O(\ep),\\
		&\text{the area}\ \  |B_{(1+d_\ep)s_\ep}(x_\ep)\setminus B_{(1-d_\ep)s_\ep}(x_\ep)|=O(s_\ep^2 d_\ep),\\
		&\psi_\ep(x)-\frac{W}{2}x_1^2\ln\frac{1}{\ep}-\mu_\ep=O(d_\ep), \ \ U_{\ep, x_\ep, a_\ep}(x)- a_\ep\ln \frac{1}{\ep}=O(d_\ep), \quad \forall \    x\in B_{(1+d_\ep)s_\ep}(x_\ep)\setminus B_{(1-d_\ep)s_\ep}(x_\ep),
	\end{cases}
\end{equation*}
 which can be verified by easy computations.

Therefore, we conclude
	\begin{equation*}
		\begin{array}{ll}
			\varepsilon^{1-\frac{2}{q}} || R_\varepsilon (\phi_\varepsilon)||_{W^{-1,q}(B_{Ls_\ep}(x_\ep))}+\varepsilon^2||R_\varepsilon(\phi_\varepsilon)||_{L^\infty(B_{\frac{s_\ep}{2} }(x_\ep))}
			=O \left(\ep+o(1)\|\phi_\ep\|_{L^\infty}\right).
		\end{array}
	\end{equation*}
	Then from the above discussion and Lemma \ref{lem2-9}, we finally obtain
	\[	
	||\phi_\varepsilon||_{L^\infty(\mathbb{R}_+^2)}+\varepsilon^{1-\frac{2}{q}} ||\nabla \phi_\varepsilon||_{L^{q}(B_{Ls_\ep}(x_\ep))}=O \left(\ep+o(1)\|\phi_\ep\|_{L^\infty}\right),
	\]
	which implies the  desired estimate \eqref{2-30} for $\ep$ sufficiently small and completes the proof.
\end{proof}

\subsubsection{The second approximation}\label{sec2-2-2}
The estimate in Proposition \ref{prop2-10} is not  precise enough yet. However, it seems hard to improve Proposition \ref{prop2-10} directly any more, since the estimate of the function $\mathcal F$ (i.e. \eqref{B-3} ) can not be improved.

Fortunately, by careful analysis, we find that Proposition \ref{prop2-10} can actually be improved by symmetries. The key observation is that the function $\mathcal F(x)$ is odd in $x_1-x_{\ep,1}$ in the sense that $$\mathcal F((-x_1,x_2)+x_\ep)=-\mathcal F((x_1,x_2)+x_\ep).$$ So some linear terms may \emph{be eliminated   by such odd symmetry}.

To do this,  we introduce \emph{ the second approximation} to eliminate $\mathcal F$ by an existence theory of the linearized equation with odd nonlinearities, which is established in Appendix \ref{appD}. Taking $h$ in Theorem \ref{thmD-2} as $h_\ep(x):=pU_+^{p-1}(x)\mathcal F(s_\ep x+x_\ep)$, we can easily see that $h_\ep$ satisfies all the assumptions in Theorem \ref{thmD-2}  and $\|h_\ep\|_ {W^{-1,q}(B_4(0))}=O(\ep)$. By Theorem  \ref{thmD-2}, we obtain a solution $\tilde \phi_\ep^o$ to \eqref{D-1} with the constant $b_\ep$ given by \eqref{D-2}. Moreover, by Theorem \ref{thmD-2} and Lemma \ref{lemD-3}, the following estimates hold  $$\|\tilde \phi_\ep^o\|_1+ \|\nabla \tilde \phi_\ep^o\|_{L^q(B_2(0))} =O(\ep), \quad |b_\ep|=O(\ep^2|\ln\ep|).$$
 Let $\phi_\ep^o(x):=\tilde \phi_\ep^o\left(\frac{x-x_\ep}{s_\ep}\right)$. Then $\phi_\ep^o(x)$ satisfies
 \begin{equation}\label{2-33}
 	\begin{cases}
 		&-\Delta \phi_\ep^o=\frac{p x_{\ep,1}^2}{\ep^2}\left(U_{\ep, x_\ep, a_\ep}(x)- a_\ep\ln \frac{1}{\ep}\right)_+^{p-1} \left(\phi_\ep^o +\mathcal F-b_\ep \frac{\partial U}{\partial x_1}\left(\frac{x-x_\ep}{s_\ep}\right)\right), \   \   \   \text{in}\  \   \mathbb{R}^2\\
 		&\|  \phi_\ep^o\|_{L^\infty(\mathbb{R}^2)}+\varepsilon^{1-\frac{2}{q}}\|\nabla   \phi_\ep^o\|_{L^q(B_{2s_\ep}(x_\ep))} =O(\ep),\quad \forall\ q\in(2,+\infty], \\
 		&\phi_\ep^{o}((-x_1,x_2)+x_\ep)=-\phi_\ep^{o}((x_1,x_2)+x_\ep).
 	\end{cases}
 \end{equation}

 Let $\bar \phi_\ep:= \phi_\ep-\phi_\ep^o$ be the remaining term in the error function $\phi_\ep$, then $\bar \phi_\ep$ satisfies the following equation
 \begin{align}\label{2-34}
 		&\quad\mathbb L_\ep \bar \phi_\ep  =\frac{x_1}{\varepsilon^2}\left(\left(\psi_\ep(x)-\frac{W}{2}x_1^2\ln\frac{1}{\ep}-\mu_\ep\right)_+^p -\left(U_{\ep, x_\ep, a_\ep}(x)- a_\ep\ln \frac{1}{\ep}\right)_+^{p}\right.\\
 		&\  \  \left.-p\left(U_{\ep, x_\ep,a_\ep}(x)- a_\ep\ln \frac{1}{\ep}\right)_+^{p-1} (\phi_\ep^o +\bar \phi_\ep+\mathcal F)\right)\nonumber\\
 		+&\frac{p }{\varepsilon^2}\left(U_{\ep, x_\ep, a_\ep}(x)- a_\ep\ln \frac{1}{\ep}\right)_+^{p-1}\left[\left(x_1-\frac{x_{\ep,1}^2}{x_1}\right)(\phi_\ep^o +\mathcal F )+\frac{ b_\ep x_{\ep,1}^2}{x_1}    \frac{\partial U}{\partial x_1}\left(\frac{x-x_\ep}{s_\ep}\right) \right]-\frac{1}{x_1^2}\partial_{x_1} \phi_\ep^o.\nonumber
 \end{align}

 Next we show that $\bar \phi_\ep$ is indeed a higher order term.
 \begin{proposition}\label{prop2-11}
 	For $q\in (2,+\infty]$ and $\varepsilon\in (0,\varepsilon_0]$, it holds
 	\begin{equation}\label{2-35}
 		||\bar \phi_\ep ||_{L^\infty(\mathbb{R}_+^2)}+\varepsilon^{1-\frac{2}{q}} ||\nabla \bar \phi_\ep ||_{L^{q}(B_{Ls_\ep}(x_\ep))}=O \left(\ep^{2}|\ln\ep|\right).
 	\end{equation}
 \end{proposition}	
 \begin{proof}
 	In view of Lemma \ref{lem2-9},   we need to show that
 	\begin{equation}\label{2-36}
  \varepsilon^{1-\frac{2}{q}} ||\mathbb L_\ep \bar \phi_\ep||_{W^{-1,q}(B_{Ls_\ep}(x_\ep))}+\varepsilon^2||\mathbb L_\ep \bar \phi_\ep||_{L^\infty(B_{\frac{s_\ep}{2} }(x_\ep))}=O \left(\ep^{2}|\ln\ep|\right).
 	\end{equation}

 	 By \eqref{2-33} and \eqref{D-7}, we know that
 	 \begin{equation}\label{2-37}
 	 	|b_\ep|=O(\ep^2|\ln\ep|), \quad \|  \phi_\ep^o\|_{L^\infty(\mathbb{R}^2)}+\varepsilon\|\nabla   \phi_\ep^o\|_{L^\infty(B_{2s_\ep}(x_\ep))} =O(\ep).
 	 \end{equation}  For $x\in B_{\frac{s_\ep}{2}}$, by the expansion \eqref{2-31} we find
 	\begin{align*}
 			&\quad\ep^2\mathbb L_\ep \bar \phi_\ep  =  {x_1}\left(\left(\psi_\ep(x)-\frac{W}{2}x_1^2\ln\frac{1}{\ep}-\mu_\ep\right)_+^p -\left(U_{x_\ep,\ep, a}(x)- a\ln \frac{1}{\ep}\right)_+^{p}\right.\\
 			&\  \  \left.-p\left(U_{x_\ep,\ep, a}(x)- a\ln \frac{1}{\ep}\right)_+^{p-1} (\phi_\ep^o +\bar \phi_\ep+\mathcal F)\right)\\
 			+& {p } \left(U_{x_\ep,\ep, a}(x)- a\ln \frac{1}{\ep}\right)_+^{p-1}\left[\left(x_1-\frac{x_{\ep,1}^2}{x_1}\right)(\phi_\ep^o +\mathcal F )+\frac{ b_\ep x_{\ep,1}^2}{x_1}    \frac{\partial U}{\partial x_1}\left(\frac{x-x_\ep}{s_\ep}\right) \right]-\frac{\ep^2}{x_1^2}\partial_{x_1} \phi_\ep^o\\
 			=&O\left((\|\phi_\ep\|_{L^\infty}+\|\mathcal F\|_{L^\infty}+\ep^2 |\ln {\ep} |)^{2} +\ep (\|\phi_\ep^o\|_{L^\infty}+\|\mathcal F\|_{L^\infty})+|b_\ep|+\ep^2\right)
 			= O\left(\ep^{2}|\ln\ep|\right).
 		\end{align*}

 	For each $\varphi\in C_0^1(B_{Ls_\ep}(x_\ep))$, by the equation \eqref{2-34}, we have
 	\begin{align*}
 			&\quad\langle \mathbb L_\ep \bar \phi_\ep,\varphi \rangle\\ =&\int_{\mathbb{R}^2_+}\frac{x_1}{\varepsilon^2}\left(\left(\psi_\ep(x)-\frac{W}{2}x_1^2\ln\frac{1}{\ep}-\mu_\ep\right)_+^p -\left(U_{\ep, x_\ep, a_\ep}(x)- a_\ep\ln \frac{1}{\ep}\right)_+^{p}\right.\\
 			&\  \  \left.-p\left(U_{\ep, x_\ep, a_\ep}(x)- a_\ep\ln \frac{1}{\ep}\right)_+^{p-1} (\phi_\ep^o +\bar \phi_\ep+\mathcal F)\right)\varphi(x)dx\\
 			+&\int_{\mathbb{R}^2_+}\frac{p }{\varepsilon^2}\left(U_{\ep, x_\ep, a_\ep}(x)- a_\ep\ln \frac{1}{\ep}\right)_+^{p-1}\left[\left(x_1-\frac{x_{\ep,1}^2}{x_1}\right)(\phi_\ep^o +\mathcal F )+\frac{ b_\ep x_{\ep,1}^2}{x_1}    \frac{\partial U}{\partial x_1}\left(\frac{x-x_\ep}{s_\ep}\right) \right]\varphi(x)dx\\
 			&-\int_{\mathbb{R}^2_+} \frac{1}{x_1^2}\partial_{x_1} \phi_\ep^o \varphi(x)dx \
 			=: Y_1+Y_2.
 			\end{align*}

 		For the integral $Y_2$, using \eqref{2-38} , \eqref{2-37} and the fact $\left(U_{\ep, x_\ep, a_\ep}(x)- a_\ep\ln \frac{1}{\ep}\right)_+$ is bounded independent of $\ep$ and supported in $B_{s_\ep}(x_\ep)$, we conclude
 		\begin{align*}
 				Y_2=&\int_{\mathbb{R}^2_+}\frac{p }{\varepsilon^2}\left(U_{\ep, x_\ep, a_\ep}(x)- a_\ep\ln \frac{1}{\ep}\right)_+^{p-1}\left[\left(x_1-\frac{x_{\ep,1}^2}{x_1}\right)(\phi_\ep^o +\mathcal F)+\frac{ b_\ep x_{\ep,1}^2}{x_1}    \frac{\partial U}{\partial x_1}\left(\frac{x-x_\ep}{s_\ep}\right) \right]\varphi(x)dx\\
 				&-\int_{\mathbb{R}^2_+} \frac{1}{x_1^2}\partial_{x_1} \phi_\ep^o \varphi(x)dx\\
 				=&O\left( \ep^{-1} (\|\phi_\ep^o\|_{L^\infty(B_{Ls_\ep}(x_\ep))}+\|\mathcal F\|_{L^\infty}(B_{s_\ep}(x_\ep)))+\ep^{-2}|b_\ep|+1\right)\int_{\mathbb{R}^2_+}|\varphi(x)|dx\\
 				=&O\left(\ep^{2}|\ln\ep|\right)\ep^{-1+\frac{2}{q}}\|\nabla\varphi\|_{L^{q'}(B_{Ls_\ep}(x_\ep))}.
 		     \end{align*}
 	Recall the definition of $d_\ep$ in \eqref{2-32}. We infer from \eqref{2-30} and \eqref{B-6} that $d_\ep=O(\ep)$. Using the expansion \eqref{2-31}, the inequality  \eqref{2-38}, \eqref{2-37} and   the odd symmetry,  similar calculation  as in the proof of Proposition \ref{prop2-10} gives
 		\begin{align*}
 			Y_1=&\int_{\mathbb{R}^2_+}\frac{x_1}{\varepsilon^2}\left(\left( U_{x_\ep, \ep, a_{\ep}}(x)-a_\ep\ln\frac{1}{\ep} + \mathcal F(x)+\phi_\ep^o+\bar \phi_\ep+O(\ep^2 |\ln {\ep} |)\right)_+^p  -\left(U_{\ep, x_\ep, a_\ep}(x)- a_\ep\ln \frac{1}{\ep}\right)_+^{p}\right.\\
 			&\  \  \left.-p\left(U_{\ep, x_\ep, a_\ep}(x)- a_\ep\ln \frac{1}{\ep}\right)_+^{p-1} (\phi_\ep^o +\bar \phi_\ep+\mathcal F)\right)\varphi(x)dx\\
 			=&\left(\int_{B_{(1-d_\ep)s_\ep}(x_\ep)}+\int_{B_{(1+d_\ep)s_\ep}(x_\ep)\setminus B_{(1-d_\ep)s_\ep}(x_\ep)}\right)\frac{x_1}{\varepsilon^2}\left(\left( U_{\ep, x_\ep, a_{\ep}}(x)-a_\ep\ln\frac{1}{\ep} + \mathcal F(x)+\phi_\ep^o+\bar \phi_\ep+O(\ep^2 |\ln {\ep} |)\right)_+^p\right.\\
 			&\left.-\left(U_{\ep, x_\ep,  a_\ep}(x)- a_\ep\ln \frac{1}{\ep}\right)_+^{p}-p\left(U_{\ep, x_\ep, a_\ep}(x)- a_\ep\ln \frac{1}{\ep}\right)_+^{p-1} (\phi_\ep^o +\bar \phi_\ep+\mathcal F)\right)\varphi(x)dx\\
 			=& O(\ep^{-2}(\|\phi_\ep\|_{L^\infty}+\|\mathcal F\|_{L^\infty}+\ep^2 |\ln {\ep} |)^{2})\int_{B_{(1-d_\ep)s_\ep}(x_\ep)} |\varphi(x)|dx\\
 			&+O\left( \ep^{-2}(\|\phi_\ep\|_{L^\infty}+\|\mathcal F\|_{L^\infty} )d_\ep^{p-1}\right)\int_{B_{(1+d_\ep)s_\ep}(x_\ep)\setminus B_{(1-d_\ep)s_\ep}(x_\ep)} |\varphi(x)|dx\\
 			= &O\left((\|\phi_\ep\|_{L^\infty}+\|\mathcal F\|_{L^\infty}+\ep^2 |\ln {\ep} |)^{2}+(\|\phi_\ep\|_{L^\infty} +\|\mathcal F\|_{L^\infty}) d_\ep^{p-1+\frac{1}{q}} \right)   \ep^{-1+\frac{2}{q}}\|\nabla\varphi\|_{L^{q'}(B_{Ls_\ep}(x_\ep))}\\
 			=&O\left(\ep^{2}|\ln\ep|\right)\ep^{-1+\frac{2}{q}}\|\nabla\varphi\|_{L^{q'}(B_{Ls_\ep}(x_\ep))}.			
 	\end{align*}
 	Therefore, we arrive at \eqref{2-36} and complete the proof by applying Lemma \ref{lem2-9}.
 \end{proof}

 \subsubsection{Refined estimates}\label{sec2-2-3}
 The estimate  in Proposition \ref{prop2-11} is sufficiently accurate for us to give a refined Kelvin-Hicks formula of the location $x_\ep$ via a local Pohozaev identity.

  As a result of \eqref{2-35}, \eqref{D-6} and Lemma \ref{lemC-4}, we obtain the following identity for the location $x_\ep=(x_{\ep,1},0)$ and the radius $s_\ep$.
\begin{corollary}\label{cor2-12}
	It holds
	\begin{equation}\label{2-39}
		\begin{split}
			 Wx_{\ep,1}\ln\frac{1}{\varepsilon}-\frac{\kappa}{4\pi}\left(\ln\frac{8x_{\ep,1}}{s_\ep}+\frac{p-1 }{4}\right)= O\left( \ep^2|\ln\ep| \right),
		\end{split}
	\end{equation}
and
\begin{equation}\label{2-40}
	\begin{split}
		 \Lambda_p x_{\ep,1}^{-\frac{p+1}{p-1}}\left(\frac{\ep}{s_\ep}\right)^{\frac{2}{p-1}}=\kappa+O\left( \ep^2|\ln\ep| \right).
	\end{split}
\end{equation}
\end{corollary}

For $\ep$ sufficiently small, let $(r_\ep^*, s_\ep^*) $ be the unique solution to the following system of equations of $(r,s)$ in a neighborhood of $\left(\kappa/(4\pi W), \Lambda_p^{\frac{p-1}{2}}\kappa^{-p}(4\pi W)^{\frac{p+1}{2}}\ep\right)$.
\begin{equation}\label{2-41}
	\begin{cases}
		&Wr\ln\frac{1}{\varepsilon}-\frac{\kappa}{4\pi}\left(\ln\frac{8r}{s }+\frac{p-1 }{4}\right)=0,\\
		&  \Lambda_p r^{-\frac{p+1}{p-1}}\left(\frac{\ep}{s }\right)^{\frac{2}{p-1}}= \kappa.
	\end{cases}
\end{equation}
Existence and uniqueness of solution to \eqref{2-41} in a small neighborhood of  $\left(\kappa/(4\pi W), \Lambda_p^{\frac{p-1}{2}}\kappa^{-p}(4\pi W)^{\frac{p+1}{2}}\ep\right)$ can be verified easily by the implicit function theorem for $\ep$ small, so we omit the details.

With Proposition \ref{prop2-11} and Corollary \ref{cor2-12}, we have the following fine estimates of the parameters used above, which are essential in the proof of uniqueness later.
 \begin{proposition}\label{prop2-13}
 	For each $\varepsilon\in(0,\varepsilon_0]$ with $\varepsilon_0>0$ sufficiently small, let $(r_\ep^*, s_\ep^*) $ be the unique solution to   system \eqref{2-41} in a small neighborhood of \\$\left(\kappa/(4\pi W), \Lambda_p^{\frac{p-1}{2}}\kappa^{-p}(4\pi W)^{\frac{p+1}{2}}\ep\right)$. Define $$a_\ep^*:=-\frac{ \ln s_\ep^*  }{  \ln \ep } {r_\ep^*}^{-\frac{2}{p-1}}\left(\frac{\ep}{s_\ep^* }\right)^{\frac{2}{p-1}}U'(1),$$
 	and
 	$$\mu_\ep^*:=a_\ep^* \ln \frac{1}{\ep}-\frac{ W}{2}{r_\ep^*}^2\ln \frac{1}{\ep}+\frac{{r_\ep^*}^{-\frac{2}{p-1}}}{2\pi}\left(\frac{\ep}{s_\ep^* }\right)^{\frac{2}{p-1}}\Lambda_p  \left( \ln(8r_\ep^*) -2  \right).$$
 	Then the following asymptotic behaviors hold as $\ep\to0$,
 	\begin{equation*}
 		|x_{\ep,1}-r_\ep^*|=O(\varepsilon^2),\ \	|s_\ep-s_\ep^*|=O(\varepsilon^3|\ln\ep|),\ \		|a_\ep-a_\ep^*|=O(\varepsilon^2|\ln\ep|),\  \  |\mu_\ep-\mu_\ep^*|=O(\varepsilon^2|\ln\ep|^2).
 	\end{equation*}
 \end{proposition}
 \begin{proof}
  To show $|x_{\ep,1}-r_\ep^*|=O(\varepsilon^2)$, one just needs to take the difference of \eqref{2-39} and the first equation in \eqref{2-41} and apply the  implicit function theorem. Then the rest three estimates follow  from $|x_{\ep,1}-r_\ep^*|=O(\varepsilon^2)$, \eqref{2-21}, \eqref{2-40}, \eqref{2-41} and the definitions of $s_\ep^* $, $a_\ep^*$ and $\mu_\ep^*$ by direct calculations. We omit the details and complete the proof.
 \end{proof}
\begin{remark}
	It can be seen that parameters $r_\ep^*, s_\ep^*, a_\ep^*, s_\ep^*$ may depend on $\ep$, but do not depend on the vortex ring $\zeta_\ep$ directly. This fact will be used later in our proof.
\end{remark}

\subsection{The proof of uniqueness}\label{sec2-3}
	With the estimates on asymptotic behaviors in the previous subsection, we are in position to prove the uniqueness (i.e. Theorem \ref{thm2-1}). We argue by way of contradiction.	Suppose on the contrary that there are two different functions $\psi_\varepsilon^{(1)}$ and $\psi_\varepsilon^{(2)}$ that satisfy \eqref{2-2}--\eqref{2-5} and $\psi_\varepsilon^{(1)}$ and $\psi_\varepsilon^{(2)}$  are even symmetric with respect to the $x_1$-axis.
	
	Define the normalized difference between  $\psi_\varepsilon^{(1)}$ and $\psi_\varepsilon^{(2)}$ by
\begin{equation*}
	\xi_\varepsilon( x):=\frac{\psi_\varepsilon^{(1)}(  x)-\psi_\varepsilon^{(2)}(  x)}{||\psi_\varepsilon^{(1)}-\psi_\varepsilon^{(2)}||_{L^\infty(\mathbb R^2_+)}}.
\end{equation*}
Then  $\xi_\varepsilon$ satisfies $||\xi_\varepsilon||_{L^\infty(\mathbb R^2_+)}=1$ and
\begin{equation*}
	\begin{cases}
		-\text{div}\left(\frac{1}{x_1}\nabla \xi_\varepsilon\right) =f_\varepsilon(  x), & \text{in} \ \mathbb R^2_+,
		\\
		\xi_\varepsilon=0, & \text{on} \ \partial \mathbb R^2_+,
		\\
		\xi_\varepsilon   \to 0, &\text{as} \ |  x |\to +\infty,
	\end{cases}
\end{equation*}
where
\begin{equation*}
	f_\varepsilon( x)=\frac{x_1\left(\left(\psi_\ep^{(1)}(x)-\frac{W}{2}x_1^2\ln\frac{1}{\ep}-\mu_\ep^{(1)}\right)_+^p-\left(\psi_\ep^{(2)}(x)-\frac{W}{2}x_1^2\ln\frac{1}{\ep}-\mu_\ep^{(2)}\right)_+^p \right) }{\varepsilon^2||\psi_\varepsilon^{(1)}-\psi_\varepsilon^{(2)}||_{L^\infty(\mathbb R^2_+)}}.
\end{equation*}

In the following, we are to obtain a number of estimates for $\xi_\varepsilon$ and $f_\varepsilon$. Then we will derive a contradiction by local Pohozaev identity whenever $\psi_\varepsilon^{(1)} \not\equiv \psi_\varepsilon^{(2)}$.

Since the differences between the parameters $x_\ep^{(1)}$ and $x_\ep^{(2)}$, $a_\ep^{(1)}$ and $a_\ep^{(2)}$, $s_\ep^{(1)}$ and $s_\ep^{(2)}$ as well as $\mu_\ep^{(1)}$ and $\mu_\ep^{(2)}$ are sufficiently small by Proposition \ref{prop2-13}, we can approximate both $\psi_\varepsilon^{(1)}$ and $ \psi_\varepsilon^{(2)}$ by the same approximate solution.
\begin{lemma}\label{lem2-14}
	The following asymptotic behaviors hold as $\ep\to0$
		\begin{equation*}
		\psi_\ep^{(1)}(x)-\frac{W}{2}x_1^2\ln\frac{1}{\ep}-\mu_\ep^{(1)}= U_{\ep, x_\ep^{(1)},  a_{\ep}^{(1)}}(x)-a_\ep^{(1)}\ln\frac{1}{\ep} + \mathcal F^{(1)}(x)+\phi_\ep^{o,(1)}+O(\ep^2 |\ln {\ep}| ),
	\end{equation*}
\begin{equation*}
\psi_\ep^{(2)}(x)-\frac{W}{2}x_1^2\ln\frac{1}{\ep}-\mu_\ep^{(2)}= U_{\ep, x_\ep^{(1)},  a_{\ep}^{(1)}}(x)-a_\ep^{(1)}\ln\frac{1}{\ep} + \mathcal F^{(1)}(x)+\phi_\ep^{o,(1)}+O(\ep^2 |\ln {\ep}| ),
\end{equation*}
where $\mathcal F^{(1)}$ is defined by \eqref{B-1} with $x_\ep$, $a_\ep$ and $\mu_\ep$ replaced by $x_\ep^{(1)}$, $a_\ep^{(1)}$ and $\mu_\ep^{(1)}$ and $\phi_\ep^{o,(1)}$ is the second approximation for $\psi_\ep^{(1)}$  obtained in subsubsection \ref{sec2-2-2}.
\end{lemma}
\begin{proof}
	The first equation is a consequence of  \eqref{2-23}, \eqref{B-2} and Proposition \ref{prop2-11}. While the second equation follows by its own estimate
		\begin{equation*}
		\psi_\ep^{(2)}(x)-\frac{W}{2}x_1^2\ln\frac{1}{\ep}-\mu_\ep^{(2)}= U_{\ep, x_\ep^{(2)},  a_{\ep}^{(2)}}(x)-a_\ep^{(2)}\ln\frac{1}{\ep} + \mathcal F^{(2)}(x)+\phi_\ep^{o,(2)}+O(\ep^2 |\ln {\ep}| )
	\end{equation*}
and Proposition \ref{prop2-13}. So we complete the proof of this lemma.
\end{proof}

  To simplify notation, we always abbreviate the parameters $(s_\ep^{(1)},  x_\ep^{(1)}, a_\ep^{(1)})$ as $(s_\ep, x_\ep, a_\ep)$. We also denote the norm $\|\cdot\|_{L^\infty(\mathbb R^2_+)}$ as $|\cdot|_\infty$ throughout this subsection. We have the following estimates for the asymptotic behaviors of the nonlinearity $f_\varepsilon$.
\begin{lemma}\label{lem2-15}
	For $q\in (2,\infty]$ and any large $L>0$, it holds
	$$||s_\ep^2 f_\varepsilon(s_\ep y+x_\ep)||_{W^{-1,q}(B_L(  0))}=O_\varepsilon(1).$$
	Moreover, as $\varepsilon\to 0$, for all $ \varphi\in \ C_0^\infty(\mathbb{R}^2)$ it holds
	\begin{equation*}
		\int_{\mathbb{R}^2} s_\ep^2 f_\varepsilon(s_\ep y+x_\ep)\varphi(y) d y=  p\int_{\mathbb{R}^2 } \left(U (y)\right)_+^{p-1} \left(c_\varepsilon \frac{\partial U}{\partial x_1}(y)+\textcolor{black}{O(\varepsilon )}\right)  \varphi(y)dy,
	\end{equation*}
	where $c_\varepsilon$ is bounded by a constant independent of $\varepsilon$.
\end{lemma}
\begin{proof}
	
	The circulation constraint \eqref{2-3} implies
	\begin{equation*}
		\begin{split}
			0=&\int_{\mathbb R^2_+} x_1\left(\left(\psi_\ep^{(1)}(x)-\frac{W}{2}x_1^2\ln\frac{1}{\ep}-\mu_\ep^{(1)}\right)_+^p-\left(\psi_\ep^{(2)}(x)-\frac{W}{2}x_1^2\ln\frac{1}{\ep}-\mu_\ep^{(2)}\right)_+^p \right)dx\\
			=&p\int_{\mathbb R^2_+} x_1 \left(U_{\ep,x_\ep, a}(x)- a_\ep\ln \frac{1}{\ep}+O(\ep)\right)_+^{p-1}\left(\psi_\ep^{(1)}(x)-\psi_\ep^{(2)}(x)-(\mu_\ep^{(1)}-\mu_\ep^{(2)})\right)  dx,
		\end{split}
	\end{equation*}
	which implies 	
	\begin{equation}\label{2-42}
		\frac{|\mu_\varepsilon^{(1)}-\mu_\varepsilon^{(2)}|}{|\psi_\varepsilon^{(1)}-\psi_\varepsilon^{(2)}|_\infty}=O_\varepsilon(1).
	\end{equation}

	For any $\varphi \in C_0^\infty(B_{L }(0))$ and $q'\in [1,2)$, by Lemma \ref{lem2-14}, \eqref{2-42} and the Sobolev embedding, we have
	\begin{align}\label{2-43}
			&\quad \int_{\mathbb{R}^2}  s_\ep^2 f_\varepsilon(s_\ep y+x_\ep) \varphi(y) d y\nonumber\\
			&=\frac{s_\ep^2}{\varepsilon^2|\psi_\varepsilon^{(1)}-\psi_\varepsilon^{(2)}|_\infty} \int_{\mathbb R^2_+} (s_\ep y_1+x_{\ep,1})\left(\left(\psi_\ep^{(1)}(s_\ep y+x_\ep)-\frac{W}{2}(s_\ep y_1+x_{\ep,1})^2\ln\frac{1}{\ep}-\mu_\ep^{(1)}\right)_+^p\right.\nonumber\\
			&\quad\left.-\left(\psi_\ep^{(2)}(s_\ep y+x_\ep)-\frac{W}{2}(s_\ep y_1+x_{\ep,1})^2\ln\frac{1}{\ep}-\mu_\ep^{(2)}\right)_+^p \right) \varphi(y)dy\nonumber\\
			&=p\int_{\mathbb R^2 } (s_\ep y_1+x_{\ep,1})x_{\ep,1}^{-2}\left(U +O(\ep)\right)_+^{p-1}\frac{\left(\psi_\ep^{(1)}(s_\ep y+x_\ep)-\psi_\ep^{(2)}(s_\ep y+x_\ep)-(\mu_\ep^{(1)}-\mu_\ep^{(2)})\right)}{|\psi_\varepsilon^{(1)}-\psi_\varepsilon^{(2)}|_\infty} \varphi(y)dy\nonumber\\
			&=O_\ep(1) \int_{\mathbb R^2 } |\varphi(y)|dy =O_\varepsilon(||\varphi||_{W^{1,q'}(B_L(  0))}).
		\end{align}
	So for $q\in(2,+\infty]$, we obtain
	$$||s_\ep^2 f_\varepsilon(s_\ep y+x_\ep)||_{W^{-1,q}(B_L(  0))}=O_\varepsilon(1).$$

To obtain the second equation in Lemma \ref{lem2-15},	we   define the normalized difference of $\psi_\varepsilon^{(i)}-\mu_\varepsilon^{(i)}$ as
	\begin{equation*}
		\xi_{\varepsilon,\mu}(y):=\frac{\left(\psi_\varepsilon^{(1)}(s_\ep y+x_\ep)-\mu_\varepsilon^{(1)}\right)-\left(\psi_\varepsilon^{(2)}(s_\ep y+x_\ep)-\mu_\varepsilon^{(2)}\right)}{|\psi_\varepsilon^{(1)}-\psi_\varepsilon^{(2)}|_\infty}.
	\end{equation*}
	Then $|\xi_{\varepsilon,\mu}|_\infty=O_\ep(1)$ by \eqref{2-42} and $\xi_{\varepsilon,\mu}$ satisfies
	$$-\text{div}\left(\frac{1}{s_\ep y_1+x_{\ep,1}}\nabla \xi_{\varepsilon,\mu}(y)\right) =s_\ep^2f_\varepsilon(s_\ep y+x_\ep).$$
	
	By standard theory of regularity  for elliptic equations, $\xi_{\varepsilon,\mu}$ is bounded in $W^{1,q}_{\text{loc}}(\mathbb{R}^2)$ for $q\in(2,+\infty)$ and hence in  $C^\alpha_{\text{loc}}(\mathbb{R}^2)$. We let
	\begin{equation}\label{2-44}
		\xi_{\varepsilon,\mu}^*(y):=\xi_{\varepsilon,\mu}(y)-c_\varepsilon\frac{\partial U}{\partial x_1}(y)
	\end{equation}
	with the constant
	\begin{equation*}
		c_\varepsilon=\left(\int_{B_L(  0)} \xi_{\varepsilon,\mu}(y)(U(y))_+^{p-1} \frac{\partial U}{\partial x_1}(y)d y\right)\left(\int_{B_L(  0)}(U(y))_+^{p-1} \left(\frac{\partial U}{\partial x_1}(y)\right)^2 d  y\right)^{-1}
	\end{equation*}
	as the projection coefficient, which is bounded independent of $\varepsilon$. Then for any $\varphi \in C_0^\infty(B_{L }(0))$,  $\xi_{\varepsilon,\mu}^*$ satisfies
	\begin{align}\label{2-45}
			&\quad\int_{\mathbb{R}^2} \frac{1}{s_\ep y_1+x_{\ep,1}} \nabla \xi_{\varepsilon,\mu}^*(y)\cdot\nabla \varphi(y) d  y-\frac{p}{x_{\ep,1}}\int_{\mathbb{R}^2}U_+^{p-1}(y)\xi_{\varepsilon,\mu}^*(y) \varphi(y)dy\\
			&=-c_\varepsilon\left(\int_{B_L(  0)}  \frac{1}{s_\ep y_1+x_{\ep,1}} \nabla\left(\frac{\partial U}{\partial x_1}\right)(y)\cdot\nabla \varphi(y) d  y-\frac{p}{x_{\ep,1}}\int_{\mathbb{R}^2}U_+^{p-1}(y)\frac{\partial U}{\partial x_1}(y)\varphi(y)dy\right)\nonumber\\
			&\quad+\left(\int_{\mathbb{R}^2} s_\ep^2 f_\varepsilon(s_\ep y+x_\ep) \varphi(y) d  y-\frac{p}{x_{\ep,1}}\int_{\mathbb{R}^2}U_+^{p-1}(y)\xi_{\varepsilon,\mu}(y) \varphi(y)dy\right) =:N_1+N_2.\nonumber
		\end{align}
	Since $\text{ker}\{-\Delta-pU_+^{p-1}\}=\text{span}\left\{ \frac{\partial U}{\partial x_1}, \frac{\partial U}{\partial x_2}   \right\}$ by Lemma \ref{lem2-8},
	we deduce $$N_1=c_\varepsilon \int_{B_L(  0)}  \frac{s_\ep y_1}{x_{\ep,1}(s_\ep y_1+x_{\ep,1})} \nabla\left(\frac{\partial U}{\partial x_1}\right)(y)\cdot\nabla \varphi(y) d  y= O(\varepsilon) \| \varphi\|_{W^{1,q'}(B_L(  0))}.$$ For the term $N_2$, using the expansion of the function $t^{p-1}$ ($p\geq 2$), calculations in \eqref{2-43} and the boundedness of $\xi_{\varepsilon,\mu}$, we have
	\begin{align*}
			N_2&=p\int_{\mathbb R^2 } (s_\ep y_1+x_{\ep,1})x_{\ep,1}^{-2}\left(U +O(\ep)\right)_+^{p-1}\xi_{\varepsilon,\mu}(y) \varphi(y)dy -\frac{p}{x_{\ep,1}}\int_{\mathbb{R}^2}U_+^{p-1}(y)\xi_{\varepsilon,\mu}(y) \varphi(y)dy\\
			&=\textcolor{black}{O(\varepsilon )}  \| \varphi\|_{W^{1,q'}(B_L(  0))}.
		\end{align*}
	Since $\xi_{\varepsilon,\mu}^*$ satisfies the orthogonality condition by projection \eqref{2-44}, by virtue of an argument similar to the proof of Lemma \ref{lemD-1}, we can deduce from the estimates for $N_1, N_2$ that
	\begin{equation*}
		\|	\xi_{\varepsilon,\mu}^*\|_{L^\infty(B_L(  0))}+\|\nabla	\xi_{\varepsilon,\mu}^*\|_{L^q(B_{L}( 0))}=\textcolor{black}{O(\varepsilon )}, \quad \forall \, q\in(2,+\infty].
	\end{equation*}
	
	Now we arrive at a conclusion: by the definition of $	\xi_{\varepsilon,\mu}^*$ in \eqref{2-44}, for each $q\in(2,+\infty]$, it holds
	\begin{equation*}
		\xi_{\varepsilon,\mu}=c_\varepsilon\frac{\partial U}{\partial x_1}+\textcolor{black}{O(\varepsilon )}, \quad \text{in} \ W^{1,q}(B_{L}(  0)),
	\end{equation*}
	and for all $ \varphi\in \ C_0^\infty(\mathbb{R}^2)$, it holds
	\begin{equation*}
		\int_{\mathbb{R}^2} s^2_\ep f_\varepsilon(s_\ep y+x_\ep) \varphi(y) dy=  p\int_{\mathbb{R}^2 } \left(U (y)\right)_+^{p-1} \left(c_\varepsilon \frac{\partial U}{\partial x_1}(y)+\textcolor{black}{O(\varepsilon )}\right)  \varphi(y)dy,
	\end{equation*}
	where $c_\varepsilon$ is bounded independent of $\varepsilon$. So we have completed the proof of Lemma \ref{lem2-15}.
\end{proof}

To apply the local Pohozaev identity to obtain a contradiction, set
\begin{equation*}
	\xi_{1,\varepsilon}( x):=\frac{\psi_{1,\varepsilon}^{(1)}(  x)-\psi_{1,\varepsilon}^{(2)}(  x)}{|\psi_\varepsilon^{(1)}-\psi_\varepsilon^{(2)}|_\infty}.
\end{equation*}
Then $\xi_{1,\varepsilon}$ has the following integral representation
\begin{equation}\label{3-16}
	\xi_{1,\varepsilon}=x_{\ep,1}^2\int_{\mathbb R^2_+} G( x, y)\cdot y_1^{-1}f_\varepsilon(y)dy.
\end{equation}
By the asymptotic estimate for $f_\varepsilon $ in Lemma \ref{lem2-15}, one can check that
\begin{equation*}
	\frac{\psi_{2,\varepsilon}^{(1)}(  x)-\psi_{2,\varepsilon}^{(2)}(  x)}{|\psi_\varepsilon^{(1)}-\psi_\varepsilon^{(2)}|_\infty}=\int_{\mathbb R^2_+} H( x, y)y_1^{-1}f_\varepsilon(y)dy=o_\varepsilon(1).
\end{equation*}
So we see that $\xi_{1,\varepsilon}$ is indeed the main part of $\xi_\varepsilon$  and
\begin{equation}\label{xi}
	||\xi_{1, \varepsilon}||_{L^\infty(\mathbb R^2_+)}=1-o_\varepsilon(1).
\end{equation}  In order to derive a contradiction, we will show that $||\xi_{1,\varepsilon}||_{L^\infty(\mathbb R^2_+)}=o_\varepsilon(1)$ by the local Pohozaev identity.

Applying  \eqref{2-27} to $\psi_{1,\varepsilon}^{(1)}$ and $\psi_{1,\varepsilon}^{(2)}$ separately and calculating their difference, we  obtain the following local Pohozaev identity:
\begin{align}\label{2-47}
		&\quad-\int_{\partial B_\delta(x_\ep)}\frac{\partial\xi_{1,\varepsilon}}{\partial\nu}\frac{\partial\psi_{1,\varepsilon}^{(1)}}{\partial x_1}ds-\int_{\partial B_\delta(x_\ep)}\frac{\partial \psi_{1,\varepsilon}^{(2)}}{\partial\nu }\frac{\partial\xi_{1,\varepsilon}}{\partial x_1}ds+\frac{1}{2}\int_{\partial B_\delta(x_\ep)}\langle\nabla(\psi_{1,\varepsilon}^{(1)}+\psi_{1,\varepsilon}^{(2)}),\nabla\xi_{1,\varepsilon}\rangle\nu_1ds\nonumber\\
		&=-\frac{x_{\ep,1}^2}{\varepsilon^2|\psi_\varepsilon^{(1)}-\psi_\varepsilon^{(2)}|_\infty}\int_{B_\delta(x_\ep)} \left(\left(\psi_\ep^{(1)}(x)-\frac{W}{2}x_1^2\ln\frac{1}{\ep}-\mu^{(1)}_\ep\right)_+^p\partial_{x_1}\psi_{2,\varepsilon}^{(1)}  \right.\\
		&\quad\left.-\left(\psi_\ep^{(2)}(x)-\frac{W}{2}x_1^2\ln\frac{1}{\ep}-\mu^{(2)}_\ep\right)_+^p\partial_{x_1}\psi_{2,\varepsilon}^{(2)} \right)d  x.\nonumber
\end{align}

The proof of the uniqueness of a vortex ring with small cross-section is based on a careful estimate for each term in \eqref{2-47}. To deal with terms involving integral on $\partial B_\delta(x_\ep)$ in the local Pohozaev identity, we need the following lemma concerning the behavior of $\xi_{1,\varepsilon}$ away from $x_\ep$.
\begin{lemma}\label{lem2-16}
	For any fixed small $\delta>0$, it holds
	\begin{equation}\label{2-48}
		\xi_{1,\varepsilon}(  x)=  B_\varepsilon  \frac{s_\ep x_{\ep,1}^2}{2\pi}\frac{x_1-x_{\ep,1}}{|  x- x_\ep|^2}+ B_\varepsilon  \frac{s_\ep x_{\ep,1}^2}{2\pi}\frac{x_1+x_{\ep,1}}{|  x-  {\bar x_\ep}|^2}+  B_\varepsilon  \frac{s_\ep x_{\ep,1}}{2\pi} \ln \frac{| x- {\bar x_\ep}|}{|  x-  x_\ep|}+O(\varepsilon^2)
	\end{equation}
	in $C^1(\mathbb R^2_+\setminus B_{\delta/2}(x_\ep))$, where
	\begin{equation*}
		B_\varepsilon:=\frac{1}{s_\ep}\int_{B_{2s_\ep}(x_\ep)} (x_1-x_{\ep,1}) x_1^{-1}f_\varepsilon(  x) d  x
	\end{equation*}
	is uniformly bounded independent of $\varepsilon$.
\end{lemma}
\begin{proof}
	Using the expansion \eqref{C-5}, for $ x\in \mathbb{R}^2_+\setminus B_{\delta/2}(x_\ep)$ we have
	\begin{align*}
			\xi_{1,\varepsilon}( x)&=\frac{x_{\ep,1}^2}{2\pi}\int_{\mathbb{R}^2_+} y_1^{-1}\ln \left(\frac{| x-  {\bar y}|}{|  x- y|} \right)f_\varepsilon(y)dy=\frac{x_{\ep,1}^2}{2\pi}\int_{B_{2s_\ep}(x_\ep)} y_1^{-1}\ln \left(\frac{| x-  {\bar y}|}{|  x- y|} \right)f_\varepsilon(y)dy\\
			&=\frac{x_{\ep,1}}{2\pi}\ln \frac{1}{|  x-x_\ep|} \int_{B_{2s_\ep}(x_\ep)}  f_\varepsilon(y)dy+ \frac{x_{\ep,1}^2}{2\pi}\int_{B_{2s_\ep}(x_\ep)}y_1^{-1}\ln \left(\frac{|  x-x_\ep|}{|  x-y|}\right) f_\varepsilon(y)dy\\
			&\quad-\frac{x_{\ep,1}}{2\pi}\ln \frac{1}{| x- {\bar x_\ep}|} \int_{B_{2s_\ep}(x_\ep)}  f_\varepsilon(y)dy-\frac{x_{\ep,1}^2}{2\pi}\int_{B_{2s_\ep}(x_\ep)}y_1^{-1}\ln \left(\frac{| x-  {\bar x_\ep}|}{|  x-  {\bar y}|} \right)f_\varepsilon(y)dy\\
			&\quad-\frac{x_{\ep,1}}{2\pi} \ln \frac{| x- {\bar x_\ep}|}{|  x-x_\ep|} \int_{B_{2s_\ep}(x_\ep)}(y_1-x_{\ep,1}) y_1^{-1}f_\varepsilon(y)d  y\\
			&=-\frac{x_{\ep,1}^2}{4\pi}\int_{B_{2s_\ep}(x_\ep)}y_1^{-1}\ln\left(1+ \frac{2(  x-x_\ep)\cdot (x_\ep-y)}{| x-x_\ep|^2} +\frac{ |x_\ep-y|^2}{|  x-x_\ep|^2}\right)f_\varepsilon(y)dy\\
			&\quad +\frac{x_{\ep,1}^2}{4\pi}\int_{B_{2s_\ep}(x_\ep)}y_1^{-1}\ln\left(1+ \frac{2( x-  {\bar x_\ep})\cdot (  {\bar x_\ep}-  {\bar y})}{|  x-  {\bar x_{\ep}}|^2} +\frac{ |  {\bar x_\ep}- {\bar y} |^2}{|  x-  \bar x_\ep|^2}\right)f_\varepsilon(y)dy\\
			&\quad-\frac{x_{\ep,1}}{2\pi} \ln \frac{| x- {\bar x_\ep}|}{|  x-x_\ep|} \int_{B_{2s_\ep}(x_\ep)}(y_1-x_{\ep,1}) y_1^{-1}f_\varepsilon(y)d  y\\
			&=B_\varepsilon  \frac{s_\ep x_{\ep,1}^2}{2\pi}\frac{x_1-x_{\ep,1}}{|  x- x_\ep|^2}+ B_\varepsilon  \frac{s_\ep x_{\ep,1}^2}{2\pi}\frac{x_1+x_{\ep,1}}{|  x-  {\bar x_\ep}|^2}+  B_\varepsilon  \frac{s_\ep x_{\ep,1}}{2\pi} \ln \frac{| x- {\bar x_\ep}|}{|  x-  x_\ep|}+O(\varepsilon^2).
		\end{align*}
	Moreover, $  B_\varepsilon$ is bounded independent of $\varepsilon$ since $||s^2 f_\varepsilon(s_\ep   y+x_\ep)||_{W^{-1,q}(B_L(\boldsymbol 0))}=O_\varepsilon(1)$ for $q\in (2,\infty]$. Then we can verify \eqref{2-48} in $C^1(\mathbb R^2_+\setminus B_{\delta/2}(x_\ep))$ by a similar argument. So we finish the proof of this lemma.
\end{proof}

Using the asymptotic estimate for $\psi_{1,\varepsilon}$ in Lemma \ref{lemC-1} and $\xi_\varepsilon$ in Lemma \ref{lem2-16}, we deduce by direct computations that
\begin{align}\label{2-49}
		&\quad-\int_{\partial B_\delta(x_\ep)}\frac{\partial\xi_{1,\varepsilon}}{\partial\nu}\frac{\partial\psi_{1,\varepsilon}^{(1)}}{\partial x_1}ds-\int_{\partial B_\delta(x_\ep)}\frac{\partial \psi_{1,\varepsilon}^{(2)}}{\partial\nu }\frac{\partial\xi_{1,\varepsilon}}{\partial x_1}ds+\frac{1}{2}\int_{\partial B_\delta(x_\ep)}\langle\nabla(\psi_{1,\varepsilon}^{(1)}+\psi_{1,\varepsilon}^{(2)}),\nabla\xi_{1,\varepsilon}\rangle\nu_1ds\nonumber\\
		&=O(\varepsilon)  B_\varepsilon+O(\varepsilon^2)=O(\ep).
\end{align}

For the right hand side of \eqref{2-47}, by careful analysis, we have the following identity.
\begin{lemma}\label{lem2-17}
	It holds
	\begin{align*}
			&  -\frac{x_{\ep,1}^2}{\varepsilon^2|\psi_\varepsilon^{(1)}-\psi_\varepsilon^{(2)}|_\infty}\int_{B_\delta(x_\ep)} \left(\left(\psi_\ep^{(1)}(x)-\frac{W}{2}x_1^2\ln\frac{1}{\ep}-\mu^{(1)}_\ep\right)_+^p\partial_{x_1}\psi_{2,\varepsilon}^{(1)}  \right.\\
			&\quad\left.-\left(\psi_\ep^{(2)}(x)-\frac{W}{2}x_1^2\ln\frac{1}{\ep}-\mu^{(2)}_\ep\right)_+^p\partial_{x_1}\psi_{2,\varepsilon}^{(2)} \right)d  x=c_\ep\frac{3 s_\ep \kappa \Lambda_p }{4\pi } \ln\left(\frac{1}{s_\ep}\right)+O(\varepsilon),
		\end{align*}
where $c_\ep$ is the constant in Lemma \ref{lem2-15}.
\end{lemma}
\begin{proof}

To deal with the right hand side of \eqref{2-47}, we write
\begin{align*}
		&\quad\frac{x_{\ep,1}^2}{\varepsilon^2|\psi_\varepsilon^{(1)}-\psi_\varepsilon^{(2)}|_\infty}\int_{B_\delta(x_\ep)} \left(\left(\psi_\ep^{(1)}(x)-\frac{W}{2}x_1^2\ln\frac{1}{\ep}-\mu^{(1)}_\ep\right)_+^p\partial_{x_1}\psi_{2,\varepsilon}^{(1)}  \right.\\
		&\quad\left.-\left(\psi_\ep^{(2)}(x)-\frac{W}{2}x_1^2\ln\frac{1}{\ep}-\mu^{(2)}_\ep\right)_+^p\partial_{x_1}\psi_{2,\varepsilon}^{(2)} \right)d  x\\
		&=\frac{x_{\ep,1}^2}{\varepsilon^2|\psi_\varepsilon^{(1)}-\psi_\varepsilon^{(2)}|_\infty}\int_{B_\delta(x_\ep)} \left(\left(\psi_\ep^{(1)}(x)-\frac{W}{2}x_1^2\ln\frac{1}{\ep}-\mu^{(1)}_\ep\right)_+^p-\left(\psi_\ep^{(2)}(x)-\frac{W}{2}x_1^2\ln\frac{1}{\ep}-\mu^{(2)}_\ep\right)_+^p\right)\partial_{x_1}\psi_{2,\varepsilon}^{(1)}dx   \\
		&\quad+\frac{x_{\ep,1}^2}{\varepsilon^2|\psi_\varepsilon^{(1)}-\psi_\varepsilon^{(2)}|_\infty}\int_{B_\delta(x_\ep)}\left( \left(\psi_\ep^{(2)}(x)-\frac{W}{2}x_1^2\ln\frac{1}{\ep}-\mu^{(2)}_\ep\right)_+^p(\partial_{x_1}\psi_{2,\varepsilon}^{(1)} -\partial_{x_1}\psi_{2,\varepsilon}^{(2)}) \right)d  x\\
		&:=K_1+K_2.
	\end{align*}
We decompose $K_1$ into four parts.
\begin{align*}
		K_1&=\frac{x_{\ep,1}^2}{\varepsilon^2}\int_{B_\delta(x_\ep)} x_1^{-1}f_\varepsilon( x)\int_{\mathbb{R}_+^2} \partial_{x_1}H( x,y)  \left(\psi_\ep^{(1)}(y)-\frac{W}{2}y_1^2\ln\frac{1}{\ep}-\mu^{(1)}_\ep\right)_+^p dydx\\
		&=K_{11}+K_{12}+K_{13}+K_{14},\end{align*}
where
\begin{flalign*}
	&K_{11}=\frac{x_{\ep,1}^2}{4\pi\varepsilon^2} \ln\left(\frac{1}{s_\ep}\right) \int_{B_\delta(x_\ep)} x_1^{-3/2}f_\varepsilon(  x)\int_{\mathbb{R}_+^2}y_1^{3/2}\left(\psi_\ep^{(1)}(y)-\frac{W}{2}y_1^2\ln\frac{1}{\ep}-\mu^{(1)}_\ep\right)_+^p dyd x, &
\end{flalign*}
\begin{flalign*}
	&K_{12}=\frac{x_{\ep,1}^2}{4\pi\varepsilon^2} \int_{B_\delta(x_\ep)} x_1^{-3/2}f_\varepsilon( x)\int_{\mathbb{R}_+^2}y_1^{3/2} \ln\left(\frac{s_\ep}{| x-y|}\right)\left(\psi_\ep^{(1)}(y)-\frac{W}{2}y_1^2\ln\frac{1}{\ep}-\mu^{(1)}_\ep\right)_+^p dyd x,&
\end{flalign*}

\begin{flalign*}
	&K_{13}=-\frac{x_{\ep,1}^2}{2\pi\varepsilon^2}\int_{B_\delta(x_\ep)} x_1^{-1}f_\varepsilon(  x)\int_{\mathbb{R}_+^2}\left(x_1^{1/2}y_1 ^{3/2}-x_{\ep,1}^2\right) \frac{x_1-y_1}{| x-y|^2}\left(\psi_\ep^{(1)}(y)-\frac{W}{2}y_1^2\ln\frac{1}{\ep}-\mu^{(1)}_\ep\right)_+^p  dyd x &
\end{flalign*}
and $K_{14}$ is the remaining  regular term. Using the circulation constraint \eqref{2-3} and Lemma \ref{lem2-15}, we deduce by Taylor expansion that
\begin{align*}
		&K_{11}=\frac{x_{\ep,1}^2}{4\pi\varepsilon^2} \ln\left(\frac{1}{s_\ep}\right) \int_{B_\delta(x_\ep)} x_1^{-3/2}f_\varepsilon(  x)\int_{\mathbb{R}_+^2}y_1 \left(x_{\ep,1}^{1/2}+O(\ep)\right)\left(\psi_\ep^{(1)}(y)-\frac{W}{2}y_1^2\ln\frac{1}{\ep}-\mu^{(1)}_\ep\right)_+^p dyd x\\
		&=\frac{\kappa x_{\ep,1}^2}{4\pi } \ln\left(\frac{1}{s_\ep}\right)\left(x_{\ep,1}^{1/2}+O(\ep)\right)   \int_{B_\delta(x_\ep)} x_1^{-3/2}f_\varepsilon(  x) d  x\\
		&=\frac{\kappa x_{\ep,1}^2}{4 } \ln\left(\frac{1}{s_\ep}\right)\left(x_{\ep,1}^{1/2}+O(\ep)\right)   \int_{B_\delta(x_\ep)} f_\varepsilon \left(x_{\ep,1}^{-3/2}-\frac{3}{2x_{\ep,1}^{5/2}} (x_1-x_{\ep,1})+O(\varepsilon^2)\right)   d  x\\
		&=\frac{\kappa x_{\ep,1}^2}{4 } \ln\left(\frac{1}{s_\ep}\right)\left(x_{\ep,1}^{1/2}+O(\ep)\right)   \int_{B_{\delta s_\ep^{-1}}( 0)} s_\ep^2f_\varepsilon (s_\ep y+x_\ep)\left( -\frac{3 s_\ep y_1 }{2x_{\ep,1}^{5/2}}  +O(\varepsilon^2)\right)   d  x\\
		&=\frac{\kappa x_{\ep,1}^2}{4\pi } \ln\left(\frac{1}{s_\ep}\right)\left(x_{\ep,1}^{1/2}+O(\ep)\right)   p\int_{\mathbb{R}^2 } \left(U (y)\right)_+^{p-1} \left(c_\varepsilon \frac{\partial U}{\partial x_1}(y)+\textcolor{black}{O(\varepsilon )}\right)  \left( -\frac{3 s_\ep y_1 }{2x_{\ep,1}^{5/2}}  +O(\varepsilon^2)\right)   d  x\\
		&=c_\ep\frac{3 s_\ep \kappa \Lambda_p }{8\pi } \ln\left(\frac{1}{s_\ep}\right)+O(\varepsilon),
	\end{align*}
where we have used the following formula by integration by part
\begin{equation*}
	p\int_{\mathbb{R}^2 } \left(U (y)\right)_+^{p-1} \frac{\partial U}{\partial x_1}(y) y_1 dy=\int_{\mathbb{R}^2 } y_1\partial_{y_1} (\left(U (y)\right)_+^{p })   dy=-\int_{\mathbb{R}^2 }\left(U (y)\right)_+^{p }=\Lambda_p.
\end{equation*}

For the term $K_{12}$, it holds
\begin{align*}
		&K_{12}=\frac{x_{\ep,1}^2}{4\pi\varepsilon^2} \int_{B_\delta(x_\ep)} \left(x_{\ep,1}^{-3/2}+O(\ep)\right)f_\varepsilon( x)\\
		&\quad\times\int_{\mathbb{R}_+^2}\left(x_{\ep,1}^{3/2}+O(\ep)\right) \ln\left(\frac{s_\ep}{| x-y|}\right)\left(\psi_\ep^{(1)}(y)-\frac{W}{2}y_1^2\ln\frac{1}{\ep}-\mu^{(1)}_\ep\right)_+^p dyd x\\
		&=\frac{x_{\ep,1}^2}{4\pi\varepsilon^2} \int_{B_\delta(x_\ep)} f_\varepsilon( x) \int_{\mathbb{R}_+^2}  \ln\left(\frac{s_\ep}{| x-y|}\right)\left(\psi_\ep^{(1)}(y)-\frac{W}{2}y_1^2\ln\frac{1}{\ep}-\mu^{(1)}_\ep\right)_+^p dyd x+O(\varepsilon)\\
		&=\frac{1}{2}\int_{B_\delta(x_\ep)} f_\varepsilon( x)U\left(\frac{x-x_\ep}{s_\ep}\right)dx+O(\ep)\\
		&=\frac{p}{2}\int_{\mathbb{R}^2 } \left(U (y)\right)_+^{p-1} \left(c_\varepsilon \frac{\partial U}{\partial x_1}(y)+\textcolor{black}{O(\varepsilon )}\right)U(x)d x+O(\varepsilon) =O(\varepsilon),
	\end{align*}
where we have used the formula
\begin{align*}
		&	\frac{x_{\ep,1}^2}{2\pi\varepsilon^2}\int_{\mathbb{R}_+^2}  \ln\left(\frac{s_\ep}{| x-y|}\right)\left(\psi_\ep^{(1)}(y)-\frac{W}{2}y_1^2\ln\frac{1}{\ep}-\mu^{(1)}_\ep\right)_+^p dy\\
		=&\frac{x_{\ep,1}^2}{2\pi\varepsilon^2}\int_{\mathbb{R}_+^2}  \ln\left(\frac{s_\ep}{| x-y|}\right)\left(U_{\ep,x_\ep,a_\ep}-a_\ep\ln\frac{1}{\ep}\right)_+^p dy+O(\ep)\\
		=&U\left(\frac{x-x_\ep}{s_\ep}\right) +O(\ep)
	\end{align*}
and $\int_{\mathbb{R}^2 } \left(U (y)\right)_+^{p-1}  \frac{\partial U}{\partial x_1}(y) dy =0$ due to the odd symmetry.

Similarly, using the fact that
$\int_{\mathbb{R}_+^2}\left((x_1-x_{\ep,1})+3(y_1-x_{\ep,1})\right)  \frac{x_1-y_1}{| x-y|^2}\left(U_{\ep,x_\ep,a_\ep}-a_\ep\ln\frac{1}{\ep}\right)_+^p  dy$ is a bounded even function depending on  $|x_1-x_{\ep,1}|$ and $|x_2-x_{\ep,2}|$, we find
\begin{align*}
		&  K_{13} =-\frac{x_{\ep,1}^2}{2\pi\varepsilon^2}\int_{B_\delta(x_\ep)} x_1^{-1}f_\varepsilon(  x)\int_{\mathbb{R}_+^2}\left(x_1^{1/2}y_1 ^{3/2}-x_{\ep,1}^2\right) \frac{x_1-y_1}{| x-y|^2}\left(\psi_\ep^{(1)}(y)-\frac{W}{2}y_1^2\ln\frac{1}{\ep}-\mu^{(1)}_\ep\right)_+^p  dyd x\\
		&=-\frac{x_{\ep,1}^2}{2\pi\varepsilon^2}\int_{B_\delta(x_\ep)} x_{\ep,1}^{-1}f_\varepsilon(  x)\int_{\mathbb{R}_+^2}\left((x_1-x_{\ep,1})+3(y_1-x_{\ep,1})\right)  \frac{x_1-y_1}{| x-y|^2}\left(U_{\ep,x_\ep,a_\ep}-a_\ep\ln\frac{1}{\ep}\right)_+^p  dydx\\
		&\quad+O(\varepsilon)=O(\ep).
	\end{align*}

For the regular term $K_{14}$, it is easy to verify that $K_{14}=O(\varepsilon)$. Summarizing all the estimates above, we get
\begin{equation}\label{2-50}
	K_1=c_\ep\frac{3 s_\ep \kappa \Lambda_p }{8\pi } \ln\left(\frac{1}{s_\ep}\right)+O(\varepsilon).
\end{equation}

Next, we turn to deal with $K_2$. Using Fubini's theorem, one can easily check that
\begin{equation*}
	\begin{split}
		K_2&=\frac{x_{\ep,1}^2}{\varepsilon^2}\int_{B_\delta(x_\ep)} x_1^{-1}f_\varepsilon( x)\int_{\mathbb{R}_+^2} \partial_{x_1}H( x,y)  \left(\psi_\ep^{(2)}(y)-\frac{W}{2}y_1^2\ln\frac{1}{\ep}-\mu^{(2)}_\ep\right)_+^p dydx,
	\end{split}
\end{equation*}
which is quite similar to $K_1$. Therefore, by calculations similarly, we obtain
\begin{equation}\label{2-51}
	K_2=c_\ep\frac{3 s_\ep \kappa \Lambda_p }{8\pi} \ln\left(\frac{1}{s_\ep}\right)+O(\varepsilon).
\end{equation}
Then this lemma follows immediately from \eqref{2-50} and \eqref{2-51} and  the proof is hence completed.
\end{proof}

Now, we are able to prove the uniqueness.

\noindent{\bf Proof of Theorem \ref{thm2-1}:}
By \eqref{2-47}, \eqref{2-49} and Lemma \ref{lem2-17}, we conclude
\begin{equation}\label{2-52}
	c_\ep\frac{3 s_\ep \kappa \Lambda_p }{4\pi} \ln\left(\frac{1}{s_\ep}\right)=O(\varepsilon),
\end{equation}
from which we derive $|	c_\varepsilon| =O\left(\frac{1}{|\ln\varepsilon|}\right)$.
According to Lemma \ref{lem2-15}, we can   use the fact that for fixed $  y\in\mathbb R^2$
$$\frac{1}{2\pi}\ln\left(\frac{1}{| y-\cdot|}\right)\in W_{\text{loc}}^{1,q'}(\mathbb R^2),\quad \forall \, q'\in [1,2)$$
to deduce
\begin{align*}
		& \xi_{1,\varepsilon}(s_\ep y+x_\ep) =\frac{x_{\ep,1}^2}{2\pi}\int_{B_2(0)} (s_\ep z_1+x_{\ep,1})^{-1}\ln \left(\frac{2x_{\ep,1}+s_\ep(y_1+z_1, y_2-_2)|}{s_\ep|y- z|} \right)s_\ep^2f_\varepsilon(s_\ep z+x_\ep)dz\\
		&= \frac{x_{\ep,1}^2}{2\pi}\int_{B_2(0)} (s_\ep z_1+x_{\ep,1})^{-1}\ln \left(\frac{2x_{\ep,1}+s_\ep(y_1+z_1, y_2-_2)|}{s_\ep|y- z|} \right) \left(U (z)\right)_+^{p-1} \left(c_\varepsilon \frac{\partial U}{\partial x_1}(y)+\textcolor{black}{O(\varepsilon )}\right) dz\\
		&=O\left(|c_\ep|\right)=O\left(\frac{1}{|\ln\varepsilon|}\right).
	\end{align*}
Thus we conclude that $||\xi_{1,\varepsilon}||_{L^\infty(\mathbb R^2_+)}=O(1/|\ln\varepsilon|)=o_\ep(1)$, which   contradicts \eqref{xi}.  So, we obtain a contradiction, which implies that   Theorem \ref{thm2-1} holds true.
\qed

\section{Nonlinear orbital stability of vortex rings}\label{sec3}
In Section \ref{sec2}, we obtain a nonlinear stability theorem on the set of maximizers. As a consequence, it is nature to study the structure of the set of maximizers, which is extremely difficult in general. Thus, we restrict ourselves to consider the case of  concentrated steady vortex rings, of which the existence was established for a large class of nonlinearities in \cite{CWWZ}, including the vortex rings in Proposition \ref{thmE} as a special case.

In what follows, we will apply the general stability result Theorem \ref{thmSs} to obtain orbital stability of =vortex rings. We can reduce the stability of the concentrated vortex ring $\zeta_\ep$ obtained in \cite{CWWZ} to a problem of uniqueness (see Theorem \ref{U-S}) and hence by Theorem \ref{thmU}, we finally obtain the nonlinear stability of the steady vortex rings given in Proposition \ref{thmE} and finish Theorem \ref{thmS}.

\subsection{Uniqueness implies orbital stability}\label{sec3-3}

First let us recall a result in \cite{CWWZ}.
 Let $f:\mathbb{R}\to \mathbb{R}$ be a function satisfying\\
\noindent\, $(f_1).$ $f(s)\equiv 0$ for $s\le 0$, $f(s)>0$ for $s> 0$, and $f$ is non-decreasing in $\mathbb{R}$; \\
 and either\\
\noindent\, $(f_2).$ $f$ is bounded;\\
or\\
\noindent\,$(f_3).$ There exists some positive numbers $\delta_0\in(0,1)$ and $\delta_1>0$ such that
 	\[F(s):=\int_{0}^{s}f(t)dt \le \delta_0 f(s)s+\delta_1f(s), \ \ \forall~ s\ge 0. \]
 	In addition, for all $\tau >0$,
	$$\lim_{s\to +\infty}\left(f(s)\,e^{-\tau s}\right)=0.$$

 Let $\kappa>0$, $W>0$ be fixed and $0<\varepsilon<1$ be a parameter. Let  $D_0=\{x=(x_1, x_2)\in \mathbb{R}_+^2\mid \kappa/(8\pi W)\le x_1\le \kappa/(4\pi W)+1, -1\le x_2\le 1\},$ and define
 \begin{equation*}
 	\mathcal{A}:=\left\{0\leq \zeta\in L^\infty(\mathbb{R}_+^2)~\bigg|~\int_{\mathbb{R}_+^2} \zeta d\nu\le \kappa,\ \text{supp}(\zeta)\subseteq D_0\right\}.
 \end{equation*}
 Consider the maximization of the following functional over $\mathcal{A}$
 \begin{equation}\label{3-21}
 	E_\varepsilon(\zeta):= (E-W\ln\frac{1}{\ep} P)(\zeta)-\frac{1}{\varepsilon^2}\int_{D_0} G(\varepsilon^2\zeta)d\nu,
 \end{equation}
where $	E(\zeta)=\frac{1}{2}\int_{\mathbb R^2_+}\xi(x)\mathcal{G}_1\xi(x) d\nu $, $P(\zeta)=\frac{1}{2}\int_{\mathbb R^2_+}\xi(x) x_1^2 d\nu$ and
 \begin{equation*}
 	G(s)=\sup_{s'\in\mathbb{R}}\left[ss'-F(s')\right]
 \end{equation*}
is the conjugate function to $F(s):=\int_{0}^{s}f(t)dt$.

 Denote by $\tilde\Sigma_{\varepsilon}\subset 	\mathcal{A}$ the set of maximizers of \eqref{3-21}. Note that any $x_2$-directional translation of $\zeta\in \tilde\Sigma_{\varepsilon}$ is still in $\tilde \Sigma_{\varepsilon}$.

 The following result is essentially contained in \cite{CWWZ}.
 \begin{proposition}[\cite{CWWZ}]\label{prop3-11}
 	Suppose that $f$ is a function satisfying all the conditions  $(f_1)$ and either $(f_2)$ or $(f_3)$. If $\varepsilon>0$ is sufficiently small, then $\tilde\Sigma_{\varepsilon}\neq \emptyset$ and each maximizer ${\zeta}_\varepsilon \in \tilde\Sigma_{\varepsilon}$ is a steady vortex ring satisfying the following properties.
 	\begin{itemize}
 		\item [(i).]$ {\zeta}_\varepsilon=\varepsilon^{-2}f(\mathcal{G}_1\zeta_\ep-\frac{W}{2}x_1^2\ln\frac{1}{\ep}-\mu_\ep)$ for some constant $\mu_\ep$ and $\mathcal{G}_1\zeta_\ep-\frac{W}{2}x_1^2\ln\frac{1}{\ep}-\mu_\ep$ is bounded from above independent of $\ep$.
 		\item [(ii).]  $C_1\ep \le \mathrm{diam}\left(\mathrm{supp}({\zeta}_\varepsilon)\right)<C_2\ep$ for some constants $0<C_1<C_2<\infty$ and $$\sup_{x\in \mathrm{supp}({\zeta}_\varepsilon)} |x-(r_*,0)|=o_{\ep}(1)$$ with $r_*={\kappa}/{4\pi W}$.
 		\item [(iii).] $\int_{\mathbb{R}_+^2}  \zeta_{\varepsilon}d\nu=\kappa$, $\ep^{-2}\int_{D_0} G(\varepsilon^2\zeta)d\nu =O(1),$
 		\begin{equation*}
 			E_\varepsilon(\zeta_\varepsilon)  =\left(\frac{\kappa^2r_*}{4\pi}-\frac{\kappa Wr_*^2}{2}\right)\log{\frac{1}{\varepsilon}}+O(1),
 		\end{equation*}
 	and
 	\begin{equation*}
 		\mu^\varepsilon  =\left(\frac{\kappa r_*}{2\pi}-\frac{Wr_*^2}{2}\right)\log{\frac{1}{\varepsilon}}+o\left(\log{\frac{1}{\varepsilon}}\right).
 	\end{equation*}

 	\end{itemize}
 \end{proposition}

 Note that the above maximization problem is quite different from that we have considered in Section \ref{sec2}. Taking a steady vortex ring  $\zeta_\ep\in \tilde\Sigma_\ep$ obtained in Proposition \ref{prop3-11}, to apply our stability result Theorem \ref{thmSs} (or Theorem \ref{Sset}) to show the stability of $\zeta_\ep$,  we need to consider the maximizing problem in rearrangement class:
 \begin{equation}\label{3-22}
 	\begin{split}
  \sup_{\zeta\in\overline{\mathcal{R}(\zeta_\ep)^w}}\mathcal E_{-W\ln\ep}(\zeta),
 	\end{split}
 \end{equation}
where $$\mathcal E_{-W\ln\ep}(\zeta)	:=\frac{1}{2}\int_{\mathbb{R}_+^2}{\zeta \mathcal{G}_1\zeta}d\nu-\frac{W}{2}\log\frac{1}{\varepsilon}\int_{\mathbb{R}_+^2}x_1^2\zeta d\nu.$$
 For simplicity, we    denote $\mathcal E_\ep:=\mathcal E_{-W\ln\ep}$ and  $  \Sigma_\ep$ as the set of maximizers of $\mathcal E_\varepsilon$ over the rearrangement class $\overline{\mathcal{R}(\zeta_\ep)^w}$. Our  results for problem \eqref{3-22} are presented in the following theorem, whose proof will be given in the next subsection.
 \begin{theorem}\label{thm3-12}
 	For $\ep>0$ small, the set of maximizers $\Sigma_\ep$ satisfies $$\emptyset\neq  \Sigma_\ep\subset \mathcal{R}(\zeta_\ep).$$ Moreover, for each  $\zeta^\ep \in   \Sigma_\ep$,  the following claims hold.
 	\begin{itemize}
 	\item [(i).]$ {\zeta}^\varepsilon= g_\ep(\mathcal{G}_1\zeta^\ep-\frac{W}{2}\ln\frac{1}{\ep}x_1^2-\tilde\mu_\ep)$ for some non-negative and non-decreasing function $g_\varepsilon:\mathbb{R}\to \mathbb{R}$ satisfying $g_\varepsilon(\tau)>0$ if $\tau>0$ and $g_\varepsilon(\tau)=0$ if $\tau\le0$ and some constant $\tilde \mu_\ep$.
 	\item [(ii).]  $C_3\ep \le \text{diam}\left(\mathrm{supp}({\zeta}^\varepsilon)\right)<C_4\ep$ for some constants $0<C_3<C_4<\infty$ and up to a suitable translation in the $x_2$ direction $\sup_{x\in \mathrm{supp}({\zeta}^\varepsilon)} |x-(r_*,0)|=o_{\ep}(1)$ with $r_*={\kappa}/{4\pi W}$.\end{itemize}
 \end{theorem}

With the asymptotic behavior obtained in Theorem \ref{thm3-12}, we prove the following theorem, which reduces the stability of a concentrated vortex ring to a problem of the local uniqueness of solutions to an elliptic system.
\begin{theorem}\label{U-S}
	Suppose that $f$ satisfies conditions $(f_1)$ and either $(f_2)$ or $(f_3)$. For given $\kappa>0$ and $W>0$, let $\zeta_\ep$ be the vortex ring obtained in Proposition \ref{prop3-11} with the distribution function $f$. Suppose that for fixed $\ep>0$ small, the solution to following problem is unique up to translations in the $x_2$-direction:
	\begin{equation}\label{3-45}
		\begin{cases}
			-\frac{1}{x_1}\mathrm{div}\left(\frac{1}{x_1}\nabla u\right) =\frac{1}{\ep^2}f\left(u-\frac{W}{2}x_1^2\ln\frac{1}{\ep}-\mu_\ep\right) , & \text{in} \ \mathbb R^2_+,
			\\
			u=0, & \text{on} \ x_1=0,
			\\
			u  \to0, &\text{as} \ |  x |\to \infty,\\
			\ep^{-2}\int_{\mathbb{R}_+^2} f\left(u-\frac{W}{2}x_1^2\ln\frac{1}{\ep}-\mu_\ep\right)=\kappa,\\
			\mathrm{diam}\left(\mathrm{supp}\left(\left(u-\frac{W}{2}x_1^2\ln\frac{1}{\ep}-\mu_\ep\right)_+\right) \right)\leq C\ep,  & \text{for some constant}\  C>0,\\
			\mathrm{dist}\left(\mathrm{supp}\left(\left(u-\frac{W}{2}x_1^2\ln\frac{1}{\ep}-\mu_\ep\right)_+\right), x_0 \right)\to 0, & \text{as}\ \ep\to0\  \text{for some}\ x_0\in \mathbb{R}_+^2.
		\end{cases}
	\end{equation}
	Then $\zeta_\ep$ is orbitally stable in the following sense:
	
	For every $\epsilon>0$, there exists $\delta>0$ such that for any nonnegative $\omega_0 $ satisfying $\omega_0, x_1\omega_0\in L^\infty(\mathbb R^2_+)$,  $P(|\omega_0|)<+\infty$  and
	\begin{equation*}
		\inf_{c\in \mathbb{R}}\Bigg{\{}\vertiii{\omega_0-\zeta_{\ep}(\cdot+c\mathbf{e}_2)}_1+\vertiii{\omega_0-\zeta_{\ep}(\cdot+c\mathbf{e}_2)}_2+|P(\omega_0-\zeta_{\ep}(\cdot+c\mathbf{e}_2))|\Bigg{\}} \leq \delta,
	\end{equation*}
	it holds
	\begin{equation*}
		\inf_{c\in \mathbb{R}} \Big{\{}\vertiii{\omega(t)-\zeta_{\ep}(\cdot+c\mathbf{e}_2)}_1+\vertiii{\omega(t)-\zeta_{\ep}(\cdot+c\mathbf{e}_2)}_2+P(|\omega(t)-\zeta_{\ep}(\cdot+c\mathbf{e}_2)|) \Big{\}}\leq \epsilon,\ \ \text{for all}\,\, t>0,
	\end{equation*}
	where $\omega(t)$ is the corresponding  solution of \eqref{1-3}  with initial data $\omega_0$.
\end{theorem}
\begin{proof}
	  Recall that  $\tilde\Sigma_{\varepsilon}\subset 	\mathcal{A}$ denotes the set of maximizers of $E_\ep$ defined by \eqref{3-21} and there exists at least one maximizer $\zeta_\ep\in \tilde\Sigma_{\varepsilon}$ for $\ep$ small. For arbitrary $\tilde\zeta_\ep\in \tilde\Sigma_{\varepsilon}$, we deduce from Proposition \ref{prop3-11} that the corresponding stream function $\tilde \psi_\ep:=\mathcal{G}_1\tilde\zeta_\ep$ is a solution \eqref{3-45}. Therefore, by the assumption of the uniqueness, we conclude that
	$$\tilde\Sigma_{\varepsilon}=\left\{ \zeta_\ep(\cdot+c\mathbf{e}_2)\mid c\in\mathbb{R}\right\}.$$
	
	Next, as in Theorem \ref{thm3-12}, we maximize   $\mathcal E_\varepsilon$ defined by \eqref{3-22} over the rearrangement class $\overline{\mathcal{R}(\zeta_\ep)^w}$ and denote $  \Sigma_\ep$ as the set of maximizers. Then, by Theorem \ref{thm3-12}, we have $$\emptyset\neq  \Sigma_\ep\subset \mathcal{R}(\zeta_\ep)$$ and $$\text{supp}(\zeta^\ep)\subset D_0,$$
	which implies $\Sigma_\ep\subset \mathcal{A}$.
	Notice that $\frac{1}{\varepsilon^2}\int_{D_0} G(\varepsilon^2\zeta)d\nu=const., \quad \forall \zeta\in \mathcal{R}(\zeta_\ep)$ by the property of rearrangement. Then it can be  easily seen that
	$$\Sigma_\ep=\tilde\Sigma_{\varepsilon}=\left\{ \zeta_\ep(\cdot+c\mathbf{e}_2)\mid c\in\mathbb{R}\right\}.$$
	
	Therefore,   the stability of $\zeta_\ep$ follows immediately by applying Theorem \ref{Sset} to $\Sigma_\ep$ and thus the proof is complete.
\end{proof}

Now, we apply Theorem \ref{thmU} (or Theorem \ref{thm2-1}) and Theorem \ref{U-S} to prove Theorem \ref{thmS}.

\noindent{\bf Proof of Theorem \ref{thmS}:}
The proof of Theorem \ref{thmS} is a combinations of Theorem \ref{thmU} and Theorem \ref{U-S}. \qed

\subsection{Asymptotic behavior of maximizers in rearrangement class and proof of Theorem 4.2}\label{sec3-2}
We   study the asymptotic behaviors of maximizers of   problem \eqref{3-22} in rearrangement class by an adapted vorticity method of \cite{CWWZ}, which will prove Theorem \ref{thm3-12}.
A series of lemmas are needed.
 \begin{lemma}\label{lem3-13}
 	\begin{equation*}
 		\max_{\zeta\in\overline{\mathcal{R}(\zeta_\ep)^w}}\mathcal E_\varepsilon(\zeta)\geq \left(\frac{\kappa^2r_*}{4\pi}-\frac{\kappa Wr_*^2}{2}\right)\log{\frac{1}{\varepsilon}}+O_\varepsilon(1).
 	\end{equation*}
 \end{lemma}
 \begin{proof}
 	This is a simple consequence of the fact $\zeta_\ep\in\overline{\mathcal{R}(\zeta_\ep)^w}$ and Proposition \ref{prop3-11} (iii).
 \end{proof}

 \begin{lemma}\label{lem3-14}
 	$\mathcal{E}_\ep$ attains its maximum value over $\overline{\mathcal{R}(\zeta_\ep)^w}$ at some $\zeta^\varepsilon$, which is Steiner symmetric about the $x_1$-axis.
 \end{lemma}
 \begin{proof}
 	By Lemma \ref{lem3-5}, there exists a constant $R_\ep>0$   depending  on $\ep$, $\zeta_\ep$ and $W$ such that
 	$$\mathcal{G}_1\zeta(x)-\frac{W }{2}x_1^2\ln \frac{1}{\ep}<0,\quad \forall x\ \text {with}\  x_1>R_\ep, \quad \forall \zeta\in \overline{\mathcal{R}(\zeta_\ep)^w}.$$
 	 It is easy to check that $\mathcal{E}_\ep$ is bounded from above over $\overline{\mathcal{R}(\zeta_\ep)^w}$ by Lemma \ref{lem3-7}. Let $\{\zeta_j\}\subset \overline{\mathcal{R}(\zeta_\ep)^w}$ be a sequence such that as $j\to+\infty$
 	\begin{equation*}
 		\mathcal{E}_\ep(\zeta_j)\to \sup_{\overline{\mathcal{R}(\zeta_\ep)^w}}\mathcal{E}_\ep.
 	\end{equation*}
 	We may assume that $\zeta_j$ is supported in $(0, R_\ep)\times \mathbb R$ by Lemma \ref{lem3-5}. We can also assume $\zeta_j$ is Steiner symmetric about the $x_1$-axis by replacing $\zeta_j$ with its own  Steiner symmetrization. Since $\int_{\mathbb{R}_+^2} x_1^2 \zeta_j d\nu \leq R_\ep^2 \kappa$,  there exists a constant $Z_\ep>0$ such that $\mathcal{G}_1\zeta_j(x)-\frac{W}{2}x_1^2\ln\frac{1}{\ep} <0,\quad \forall x\ \text {with}\  |x_2|>Z_\ep,\quad\forall \  j$ due to \eqref{3-4}. Thus, by  Lemma \ref{lem3-5} ($\mathrm{ii}$), we may assume that  $\zeta_j$ is supported in $[0, R_\ep]\times [-Z_\ep,Z_\ep]$. Then, up to a subsequence, we may assume that as $j\to +\infty$, $\zeta_j\to \zeta^\varepsilon \in \overline{\mathcal{R}(\zeta_\ep)^w}$ weakly in $L^{2}(\mathbb{R}^2, \nu)$. Since $G_1(\cdot,\cdot)\in L^2_{loc} $ by \eqref{2-9}, we deduce that
 	\begin{equation*}
 		\lim_{j\to +\infty}\mathcal{E}_\ep(\zeta_j)=\mathcal{E}_\ep(\zeta^\varepsilon).
 	\end{equation*}
 	This means that $\zeta^\varepsilon$ is a maximizer and completes the proof.
 \end{proof}
 Having shown the existence of a maximizer, to study the properties of maximizers, we need the following lemma from \cite{Bu1}.
 \begin{lemma}[Lemmas 2.4 and 2.9 in \cite{Bu1}]\label{lem3-15}
 	Let $(\Omega,\nu)$ be a finite positive measure space. Let $\xi_0: \Omega\to \mathbb{R}$ and $\zeta_0: \Omega\to \mathbb{R}$ be $\nu$-measurable functions, and suppose that every level set of $\zeta_0$ has zero measure. Then there is a nondecreasing function $f$ such that $f\circ \zeta_0$ is a rearrangement of $\xi_0$. Moreover, if $\xi_0 \in L^q(\Omega,\nu)$ for some $1\le q<+\infty$ and $\zeta_0\in L^{q'}(\Omega,\nu)$, then $f\circ \zeta_0$ is the unique maximizer of linear functional
 	\begin{equation*}
 		M(\xi):=\int_\Omega \xi(x)\zeta_0(x)d\nu(x)
 	\end{equation*}
 	relative to $\overline{\mathcal{R}(\xi_0)^w}$.
 \end{lemma}

 \begin{lemma}\label{lem3-16}
 	Let  $\zeta^\varepsilon \in \Sigma_\ep$ be an arbitrary maximizer. Then, up to a translation in the $x_2$ direction, $\zeta^\ep$ must be Steiner symmetric about the $x_1$-axis and supported in $[0, R_\ep]\times [-Z_\ep,Z_\ep]$.   Moreover,   for $\ep>0$ small,     $0\not\equiv \zeta^\varepsilon\in \mathcal{R}(\zeta_\ep)$ and there exists a nonnegative and nondecreasing function $g_\varepsilon: \mathbb{R}\to \mathbb{R}$ satisfying $g_\varepsilon(\tau)>0$ if $\tau>0$  and $g_\varepsilon(\tau)=0$ if $\tau\le 0$, such that for some constant $\tilde\mu_\varepsilon \geq 0$,
 	\begin{equation}\label{3-23}
 		\zeta^\varepsilon(x)=g_\varepsilon(\mathcal{G}_1\zeta^\varepsilon(x)-\frac{W}{2}x_1^2\ln\frac{1}{\ep} -\tilde\mu_\varepsilon),\ \ \forall\,x\in \mathbb{R}_+^2.
 	\end{equation}
 	
 \end{lemma}
 \begin{proof}
 	For $\ep>0$ small, we see from Lemma \ref{lem3-13} that $\mathcal{E}_\ep(\zeta^\varepsilon)>0$ and hence $\zeta^\ep\not\equiv 0$.

 	Thanks to Lemma \ref{lem3-1}, we know that $\overline{\mathcal{R}(\zeta_\ep)^w}$ is a convex set. Thus for each $\zeta\in \overline{\mathcal{R}(\zeta_\ep)^w}$, it holds $\zeta_\tau:=\zeta^\varepsilon+\tau(\zeta-\zeta^\varepsilon)\in \overline{\mathcal{R}(\zeta_\ep)^w}$ for any $\tau\in [0,1]$. Since $\zeta^\varepsilon$ is a maximizer of $\mathcal E_\ep$, we have
 	\begin{equation*}
 		 \frac{d}{d\tau}\bigg|_{\tau=0^+}\mathcal{E}_\ep(\zeta_\tau)=\int_{\mathbb{R}_+^2}(\zeta-\zeta^\varepsilon)\left(\mathcal{G}_1\zeta^\varepsilon -\frac{W}{2}x_1^2 \ln \frac{1}{\ep}\right) d\mu\le 0,
 	\end{equation*}
 	which yields
 	\begin{equation*}
 		\int_{\mathbb{R}_+^2}\zeta\left(\mathcal{G}_1\zeta^\varepsilon -\frac{W}{2}x_1^2 \ln \frac{1}{\ep}\right) d\nu\le \int_{\mathbb{R}_+^2}\zeta^\varepsilon \left(\mathcal{G}_1\zeta^\varepsilon -\frac{W}{2}x_1^2 \ln \frac{1}{\ep}\right) d\nu, \quad \forall \  \zeta\in \overline{\mathcal{R}(\zeta_\ep)^w}.
 	\end{equation*}

 	By \eqref{rieq}, we know that after a translation in the $x_2$ direction, $\zeta^\ep$ must be Steiner symmetric about the $x_1$-axis. Then 	by Lemma \ref{lem3-5} ($\mathrm{ii}$), through an similar argument as the proof of Lemma \ref{lem3-14}, we find that $\zeta^\ep$ is supported in $[0, R_\ep]\times [-Z_\ep,Z_\ep]$.  Using the fact that $\zeta^\varepsilon$ is Steiner symmetric about the $x_1$-axis, one can verified   that $\mathcal{G}_1\zeta^\varepsilon$ is even and strictly decreasing with respect to $x_2$. It follows from the implicit function theorem that every level set of $\mathcal{G}_1\zeta^\varepsilon-\frac{W}{2}x_1^2 \ln \frac{1}{\ep}$  has measure zero. By  Lemma \ref{lem3-15}, there exists a non-decreasing function $\tilde{g}_\varepsilon: \mathbb{R}\to \mathbb{R}$, such that,
 	\begin{equation*}
 		\tilde\zeta^\varepsilon(x)=\tilde{g}_\varepsilon\left(\mathcal{G}_1\zeta^\varepsilon(x)-\frac{W}{2}x_1^2 \ln \frac{1}{\ep}\right).
 	\end{equation*}
 	for some $\tilde\zeta^\varepsilon\in \mathcal{R}(\zeta_\ep)$. From the conclusion of Lemma \ref{lem3-15}, we also know $\tilde\zeta^\varepsilon(x)$ is the unique maximizer of the linear functional $\zeta \mapsto \int_{[0, R_\ep]\times [-Z_\ep,Z_\ep]} \zeta \left(\mathcal{G}_1\zeta^\varepsilon -\frac{W}{2}x_1^2 \ln \frac{1}{\ep}\right) d\nu$ relative to $\overline{\mathcal{R}(\zeta_\ep)^w}$. Hence it holds $\zeta^\varepsilon=\tilde\zeta^\varepsilon\in \mathcal{R}(\zeta_\ep)$.
 	
 	Now, let
 	\begin{equation*}
 		\tilde\mu_\varepsilon:=\sup\left\{\mathcal{G}_1\zeta^\varepsilon(x)-\frac{W}{2}x_1^2 \ln \frac{1}{\ep}\bigg|\ x\in \mathbb{R}_+^2\ \ \text{s.t.}\  \zeta^\varepsilon(x)=0 \right\}\in \mathbb{R},
 	\end{equation*}
 	and $g_\varepsilon(\cdot)=\max\{\tilde{g}_\varepsilon(\cdot+\tilde\mu_\varepsilon),0\}$. We have $$\zeta^\varepsilon(x)=g_\varepsilon\left(\mathcal{G}_1\zeta^\varepsilon(x)-\frac{W}{2}x_1^2 \ln \frac{1}{\ep}-\tilde\mu_\varepsilon\right) \quad \text{for any} \ x\in \mathbb{R}_+^2.$$
 	Moreover, by the definition of $\tilde\mu_\varepsilon$ and the continuity of $\mathcal{G}_1\zeta^\varepsilon(x)$, we have $g_\ep(\tau)>0$ if $\tau>0$ and $g_\varepsilon(\tau)=0$ if $\tau\le 0$.
 	
 	It remains to show that $\tilde \mu_\ep\geq 0$. Since $\nu(\text{supp}(\zeta^\varepsilon))=\nu(\text{supp} (\zeta_\varepsilon))<\infty$ by the property of rearrangement, we can take a sequence $\{x_n\}_{n=1}^\infty\subset \mathbb{R}_+^2$ with $x_{n,1}\to 0$ and $\zeta^\varepsilon(x_n)=0$. Then by the definition of $\tilde \mu_\ep$, we have $$\tilde \mu_\ep\geq \limsup_{n\to\infty}\left\{\mathcal{G}_1\zeta^\varepsilon(x_n)-\frac{W}{2}x_{n,1}^2 \ln \frac{1}{\ep}\right\}\geq \limsup_{n\to\infty}\left\{ -\frac{W}{2}x_{n,1}^2 \ln \frac{1}{\ep}\right\}= 0.$$

 	The proof is thus complete.
 \end{proof}

 By Lemma \ref{lem3-16}, we may assume that the $\zeta^\ep$ considered later in Lemmas \ref{lem3-17}- \ref{lem3-24} is Steiner symmetric about the $x_1$-axis. We improve the estimate of the support of a maximizer in the above lemmas step by step.
 \begin{lemma}\label{lem3-17}
 	There exists a constant $R_1>0$ independent of $\ep$ such that
 	$$\mathrm{supp}(\zeta^\ep)\subset [0, R_1]\times \mathbb R,\quad \forall \ \zeta^\ep \in   \Sigma_\ep.$$
 \end{lemma}
 \begin{proof}
 	 Let $\eta:=\left((x_1-y_1)^2+(x_2-y_2)^2\right)^{\frac{1}{2}}$. We need the following estimates for the Green function from \cite{FT} (see also (2.5) in \cite{Bad})
 	\begin{equation}\label{3-25}
 		G_1(x,y)\leq C\times
 		\begin{cases}
 			x_1 \ln\frac{x_1}{\eta},\quad &\eta \leq \frac{1}{2}x_1,\\
 			\frac{x_1^2 y_1^2 }{\eta^3},&\eta \geq \frac{1}{2}x_1,
 		\end{cases}
 	\end{equation}
 	which can also be verified directly  by using \eqref{2-7} and \eqref{2-8}.
 	We consider $x\in\mathbb{R}_+^2$ with $x_1\geq 2$ and divide the integral into three parts as following
 	\begin{equation*}
 		\mathcal{G}_1 \zeta^\ep=\int_{\eta>\frac{1}{2}x_1} G_1(x,y) \zeta^\ep(y)d\nu+\int_{x_1^{-\frac{1}{2}}\leq \eta\leq \frac{1}{2}x_1} G_1(x,y) \zeta^\ep(y)d\nu+\int_{  \eta\leq x_1^{-\frac{1}{2} }} G_1(x,y) \zeta^\ep(y)d\nu.
 	\end{equation*}
 	For the first integral, noticing  that when $\eta>\frac{1}{2}x_1$, it holds $y_1\leq x_1+\eta\leq 3 \eta$, by \eqref{3-25}  we obtain
 	\begin{equation*}
 		\int_{\eta>\frac{1}{2}x_1} G_1(x,y) \zeta^\ep(y)d\nu \leq C x_1 \vertiii{\zeta^\ep}_1=C\kappa x_1.
 	\end{equation*}
 	For the second integral,   by \eqref{3-25}  we get
 	\begin{align*}
 			&\int_{x_1^{-\frac{1}{2}}\leq \eta\leq \frac{1}{2}x_1} G_1(x,y) \zeta^\ep(y)d\nu\\
 			\leq &Cx_1\int_{x_1^{-\frac{1}{2}}\leq \eta\leq \frac{1}{2}x_1}  \ln\left(\frac{x_1}{\eta}\right) \zeta^\ep(y)d\nu\leq Cx_1 \ln\left(x_1^{\frac{3}{2}}\right)\vertiii{\zeta^\ep}_1=\frac{3}{2}C\kappa x_1 \ln( x_1).
 		\end{align*}
 	For the lat integral, on the set  $\{\eta\leq x_1^{-\frac{1}{2} }\}$, it is easy to see that $y_1\leq 2x_1$. Then the  function  $y_1\zeta^\ep(y) $  satisfies $0\leq y_1\zeta^\ep(y)\leq C_5x_1\ep^{-2}$ for some constant $C_5$ and $\int_{\eta\leq x_1^{-\frac{1}{2} }} y_1\zeta^\ep(y) dy \leq \kappa$ by Proposition \ref{prop3-11} and the property of rearrangement. Thus, by the Bathtub principle \cite{Lieb}, we obtain
 	\begin{equation*}
 		\int_{ \eta\leq x_1^{-\frac{1}{2} }}  \ln\left(\frac{1}{\eta}\right) y_1 \zeta^\ep(y) dy\leq C_5x_1\ep^{-2}	\int_{ |y|\leq  \sqrt{\kappa/ (\pi C_5x_1)}\ep}\ln\left(\frac{1}{|y|}\right)  dy \leq C\ln \frac{1}{\ep}+C\ln x_1,
 	\end{equation*}
 which deduces that
 	\begin{align*}
 			&\int_{  \eta\leq x_1^{-\frac{1}{2} }} G_1(x,y) \zeta^\ep(y)d\nu
 			\leq  Cx_1\int_{ \eta\leq x_1^{-\frac{1}{2} }}  \ln\left(\frac{x_1}{\eta}\right) \zeta^\ep(y)d\nu\\
 			= &C\kappa x_1 \ln (x_1)+Cx_1\int_{ \eta\leq x_1^{-\frac{1}{2} }}  \ln\left(\frac{1}{\eta}\right) \zeta^\ep(y)d\nu
 			\leq  Cx_1 \ln \frac{1}{\ep}+C  x_1 \ln( x_1).
 		\end{align*}
 	Summarizing the above calculations,  we reach the key estimate for $x_1\geq 2$,
 	\begin{equation*}
 		\mathcal{G}_1 \zeta^\ep(x)\leq Cx_1 \ln \frac{1}{\ep}+ Cx_1\ln( x_1), \quad \forall \ \zeta^\ep.
 	\end{equation*}
 	On the other hand, for $x\in \mathbb{R}_+^2$ such that $x_1\geq 2$ and $\zeta^\ep(x)>0$, by Lemma \ref{lem3-16} it holds
 	\begin{equation*}
 		\mathcal{G}_1\zeta^\varepsilon(x)\geq \frac{W}{2}x_1^2 \ln \frac{1}{\ep}+\tilde\mu_\varepsilon\geq \frac{W}{2}x_1^2 \ln \frac{1}{\ep}.
 	\end{equation*}
 	Thus, we deduce
 	\begin{equation*}
 		\frac{W}{2}x_1^2 \ln \frac{1}{\ep}\leq Cx_1 \ln \frac{1}{\ep}+ Cx_1\ln( x_1),
 	\end{equation*}
 	which implies
 	$$x_1\leq R_1,$$
 	for some constant $R_1$. The proof is thus finished.
 \end{proof}
 Using the previous lemma, we are able to obtain a better estimate for the constant $\tilde \mu_\ep$.
 \begin{lemma}\label{lem3-18}
 	For a maximizer $\zeta^\varepsilon\in \Sigma_\ep$, let $\tilde\mu_\ep$ be the constant in Lemma \ref{lem3-16}, then
 	\begin{equation*}
 		\tilde \mu_\ep\geq \left(\frac{\kappa r_*}{2\pi}-   Wr_*^2 \right)\log{\frac{1}{\varepsilon}}+\frac{W}{2\kappa}\log\frac{1}{\varepsilon}\int_{\mathbb{R}_+^2}x_1^2\zeta^\ep d\nu+O_\ep(1).
 	\end{equation*}
 \end{lemma}
 \begin{proof}
 	Take a maximizer $\zeta^\varepsilon\in \Sigma_\ep$. Then by Lemma \ref{lem3-16}, $\zeta^\varepsilon$ has the representation
 	\begin{equation*}
 		\zeta^\varepsilon(x)=g_\varepsilon(\mathcal{G}_1\zeta^\varepsilon(x)-\frac{W}{2}x_1^2\ln\frac{1}{\ep} -\tilde\mu_\varepsilon),\ \ \forall\,x\in \mathbb{R}_+^2
 	\end{equation*}
 	for some constant $\tilde\mu_\varepsilon \geq 0$.
 	Set $\psi^\ep:=\mathcal{G}_1\zeta^\varepsilon(x)-\frac{W}{2}x_1^2\ln\frac{1}{\ep} -\tilde\mu_\varepsilon$. Then we have
 	\begin{equation}\label{3-26}
 		-\text{div}\left(\frac{1}{x_1}\nabla \psi^\ep\right)=x_1\zeta^\ep.
 	\end{equation}
 Note that it holds
 	\begin{equation}\label{e-mu}
 		2\mathcal{E}_\ep(\zeta^\ep)= \int_{\mathbb{R}_+^2}{\zeta^\ep \mathcal{G}_1\zeta^\ep}d\nu-W\log\frac{1}{\varepsilon}\int_{\mathbb{R}_+^2}x_1^2\zeta^\ep d\nu=\int_{\mathbb{R}_+^2}\zeta^\ep (\psi^\ep)_+d \nu-\frac{W}{2}\log\frac{1}{\varepsilon}\int_{\mathbb{R}_+^2}x_1^2\zeta^\ep d\nu + \kappa\tilde \mu_\ep.
 	\end{equation}
 	Since $\mathcal{G}_1\zeta^\varepsilon(x)-\frac{W}{2}x_1^2\ln\frac{1}{\ep}=0$ on $\partial \mathbb{R}_+^2$, we can multiply  the equation \eqref{3-26} by $(\psi^\ep)_+$ and use integration by part to obtain
 	\begin{equation*}
 		\int_{\mathbb{R}_+^2}\zeta^\ep (\psi^\ep)_+d \nu=\int_{\mathbb{R}_+^2}\frac{|\nabla (\psi^\ep)_+|^2}{x_1^2} d\nu.
 	\end{equation*}
 	On the other hand, since $0\leq \zeta_\ep(y)\leq C\ep^{-2}$ and $\nu(\text{supp}( \zeta_\ep))\leq C\ep^2$ and $ \zeta^\ep$ is a rearrangement of $ \zeta_\ep $ with respect to the  measure $\nu$, we have  $0\leq  \zeta^\ep\leq C\ep^{-2}$ and $\nu (\text{supp}( \zeta^\ep))\leq C\ep^2$. Thus, by Lemma \ref{lem3-17}, H\"older's inequality and the Sobolev embedding, we calculate
 	\begin{align*}
 			&\int_{\mathbb{R}_+^2}\zeta^\ep(x) (\psi^\ep)_+(x)d \nu
 			\leq  C\ep^{-2} \int_{\mathbb{R}_+^2} (\psi^\ep)_+(x) d \nu
 			\leq  C\ep^{-2} \nu(\text{supp}( \zeta^\ep))^{\frac{1}{2}} \left(\int_{\mathbb{R}_+^2} (\psi^\ep)_+^2(x) d \nu\right)^{\frac{1}{2}}\\
 			\leq &CR_1^{\frac{1}{2}}\ep^{-2} \nu(\text{supp}( \zeta^\ep))^{\frac{1}{2}} \left(\int_{\mathbb{R}_+^2} (\psi^\ep)_+^2(x) d x\right)^{\frac{1}{2}}
 			\leq  CR_1^{\frac{1}{2}}\ep^{-2} \nu(\text{supp}( \zeta^\ep))^{\frac{1}{2}}  \int_{\mathbb{R}_+^2} |\nabla (\psi^\ep _+)| d x \\
 			=&CR_1^{\frac{1}{2}}\ep^{-2} \nu(\text{supp}( \zeta^\ep))^{\frac{1}{2}}  \int_{\mathbb{R}_+^2} \frac{|\nabla (\psi^\ep _+)| }{x_1}d \nu
 			\leq CR_1^{\frac{1}{2}}\ep^{-2} \nu(\text{supp}( \zeta^\ep))  \left(\int_{\mathbb{R}_+^2} \frac{|\nabla (\psi^\ep _+)| ^2}{x_1^2} d \nu\right)^{\frac{1}{2}}\\ \leq  &C\left(\int_{\mathbb{R}_+^2} \frac{|\nabla (\psi^\ep _+)| ^2}{x_1^2} d \nu\right)^{\frac{1}{2}}.
 		\end{align*}
 	Hence, we derive
 	$$\int_{\mathbb{R}_+^2}\frac{|\nabla (\psi^\ep)_+|^2}{x_1^2} d\nu=\int_{\mathbb{R}_+^2}\zeta^\ep(x) (\psi^\ep)_+(x)d \nu\leq C\left(\int_{\mathbb{R}_+^2} \frac{|\nabla (\psi^\ep _+)| ^2}{x_1^2} d \nu\right)^{\frac{1}{2}},$$ which implies $$ \int_{\mathbb{R}_+^2}\zeta^\ep (\psi^\ep)_+d \nu=\int_{\mathbb{R}_+^2}\frac{|\nabla (\psi^\ep)_+|^2}{x_1^2} d\nu\leq C.$$
 	So, we obtain from Lemma \ref{lem3-13} and \eqref{e-mu} that
 	\begin{equation*}
 		\tilde \mu_\ep=\frac{2}{\kappa}\mathcal{E}_\ep(\zeta^\ep)+\frac{W}{2\kappa}\log\frac{1}{\varepsilon}\int_{\mathbb{R}_+^2}x_1^2\zeta^\ep d\nu-C\geq  \left(\frac{\kappa r_*}{2\pi}-   Wr_*^2 \right)\log{\frac{1}{\varepsilon}}+\frac{W}{2\kappa}\log\frac{1}{\varepsilon}\int_{\mathbb{R}_+^2}x_1^2\zeta^\ep d\nu+O_\ep(1),
 	\end{equation*}
 	which completes the proof.
 \end{proof}

 Now we estimate the bound of the support of a maximizer in the $x_2$-direction.

 \begin{lemma}\label{lem3-19}
 	There exists a constant $Z_1>0$ independent of $\ep$ such that  up to translations in the $x_2$ direction,
 	$$\mathrm{supp}(\zeta^\ep)\subset [0, R_1]\times [-Z_1, Z_1],\quad \forall \ \zeta^\ep \in  \Sigma_\ep.$$
 \end{lemma}
 \begin{proof}
 	Take a maximizer $\zeta^\ep \in  \Sigma_\ep$. Note that $\zeta^\ep$ is Steiner symmetric about the $x_1$ axis by Lemma \ref{lem3-16}. We use this property to estimate $\mathcal{G}_1\zeta^\ep$. It is easy to see that for $x_2\neq 0$,
 	\begin{equation}\label{3-27}
 		\int_0^\infty \zeta^\ep(x_1,x_2)x_1 dx_1 \leq \frac{1}{ |x_2|} \int_{0}^{|x_2|} \int_0^\infty \zeta^\ep(x_1,y_2)x_1 dx_1 dy_2\leq \frac{1}{ |x_2|} \vertiii{\zeta^\ep}_1=\frac{\kappa}{ |x_2|}.
 	\end{equation}
 	For $|x_2|\geq 2$, by the inequality \eqref{2-9} we calculate
 	\begin{align*}
 			\mathcal{G}_1 \zeta^\ep(x)&=\int_{|x_2-y_2|>\frac{|x_2|}{2} } G_1(x,y) \zeta^\ep(y)d\nu+\int_{|x_2-y_2|\leq\frac{|x_2|}{2}} G_1(x,y) \zeta^\ep(y)d\nu\\
 			&\leq C\int_{|x_2-y_2|>\frac{|x_2|}{2} } \frac{x_1^2 y_1^2}{|x_2-y_2|^3} \zeta^\ep(y)d\nu +C\int_{|x_2-y_2|\leq\frac{|x_2|}{2}} \frac{x_1^{\frac{3}{4}} y_1^{\frac{3}{4}}}{|x_2-y_2|^{\frac{1}{2}}} \zeta^\ep(y)d\nu\\
 			&\leq \frac{C\kappa R_1^4}{|x_2|^3}+CR_1^{\frac{3}{2}} \int_{|x_2-y_2|\leq\frac{|x_2|}{2}} \frac{1}{|x_2-y_2|^{\frac{1}{2}}} \zeta^\ep(y)d\nu\\
 			&\leq \frac{C \kappa R_1^4}{|x_2|^3}+CR_1^{\frac{3}{2}} \int_0^{\frac{|x_2|}{2}}\frac{1}{|y_2|^{\frac{1}{2}}}\int_{0}^\infty  \left(\zeta^\ep(y_1, x_2-y_2)+\zeta^\ep(y_1, x_2+y_2)\right)y_1 dy_1 dy_2\\
 			&\leq \frac{C \kappa R_1^4}{|x_2|^3}+\frac{\kappa C R_1^{\frac{3}{2}}}{2} \int_0^{\frac{|x_2|}{2}}\frac{1}{|y_2|^{\frac{1}{2}}} \left(\frac{1}{|x_2-y_2|}+\frac{1}{|x_2+y_2|}\right) dy_2\\
 			&\leq C \left( \frac{1}{|x_2|^3}+\frac{1}{|x_2|^{\frac{1}{2}}}\right)\leq \frac{C}{|x_2|^{\frac{1}{2}}}.
 		\end{align*}
 	Therefore, for $x\in \mathbb{R}_+^2$ such that $|x_2|\geq 2$ and $\zeta^\ep(x)>0$, it holds
 	\begin{equation*}
 		0\leq \mathcal{G}_1\zeta^\varepsilon(x)- \frac{W}{2}x_1^2 \ln \frac{1}{\ep}-\tilde\mu_\varepsilon\leq \frac{C}{|x_2|^{\frac{1}{2}}}- \tilde\mu_\varepsilon\leq \frac{C}{|x_2|^{\frac{1}{2}}}-\left(\frac{\kappa r_*}{2\pi}-   Wr_*^2 \right)\log{\frac{1}{\varepsilon}}-\frac{W}{2\kappa}\log\frac{1}{\varepsilon}\int_{\mathbb{R}_+^2}x_1^2\zeta^\ep d\nu+O_\ep(1),
 	\end{equation*}
 	which implies $|x_2| \leq \max\{2, o_\ep(1)\}$. Therefore, for $0<\ep<\ep_0$, we can take a uniform constant $Z_1>0$ such that
 	$|x_2|<Z_1$ and finish the proof.
 \end{proof}

 Following the notations of \cite{CWWZ},  we set
 \begin{equation*}
 	\Gamma_1(t)=\frac{\kappa t}{2\pi}-Wt^2,\ t\in[ 0, +\infty).
 \end{equation*}
 Recall $ r_*=\frac{\kappa}{4\pi W} $.
 It is easy to check that $\Gamma_1(r_*)=\max_{t\in[ 0, +\infty)}\Gamma_1(t)$ and $r_*$ is the unique maximizer.
 Let $\zeta^\ep\in  \Sigma_\ep$ be a maximizer that is even in $x_2$ and
 set
 \begin{equation}\label{3-28}
 	A_\varepsilon=\inf\{x_1\mid(x_1,0)\in {supp}(\zeta^{{\varepsilon}})\},
 \end{equation}
 \begin{equation}\label{3-29}
 	B_\varepsilon=\sup\{x_1\mid(x_1,0)\in {supp}(\zeta^{{\varepsilon}})\}.
 \end{equation}
 We denote $D=( 0, R_1) \times  (-Z_1, Z_1) $.
 \begin{lemma}\label{lem3-20} It holds
 	$\lim_{\varepsilon\to 0^+}A_\varepsilon=r_*$.
 \end{lemma}

 \begin{proof}
 	Let $\rho$ be defined by $\eqref{2-6}$ and $\gamma\in(0,1)$. By Lemma \ref{lem3-16}, for any  $x=(x_1,x_2)\in \text{supp}(\zeta^{{\varepsilon}})$, we have
 	\begin{equation}\label{3-30}
 		\mathcal{G}_1\zeta^{{\varepsilon}}(x)-\frac{W x_1^2}{2}\log{\frac{1}{\varepsilon}}\ge \tilde \mu^{{\varepsilon}}.
 	\end{equation}
 	By the definition of $ \mathcal{G}_1 $, we derive
 	\begin{align*}
 			\mathcal{G}_1\zeta^{{\varepsilon}}(x)&= \int_{D}G_1(x,y)\zeta^{{\varepsilon}}(y)y_ 1  dy=\big(\int_{D\cap\{\rho>\varepsilon^\gamma\}}+\int_{D\cap\{\rho \le \varepsilon^\gamma\}}\big)G_1(x,y)\zeta^{{\varepsilon}} (y)y_ 1  dy\\
 			&:=E_1+E_2.
 		\end{align*}
 	On the one hand, by the estimate from Lemma 3.4 in \cite{CWWZ}
 	$$ 0<G(x,y)\leq\frac{(x_1 y_1)^\frac{1}{2}}{4\pi}\sinh^{-1}\left(\frac{1}{\rho}\right), \ \ \forall\,\rho >0,$$ we obtain
 	\begin{align}\label{3-31}
 			E_1 &=\int\limits_{D\cap\{\rho>\varepsilon^\gamma\}}G_1(x,y)\zeta^{{\varepsilon}}(y)y_ 1  dy\le \frac{(x_1)^{\frac{1}{2}}}{4\pi} \sinh^{-1}(\frac{1}{\varepsilon^\gamma})\int\limits_{D\cap\{\rho>\varepsilon^\gamma\}}\zeta^{{\varepsilon}}(y)y_1^{{3}/{2}}  dy  \\
 			&\le \frac{R_1}{4\pi} \sinh^{-1}(\frac{1}{\varepsilon^\gamma})\int\limits_{D\cap\{\rho>\varepsilon^\gamma\}}\zeta^{\varepsilon}(y)y_ 1  dy \le \frac{\kappa R_1}{4\pi} \sinh^{-1}(\frac{1}{\varepsilon^\gamma}).\nonumber
 	\end{align}
 	On the other hand, by the definition of $\rho$ we have $D\cap\{\rho \le \varepsilon^\gamma\} \subseteq B_{ R_1\varepsilon^\gamma} (x) $. Indeed, for any $ y\in D\cap\{\rho \le \varepsilon^\gamma\} $, we have $  \rho=[(x_1-y_1)^2+(x_2-y_2)^2]^\frac{1}{2}/(x_1 y_1)^\frac{1}{2}<\varepsilon^\gamma, $ which means that  $ [(x_1-y_1)^2+(x_2-y_2)^2]^\frac{1}{2}\le   (x_1 y_1)^\frac{1}{2}\varepsilon^\gamma\le R_1\varepsilon^\gamma $. Hence using \eqref{2-7}, we get
 	\begin{align*}
 			E_2 &=\int\limits_{D\cap\{\rho \le \varepsilon^\gamma\}}G_1(x,y)\zeta^{{\varepsilon}}(y)y_1  dy\\
 			&\le \frac{ x_1 +C\varepsilon^\gamma}{2\pi}\int\limits_{D\cap\{\rho \le \varepsilon^\gamma\}}\log [(x_1-y_1)^2+(x_2-y_2)^2]^{-\frac{1}{2}}\zeta^{\varepsilon}(y)y_1 dy+C.
 		\end{align*}
 	Notice that $ \log\frac{1}{|x|} $ is a radially symmetric monotone decreasing function, and $ y_1\zeta^{\varepsilon}(y)\le \frac{C}{\varepsilon^2}  $. We choose a radially symmetric step function $ \zeta^{\varepsilon}_*(y)=\frac{C}{\varepsilon^2}\chi_{B_{ {r}_\varepsilon}(x)} $ such that $ \int \zeta^{\varepsilon}_*(y)dy=\int\limits_{D\cap\{\rho \le \varepsilon^\gamma\}}\zeta^{\varepsilon}(y)y_1 dy $. Then by  the rearrangement inequality and direct calculations we have
 	\begin{align}\label{3-32}
 			E_2 &\le \frac{x_1+C\varepsilon^\gamma}{2\pi}\int_{\mathbb{R}_+^2}\log [(x_1-y_1)^2+(x_2-y_2)^2]^{-\frac{1}{2}}\zeta_*^{\varepsilon}(y)dy+C\\
 			&\le \frac{x_1+C\varepsilon^\gamma}{2\pi}\log\frac{1}{\varepsilon} \int\limits_{D\cap\{\rho \le \varepsilon^\gamma\}}\zeta^{\varepsilon}(y) y_1 dy+C  \frac{x_1}{2\pi}\log\frac{1}{\varepsilon} \int\limits_{B_{R_1\varepsilon^\gamma} (x) }\zeta^{\varepsilon} d\nu+C.\nonumber
 	\end{align}
 	
 	Therefore by Lemma \ref{lem3-13}, \eqref{3-31} and \eqref{3-32}, we conclude that there holds
 	\begin{align*}
 			\frac{x_1}{2\pi}\log\frac{1}{\varepsilon} \int\limits_{B_{R_1\varepsilon^\gamma} (x) }\zeta^{\varepsilon} d\nu&+\frac{\kappa R_1}{4\pi} \sinh^{-1}(\frac{1}{\varepsilon^\gamma})-\frac{Wx_1^2}{2}\log{\frac{1}{\varepsilon}} \\
 			& \ge \left(\frac{\kappa r_*}{2\pi}-   Wr_*^2 \right)\log{\frac{1}{\varepsilon}}+\frac{W}{2\kappa}\log\frac{1}{\varepsilon}\int_{\mathbb{R}_+^2}y_1^2\zeta^\ep(y) d\nu(y)-C.
 		\end{align*}
 	Dividing both sides of the above inequality by $\log\frac{1}{\varepsilon}$, we obtain
 		\begin{align}\label{3-33}
 			\Gamma_1(x_1)&\ge \frac{x_1}{2\pi}\int\limits_{B_{R_1\varepsilon^\gamma} (x) }\zeta^{\varepsilon} d\nu-Wx_1^2 \ge ~ \Gamma_1(r_*)+\frac{W}{2\kappa}\Big{\{}\int_D\zeta^{\varepsilon}(y) y_1^2d\nu(y)-\kappa x_1^2\Big{\}}   \\
 			&-\frac{\kappa R_1}{4\pi} \sinh^{-1}\left(\frac{1}{\varepsilon^\gamma}\right)/\log\frac{1}{\varepsilon}-C/\log\frac{1}{\varepsilon}.\nonumber
 		\end{align}
 	Noting that
 	\begin{equation*}
 		\int_D\zeta^{\varepsilon}(y) y_1^2d\nu(y) \ge \kappa(A_\varepsilon)^2,
 	\end{equation*}
 	taking $x=(A_\varepsilon,0)$ and letting $\varepsilon$ tend to $0^+$, we deduce from $\eqref{3-33}$ that
 	\begin{equation}\label{3-34}
 		\liminf_{\varepsilon\to 0^+}\Gamma_1(A_\varepsilon)\ge \Gamma_1(r^*)-\gamma \kappa R_1/(2\pi).
 	\end{equation}
 	Hence we get the desired result by letting   $\gamma \to 0$ since $r_*$ is the unique maximizer of the function $\Gamma_1$.
 \end{proof}

 Next, we estimate the impulse of the flow.
 \begin{lemma} \label{lem3-21}
 	As $\varepsilon\to 0^+$, one has
 	\begin{equation}\label{3-35}
 		\int_D\zeta^{\varepsilon}(y) y_1^2d\nu(y)\to \kappa r_*^2.
 	\end{equation}
 	As a consequence, for any $\eta>0$, there holds
 	\begin{equation}\label{3-36}
 		\lim_{\varepsilon\to 0^+}\int_{D\cap\{y_1\ge r_*+\eta\}}\zeta^{\varepsilon}(y)  d\nu =0.
 	\end{equation}
 \end{lemma}

 \begin{proof}
 	From $\eqref{3-33}$, we know that for any $\gamma\in(0,1)$,
 	\begin{equation*}
 		0\le \liminf_{\varepsilon\to 0^+}\Big[\int_D\zeta^{\varepsilon}(y) y_1^2d\nu(y) -\kappa(A_\varepsilon)^2\Big]\le \limsup_{\varepsilon\to 0^+}\Big[\int_D\zeta^{\varepsilon} (y) y_1^2d\nu(y)-\kappa(A_\varepsilon)^2\Big]\le \frac{\kappa^2 R_1 \gamma}{2\pi W}.
 	\end{equation*}
 	Combining this with Lemma \ref{lem3-20} we get
 	\begin{equation*}
 		\lim_{\varepsilon\to 0^+}\int_D\zeta^{\varepsilon} (y) y_1^2d\nu(y)=\lim_{\varepsilon\to 0^+}\kappa(A_\varepsilon)^2 =\kappa r_*^2.
 	\end{equation*}
 	By  $\eqref{3-35}$ and a simple calculation, we immediately get \eqref{3-36}.
 \end{proof}

 Using this result, we can   get estimates of $ B_\varepsilon. $
 \begin{lemma}\label{lem3-22}
 	$\lim_{\varepsilon\to 0^+}B_\varepsilon=r_*$.
 \end{lemma}

 \begin{proof}
 	Clearly, $ \liminf_{\varepsilon\to 0^+}B_\varepsilon\geq\lim_{\varepsilon\to 0^+}A_\varepsilon =r_*  $. Taking $x=(B_\varepsilon,0)$ in \eqref{3-33}, noting that $0<r_*<R_1$,  we obtain from $\eqref{3-33}$
 	\begin{equation*}
 		\begin{split}
 			\frac{R_1}{2\pi}\liminf_{\varepsilon\to 0^+}\int\limits_{B_{R_1\varepsilon^\gamma}\big((B_\varepsilon,0)\big)}\zeta^{\varepsilon} d\nu  & \ge \Gamma_1(r_*)+\frac{ W}{2}\liminf_{\varepsilon\to 0^+}(B_\varepsilon)^2+\frac{ Wr_*^2}{2}-\frac{\kappa d\gamma}{4\pi} \\
 			& \ge \frac{\kappa r_*}{2\pi}-\frac{\kappa R_1\gamma}{4\pi}.
 		\end{split}
 	\end{equation*}
 	Hence
 	\begin{equation*}
 		\liminf_{\varepsilon\to 0^+}\int\limits_{B_{R_1\varepsilon^\gamma}\big((B_\varepsilon,0)\big)}\zeta^{\varepsilon} d\nu\ge\frac{\kappa r_*}{R_1}-\frac{\gamma\kappa }{2}.
 	\end{equation*}
 	The desired result clearly follows from Lemma $\ref{lem3-21}$ by taking $\gamma$ so small that $r_*/R_1-\gamma/2>0$.
 \end{proof}

 Having obtained that $\lim_{\varepsilon\to 0^+}A_\varepsilon=\lim_{\varepsilon\to 0^+}B_\varepsilon=r_*$, we next show that the  diameter of the support set of $ \zeta^{\varepsilon} $ can be bounded by $ \varepsilon^{\gamma} $ for any $ \gamma\in(0, 1) $.
 \begin{lemma}\label{lem3-23}
 	For any number $\gamma\in(0, 1)$, there holds $$\mathrm{diam}\big(\mathrm{supp}(\zeta^{\varepsilon})\big) \le 2R_1 \varepsilon^{\gamma}$$ provided $\varepsilon$ is small enough. As a consequence, we arrive at $$\lim_{\varepsilon\to 0^+}\mathrm{dist}\big(\mathrm{supp}(\zeta^{\varepsilon}), (r_*,0)\big)=0.$$
 \end{lemma}

 \begin{proof}
 	Let us use the same notation as in the proof of Lemma $\ref{lem3-20}$. Since $\int_{D}\zeta^{\varepsilon} d\nu=\kappa, $ it suffices to prove that (see e.g. \cite{CWZ,T83})
 	\begin{equation}\label{3-37}
 		\int\limits_{B_{R_1\varepsilon^\gamma}(x)}\zeta^{\varepsilon}(y) d\nu(y)>\kappa/2, \ \ \ \forall\, x\in \text{supp}(\zeta^{\varepsilon}).
 	\end{equation}
 	For any $x=(x_1,x_2)\in \text{supp}(\zeta^{\varepsilon})$, from Lemma \ref{lem3-20} and  Lemma \ref{lem3-22} we know that $x_1\to r_*$  as $\varepsilon \to 0^+$.
 	Taking this into $\eqref{3-33}$ and by a direct calculation, we get
 	\begin{equation}\label{3-38}
 		\liminf_{\varepsilon\to 0^+}\int\limits_{B_{R_1\varepsilon^\gamma}(x)}\zeta^{\varepsilon} d\nu \ge \kappa-\frac{\kappa R_1\gamma}{r_*},
 	\end{equation}
 	which implies \eqref{3-37} for all small $\gamma$ such that $1-\frac{R_1\gamma}{r_*}>1/2$. Then we get $ \text{diam}\big(\text{supp}(\zeta^{\varepsilon})\big) \le 2R_1 \varepsilon^{\gamma} $ for small $ \gamma $, which in particular implies $$\text{diam}\big(\text{supp}(\zeta^{\varepsilon})\big)\le C/\log\frac{1}{\varepsilon}$$ provided $\varepsilon$ is small enough. Thus \eqref{3-31}  can be improved    as follows
 	\begin{align*}
 			E_1 &=\int\limits_{D\cap\{\rho>\varepsilon^\gamma\}}G_1(x,y)\zeta^{{\varepsilon}}(y)y_ 1  dy \le \frac{x_1^{\frac{1}{2}}}{4\pi} \sinh^{-1}\left(\frac{1}{\varepsilon^\gamma}\right)\int\limits_{D\cap\{\rho>\varepsilon^\gamma\}}\zeta^{\varepsilon}(y)y_1^{\frac{3}{2}} dy \\
 			&\le \frac{\kappa x_1}{4\pi} \sinh^{-1}\left(\frac{1}{\varepsilon^\gamma}\right)+C.
 		\end{align*}
 	Repeating the proof of \eqref{3-31}, we can sharpen $\eqref{3-38}$ as follows
 	\begin{equation}\label{3-39}
 		\liminf_{\varepsilon\to 0^+}\int\limits_{B_{R_1\varepsilon^\gamma}(x)}\zeta^{\varepsilon}d\nu \ge \kappa-\frac{\gamma\kappa}{2},
 	\end{equation}
 	which implies that for any $ \gamma\in (0,1) $ and $ \varepsilon $ sufficiently small, $$ \int\limits_{B_{R_1\varepsilon^\gamma}(x)}\zeta^{\varepsilon}d\nu > \frac{\kappa}{2}. $$ So we get $ \text{diam}\big(\text{supp}(\zeta^{\varepsilon})\big) \le 2R_1 \varepsilon^{\gamma}. $
 	
 	Since $ \zeta^{\varepsilon}$ is Steiner symmetric about the $x_1$-axis, and 	$\lim_{\varepsilon\to 0^+}A_\varepsilon=\lim_{\varepsilon\to 0^+}B_\varepsilon=r_*$, we have
 	$$\lim_{\varepsilon\to 0^+}dist\big({supp}(\zeta^{\varepsilon}), (r_*,0)\big)=0.$$
 	The proof is hence complete.\end{proof}

  We further prove that the diameter of $ \text{supp}(\zeta^\varepsilon) $ can be bounded by $ \varepsilon $.

 \begin{lemma}\label{lem3-24}
 	There exist  $0<C_3<C_4<\infty$ not depending on $\varepsilon$ such that
 	\begin{equation*}
 		C_3 \varepsilon\leq \mathrm{diam}(\mathrm{supp}(\zeta^\varepsilon)) \le C_4 \varepsilon,
 	\end{equation*}
 	provided $\varepsilon$ is small enough.
 \end{lemma}
 \begin{proof}
 	Since $\int_D \zeta^{\varepsilon} d\nu =\kappa$ and $ \zeta^{\varepsilon}\leq {C}/{\varepsilon^2}$, we get $\kappa \le C \varepsilon^{-2} \left(\text{diam}(\text{supp}(\zeta^{\varepsilon}))\right)^2$, which implies that there exists a constant $ C_3>0 $ independent of $ \varepsilon $, such that $$ \mathrm{diam}(\mathrm{supp}(\zeta^{\varepsilon}))\geq  C_3\varepsilon. $$
 	
 	Then it suffices to prove $ \text{diam}(\text{supp}(\zeta^\varepsilon)) \le C_4 \varepsilon $ for some $C_4>0$.  Let us use the same notation as in the proof of Lemma $\ref{lem3-23}$. Recalling that for any $ x\in \text{supp}(\zeta^\varepsilon) $, we have
 	\begin{equation}\label{3-40}
 		\mathcal{G}_1\zeta^{{\varepsilon}}(x)-\frac{W x_1^2}{2}\log{\frac{1}{\varepsilon}}\ge \tilde \mu^{{\varepsilon}}.
 	\end{equation}
 	According to Lemma $\ref{lem3-23}$, there holds
 	\begin{equation*}
 		\text{supp}(\zeta^\varepsilon)\subseteq B_{2R_1\varepsilon^{\frac{1}{2}}}\big((A_\varepsilon,0)\big)
 	\end{equation*}
 	for  $\varepsilon$ sufficiently small.
 	
 	Let $R>1$ be a positive number to be determined. On the one hand, we have
 	\begin{align*}
 			\mathcal{G}_1\zeta^{\varepsilon}(x)&= \int_{D}G_1(x,y)\zeta^{{\varepsilon}}(y)d\nu =\left(\int_{D\cap\{\rho>R\varepsilon\}}+\int_{D\cap\{\rho \le R \varepsilon\}}\right)G_1(x,y)\zeta^{{\varepsilon}}(y)d\nu\\
 			&:=F_1+F_2.
 		\end{align*}
 	By \eqref{2-7}   and Lemma \ref{lem3-23}, we get

 		\begin{align}\label{3-41}
 			F_1 &=\int\limits_{D\cap\{\rho>R\varepsilon\}}G_1(x,y)\zeta^{{\varepsilon}}(y)d\nu \le \frac{(A_\varepsilon)^2+O(\varepsilon^{\frac{1}{2}})}{2\pi}\int\limits_{D\cap\{\rho>R\varepsilon\}}\log{\frac{1}{\rho}}\zeta^\varepsilon(y) dy+C\\
 			&\le \frac{(A_\varepsilon)^2+O(\varepsilon^{\frac{1}{2}})}{2\pi}\log{\frac{1}{R\varepsilon}}\int\limits_{D\cap\{\rho>R\varepsilon\}}\zeta^\varepsilon(y) dy+C \le \frac{A_\varepsilon}{2\pi}\log{\frac{1}{R\varepsilon}}\int\limits_{D\cap\{\rho>R\varepsilon\}}\zeta^\varepsilon d\nu+C,\nonumber
 		\end{align}
 	and
 	\begin{align}\label{3-42}
 			F_2 &=\int\limits_{D\cap\{\rho \le R\varepsilon\}}G_1(x,y)\zeta^{{\varepsilon}}(y)d\nu\\
 			& \le \frac{(A_\varepsilon)^2+O(\varepsilon^\frac{1}{2})}{2\pi}\int\limits_{D\cap\{\rho \le R\varepsilon\}}\log\frac{1}{[(x_1-y_1)^2+(x_2-y_2)^2]^{\frac{1}{2}}}\zeta^\varepsilon(y)dy+C\nonumber \\
 			&\le \frac{(A_\varepsilon)^2+O(\varepsilon^\frac{1}{2})}{2\pi}\log\frac{1}{\varepsilon}\int\limits_{D\cap\{\rho \le R\varepsilon\}}\zeta^\varepsilon dy+C  \le \frac{A_\varepsilon}{2\pi}\log\frac{1}{\varepsilon}\int\limits_{D\cap\{\rho\le R\varepsilon\}}\zeta^\varepsilon d\nu+C.\nonumber
 		\end{align}
 	Taking \eqref{3-41} and \eqref{3-42} into \eqref{3-40}, we get
 	\begin{equation}\label{3-43}
 		\frac{A_\varepsilon}{2\pi}\log{\frac{1}{R\varepsilon}}\int\limits_{D\cap\{\rho>R\varepsilon\}}\zeta^\varepsilon d\nu+\frac{A_\varepsilon}{2\pi}\log\frac{1}{\varepsilon}\int\limits_{D\cap\{\rho \le R\varepsilon\}}\zeta^\varepsilon d\nu-\frac{W(A_\varepsilon)^2}{2}\log{\frac{1}{\varepsilon}}+C\ge \tilde\mu^{\varepsilon}.
 	\end{equation}
 	On the other hand, by Lemmas \ref{lem3-13} and the fact that $r_*$ is the maximizer of $\Gamma_1$ one has
 	\begin{equation}\label{3-44}
 		\tilde\mu^\varepsilon\ge\left(\frac{\kappa A_\varepsilon}{2\pi}-\frac{W(A_\varepsilon)^2}{2}\right)\log\frac{1}{\varepsilon}-C.
 	\end{equation}
 	Combining $\eqref{3-43}$ and $\eqref{3-44}$, we obtain
 	\begin{equation*}
 		\frac{\kappa A_\varepsilon}{2\pi}\log\frac{1}{\varepsilon} \le \frac{A_\varepsilon}{2\pi}\log{\frac{1}{R\varepsilon}}\int\limits_{D\cap\{\rho>R\varepsilon\}}\zeta^\varepsilon d\nu+\frac{A_\varepsilon}{2\pi}\log\frac{1}{\varepsilon}\int\limits_{D\cap\{\rho \le R\varepsilon\}}\zeta^\varepsilon d\nu+C.
 	\end{equation*}
 	Hence we have
 	\begin{equation*}
 		\int\limits_{D\cap\{\rho \le R\varepsilon\}}\zeta^\varepsilon d\nu\ge \kappa\left(1-\frac{C}{\log R}\right).
 	\end{equation*}
 	Choosing $R>1$ such that $C(\log R)^{-1}<1/2$, we obtain
 	\begin{equation*}
 		\int\limits_{D\cap{B_{R_1R\varepsilon}(x)}}\zeta^\varepsilon d\nu\ge\int\limits_{D\cap\{\rho \le R\varepsilon\}}\zeta^\varepsilon d\nu> \frac{\kappa}{2}.
 	\end{equation*}
 	Taking $C_4=2R_1R$, we finish the proof of Lemma \ref{lem3-24}.	
 \end{proof}

Now, we are able to prove Theorem \ref{thm3-12}.

\noindent{\bf Proof of Theorem \ref{thm3-12}:}
The proof of Theorem \ref{thm3-12} is a combination of the above lemmas. \qed

\vspace{0.5cm}

\appendix

\section{Essential estimates for the nonlinearity and the free boundary}\label{appB}
In this appendix, we will prove some estimates and statements for free boundary $\partial \Om_\varepsilon$ and the nonlinearity.  In the following, we always assume that $L>0$ is a large fixed constant. Recall that
\begin{equation*}
	\Psi_{\ep, x_\ep, a_{\ep}}(x) =\frac{1}{\ep^2}\int_{\mathbb{R}_+^2} G_1(x,y) \left(U_{\ep,x_\ep, a_\ep}(y)- a_\ep\ln \frac{1}{\ep}\right)_+^p y_1dy,
\end{equation*}
with $ U_{\ep,x_\ep, a_\ep}$ is defined by \eqref{2-19} and \eqref{2-21}.

For $x\in\mathbb{R}^2$ with the first variable $x_1>0$, let
\begin{equation}\label{B-1}
	\begin{split}
		\mathcal F(x)&=(x_1-x_{\ep,1})\left[\frac{1}{2 x_{\ep,1}} U_{\ep, x_\ep,  a_{\ep}}(x)-Wx_{\ep,1}\ln\frac{1}{\ep}+\frac{x_{\ep,1}}{4\pi }\left(\ln (8 x_{\ep,1})-1\right) x_{\ep,1}^{-\frac{2p}{p-1}} \left(\frac{\ep}{s_\ep}\right)^{\frac{2}{p-1}} \Lambda_p\right]\\
		&+\frac{3x_{\ep,1}}{4\pi \ep^2}\int_{\mathbb{R}^2_+} (y_1-x_{\ep,1})\ln \frac{s_\ep}{|x-y|}\left(U_{\ep, x_\ep, a_\ep}(y)- a_\ep\ln \frac{1}{\ep}\right)_+^p dy.
	\end{split}
\end{equation}
Then we have the following estimate for $\Psi_{\ep, x_\ep,  a_{\ep}}$.
\begin{lemma}\label{lemB-1}
	For every $x\in B_{Ls_\ep}(x_\ep)$, it holds
	\begin{equation}\label{B-2}
		\Psi_{\ep, x_\ep,  a_{\ep}}(x)-\frac{W}{2}x_1^2\ln\frac{1}{\ep}-\mu_\ep= U_{\ep, x_\ep,  a_{\ep}}(x)-a_\ep\ln\frac{1}{\ep} + \mathcal F(x)+O(\ep^2 |\ln {\ep}| ),
	\end{equation}
  and
  \begin{equation}\label{B-3}
  	  |\mathcal F(x)|=O(\ep  ).
  \end{equation}
\end{lemma}
\begin{proof}
	By the definition of $\Psi_{ \ep, x_\ep,a_{\ep}}$, it holds
	\begin{align*}
			&\Psi_{\ep, x_\ep, a_{\ep}}=\frac{1}{2\pi\ep^2}\int_{ B_{s_\ep}(  x_{\ep})}x_1^{1/2}y_1^{3/2}\ln\left(\frac{1}{|  x-y|}\right)\left(U_{\ep,x_\ep, a_{\ep}}(y)- a_\ep\ln \frac{1}{\ep}\right)_+^p d  y \\
			&\quad+\frac{1}{4\pi\varepsilon^2}\int_{ B_{s_\ep}(  x_{\ep})}x_1^{1/2}y_1^{3/2}\left(\ln(x_1y_1)+2\ln 8-4+O\left(\rho\ln\frac{1}{\rho}\right)\right) \left(U_{\ep,x_\ep, a_{\ep}}(y)- a_\ep\ln \frac{1}{\ep}\right)_+^p dy\\
			&=\frac{x_{\ep,1}^2}{2\pi\varepsilon^2}\int_{ B_{s_\ep}(  x_{\ep})}\ln\left(\frac{1}{| x-y|}\right)\left(U_{\ep, x_\ep, a_{\ep}}(y)- a_\ep\ln \frac{1}{\ep}\right)_+^p d y\\
			&\quad +\frac{1}{2\pi\varepsilon^2}\int_{ B_{s_\ep}(  x_{\ep})}(x_1^{1/2} y_1^{3/2}-x_{\ep,1}^2)\ln\left(\frac{1}{|x- y|}\right)\left(U_{\ep, x_\ep, a_{\ep}}(y)- a_\ep\ln \frac{1}{\ep}\right)_+^pdy\\
			&\quad+\frac{1}{4\pi\varepsilon^2}\int_{ B_{s_\ep}(  x_{\ep})}x_1^{1/2} y_1^{3/2}\left(\ln(x_1 y_1)+2\ln 8-4\right)\left(U_{\ep, x_\ep, a_{\ep}}(y)- a_\ep\ln \frac{1}{\ep}\right)_+^p d y+O\left(\ep^2\ln\frac{1}{ \ep}\right),
		\end{align*}
	with $\rho$ defined in \eqref{2-6}.
	By the definition of $U_{\ep, x_\ep, a_\ep}$, one can check that $$-\Delta U_{\ep, x_\ep, a_\ep}=\frac{x_{\ep,1}^2}{\ep^2}\left(U_{\ep, x_\ep,a_\ep} - a_\ep\ln \frac{1}{\ep}\right)_+^p.$$ Thus, we have
	\begin{equation*}
		\frac{x_{\ep,1}^2}{2\pi\varepsilon^2}\int_{ B_{s_\ep}(  x_{\ep})}\ln\left(\frac{1}{| x-y|}\right)\left(U_{\ep, x_\ep, a_\ep}(y)- a_\ep\ln \frac{1}{\ep}\right)_+^p d y=U_{x_\ep,\ep, a_\ep}(x).
	\end{equation*}
According to Taylor's formula, we calculate
	\begin{align*}
			&\frac{1}{2\pi\varepsilon^2}\int_{ B_{s_\ep}(  x_{\ep})}(x_1^{1/2} y_1^{3/2}-x_{\ep,1}^2)\ln\left(\frac{1}{|x- y|}\right)\left(U_{\ep,x_\ep, a_\ep}(y)- a_\ep\ln \frac{1}{\ep}\right)_+^pdy\\
			&=\frac{1}{2\pi\varepsilon^2}\int_{ B_{s_\ep}(  x_{\ep})}\left(\left(x_{\ep,1}^{1/2}+\frac{1}{2x_{\ep,1}^{1/2}}(x_1-x_{\ep,1})+O(\ep^2)\right)\left(x_{\ep,1}^{3/2}+\frac{3x_{\ep,1}^{1/2}}{2}(y_1-x_{\ep,1})+O(\ep^2)\right)-x_{\ep,1}^2\right)\\
			&\quad\times\ln\left(\frac{1}{| x-y|}\right)\left(U_{\ep, x_\ep, a_\ep}(y)- a_\ep\ln \frac{1}{\ep}\right)_+^pd y\\
			&=\frac{x_{\ep,1}}{2\pi\varepsilon^2}\int_{B_{s_\ep}(  x_{\ep})}\left(\frac{x_1-x_{\ep,1}}{2}+\frac{3(y_1-x_{\ep,1})}{2}\right)\ln\left(\frac{1}{| x- y|}\right)\left(U_{\ep, x_\ep, a_\ep}(y)- a_\ep\ln \frac{1}{\ep}\right)_+^pd y+O(\varepsilon^2|\ln\varepsilon|)\\
			&=\frac{x_1-x_{\ep,1}}{2x_{\ep,1}}U_{\ep, x_\ep, a_\ep}(x)+\frac{3x_{\ep,1}}{4\pi\varepsilon^2}\int_{B_{s_\ep}(  x_{\ep})}(y_1-x_{\ep,1})\ln\left(\frac{1}{| x- y|}\right)\left(U_{\ep, x_\ep, a_\ep}(y)- a_\ep\ln \frac{1}{\ep}\right)_+^pd y +O(\varepsilon^2|\ln\varepsilon|)\\
			&=\frac{x_1-x_{\ep,1}}{2x_{\ep,1}}U_{\ep, x_\ep, a_\ep}(x)+\frac{3x_{\ep,1}}{4\pi\varepsilon^2}\int_{B_{s_\ep}(  x_{\ep})}(y_1-x_{\ep,1})\ln\left(\frac{s_\ep}{| x- y|}\right)\left(U_{\ep, x_\ep, a_\ep}(y)- a_\ep\ln \frac{1}{\ep}\right)_+^pd y +O(\varepsilon^2|\ln\varepsilon|).
		\end{align*}
	Note that $\int_{ B_{s_\ep}(  x_{\ep})} (y_1-x_{\ep,1})\left(U_{x_\ep,\ep, a}(y)- a_\ep\ln \frac{1}{\ep}\right)_+^p d y=0$ due to the odd symmetry. Then by straightforward calculations, we find
	\begin{align*}
			&	\frac{1}{4\pi\varepsilon^2}\int_{ B_{s_\ep}(  x_{\ep})}x_1^{1/2} y_1^{3/2}\left(\ln(x_1 y_1)+2\ln 8-4\right)\left(U_{\ep, x_\ep, a}(y)- a_\ep\ln \frac{1}{\ep}\right)_+^p d y\\
			&=\frac{1}{4\pi\varepsilon^2}\int_{ B_{s_\ep}(  x_{\ep})} \left(x_{\ep,1}^{1/2}+\frac{1}{2x_{\ep,1}^{1/2}}(x_1-x_{\ep,1})+O(\ep^2)\right)\left(x_{\ep,1}^{3/2}+\frac{3x_{\ep,1}^{1/2}}{2}(y_1-x_{\ep,1})+O(\ep^2)\right)\\
			&\quad \times \left(2\ln x_{\ep,1}+2\ln 8- 4+\frac{1}{x_{\ep,1}}(x_1-x_{\ep,1}+y_1-x_{\ep,1})+O(\ep^2)\right)\left(U_{x_\ep,\ep, a}(y)- a_\ep\ln \frac{1}{\ep}\right)_+^p d y\\
			&=\frac{1}{2\pi}(\ln( 8 x_{\ep,1})-2) x_{\ep,1}^{-\frac{2}{p-1}}\left(\frac{\ep}{s_\ep}\right)^{\frac{2}{p-1}}\Lambda_p+\frac{x_{\ep,1}(x_1-x_{\ep,1})}{4\pi}(\ln( 8 x_{\ep,1})-1)x_{\ep,1}^{-\frac{2p}{p-1}}\left(\frac{\ep}{s_\ep}\right)^{\frac{2}{p-1}}\Lambda_p+O(\ep^2).
		\end{align*}
	Combining all the facts above and the second equation in \eqref{2-21}, we have
	$$\Psi_{\ep, x_\ep,  a_{\ep}}(x)-\frac{W}{2}x_1^2\ln\frac{1}{\ep}-\mu_\ep= U_{\ep, x_\ep,  a_{\ep}}(x)-a_\ep\ln\frac{1}{\ep} + \mathcal F(x)+O\left(\ep^2 \ln\frac{1}{\ep} \right).$$
	Using the definition \eqref{B-1} and \eqref{2-19}, we see that  $\|\mathcal F(x)\|_{L^\infty(B_{Ls_\ep}(x_\ep))}=O(\ep|\ln\ep|  )$.  One can prove  \eqref{B-3} by directly calculations using the definition \eqref{B-1}, \eqref{2-18}, \eqref{2-19}, \eqref{2-21}, \eqref{D-6} and Corollary \ref{cor2-7}. Note that one should take $\bar \phi_\ep=\phi_\ep$ and $\phi_\ep^o=0$ to apply \eqref{D-6} here.
	The proof of Lemma \ref{lemB-1} is thus complete.
\end{proof}

	Thanks to the implicit function theorem, now we can estimate the free boundary $\partial \Om_\varepsilon$.
\begin{lemma}\label{lemB-2}
	Suppose that $ \phi_\ep$ is a function satisfying
	\begin{equation}\label{B-4}
		\|\phi_\ep\|_{L^\infty(\mathbb{R}_+^2)}+\ep\|\nabla \phi_\ep\|_{L^\infty(B_{Ls_\ep}(x_\ep))}=o_\ep(1).
	\end{equation}
	Then the set
	$$ \Gamma_{\ep}=\{y\mid  \Psi_{\ep, x_\ep, a_{\ep}}(s_\ep y+x_\ep)-\frac{W}{2}(s_\ep y_1+x_{\ep,1})^2\ln\frac{1}{\ep}-\mu_\ep+\phi_\ep(s_\ep y+x_\ep)=0\}\cap B_L(0)$$
	is a continuous closed curve in $\mathbb{R}^2$, and can be parameterized as
	\begin{equation}\label{B-5}
		\begin{split}
			\Gamma_{\ep} &=\{(1+ t_{\ep}(\theta))(\cos\theta,\sin\theta)\mid \theta\in[0,2\pi)\}	
		\end{split}
	\end{equation}
	for some function $t_\ep(\theta)$  with the following expansion
	\begin{equation}\label{B-6}
		\begin{split}
			t_{\ep}(\theta)&=\frac{\ln s_{\ep}}{a_{\ep} \ln \ep} \left(  \phi_\ep(s_\ep y_{\ep}+x_\ep)+\mathcal{F} (s_\ep y_{\ep}+x_\ep)\right)  +O\left(  {\ep}^2+\| {\phi}_\ep\|_{L^\infty}^2\right),
		\end{split}
	\end{equation}
	where $y_\ep:=(1+t_\ep(\theta))(\cos\theta, \sin\theta)\in\Gamma_\ep$.
	Moreover, for $y=(1+t)(\cos\theta, \sin\theta)$, it holds
	\begin{equation}\label{B-7}
		\Psi_{\ep, x_\ep, a_{\ep}}(s_\ep y+x_\ep)-\frac{W}{2}(s_\ep y_1+x_{\ep,1})^2\ln\frac{1}{\ep}-\mu_\ep+\phi_\ep(s_\ep y+x_\ep)\begin{cases}
			>0 & \text{if}\  \   \ t<t_{\ep}(\theta),\\
			<0 & \text{if}\  \ \    t>t_{\ep}(\theta).
		\end{cases}
	\end{equation}
\end{lemma}
\begin{proof}
	We first show the existence of such $t_\ep$. Denote $\tilde\psi_\ep(y):=\Psi_{\ep, x_\ep, a_{\ep}}(s_\ep y+x_\ep)-\frac{W}{2}(s_\ep y_1+x_{\ep,1})^2\ln\frac{1}{\ep}-\mu_\ep+\phi_\ep(s_\ep y+x_\ep)$ for convenience. By Lemma  \ref{lemB-1} and the assumption \eqref{B-4}, we have
	\begin{align*}
			\tilde\psi_\ep(y)	=&\Psi_{\ep, x_\ep, a_{\ep}}(s_\ep y+x_\ep)-\frac{W}{2}(s_\ep y_1+x_{\ep,1})^2\ln\frac{1}{\ep}-\mu_\ep+\phi_\ep(s_\ep y+x_\ep)\\
			=&U_{\ep, x_\ep, a_{\ep}}(s_\ep y+x_\ep)-a_\ep\ln\frac{1}{\ep} + \mathcal F(s_\ep y+x_\ep)+O(\ep^2 \ln\frac{1}{\ep} )+\phi_\ep(s_\ep y+x_\ep) \\
			=&U_{x_\ep, \ep, a_{\ep}}(s_{\ep} y+x_{\ep})-a_\ep \ln \frac{1}{\ep} +o_\ep(1).
		\end{align*}
	Using the definition of $U_{\ep, x_\ep, a_{\ep}}$ and the above equation, one can check that  for any fixed small constant $l>0$, $\tilde\psi_\ep(y)>0$ when $|y|<1-l$ and $\tilde\psi_\ep(y)<0$ when $|y|>1-l$ provided $\ep$ sufficiently small. Hence, for each $\theta$,  we can find a $t_\ep(\theta)$ such that $(1+t_\ep(\theta))(\cos\theta, \sin\theta)\in \Gamma_{\ep}$.
	It follows from \eqref{B-4} that
	\begin{equation}\label{B-8}
		\frac{\partial \tilde{\psi}_\ep((1+t )(\cos\theta,\sin\theta))}{\partial t}\Bigg|_{t=0}=x_{\ep,1}^{-\frac{2}{p-1}}\left(\frac{ {\ep}^2}{s_\ep^2}\right)^{\frac{2}{p-1}}\left(  U'(1)+o_\ep(1) \right) <0.
	\end{equation}
	As a result, $t_{\ep}$ is unique by the implicit function theorem and $ \Gamma_{\ep}$ is a continuous closed curve in $\mathbb{R}^2$.  So it is easy to check   that \eqref{B-7} holds.
	
	For any point
	$y_\ep =(1+t_{\ep}(\theta))(\cos\theta,\sin\theta)\in \Gamma_{\ep},$
	we infer from Lemma \ref{lemB-1} that if $|y_{\ep} |>1$, then
	\begin{equation}\label{B-9}
		\begin{split}
			0&=\Psi_{\ep, x_\ep, a_{\ep}}(s_\ep y_\ep+x_\ep)-\frac{W}{2}(s_\ep y_{\ep,1}+x_{\ep,1})^2\ln\frac{1}{\ep}-\mu_\ep+\phi_\ep(s_\ep y_\ep+x_\ep)\\
			=&U_{\ep, x_\ep, a_{\ep}}(s_\ep y_\ep+x_\ep)-a_\ep\ln\frac{1}{\ep} + \mathcal F(s_\ep y_\ep+x_\ep)+\phi_\ep(s_\ep y_\ep+x_\ep)+O(\ep^2 \ln\frac{1}{\ep} )\\
			=&a_\ep\frac{ \ln \ep}{\ln s_\ep} \ln \frac{1}{|y_\ep|}+ \mathcal F(s_\ep y+x_\ep)+\phi_\ep(s_\ep y+x_\ep)+O(\ep^2 \ln\frac{1}{\ep} ).
		\end{split}
	\end{equation}
	Thus, we obtain
	\begin{equation}\label{B-10}
		\begin{split}
			|y_\ep|&=e^{\frac{\ln s_{\ep} \left(\mathcal F(s_\ep y_\ep+x_\ep)+\phi_\ep(s_\ep y_\ep+x_\ep)+O(\ep^2 |\ln\ep| )\right)}{a_{\ep} \ln \ep}}\\
			&=1+\frac{\ln s_{\ep}}{a_{\ep} \ln \ep}  \phi_\ep(s_\ep y_\ep+x_\ep)+\frac{ \ln s_\ep }{a_{\ep} \ln \ep}\mathcal F(s_\ep y_\ep+x_\ep)\\
			&\quad+O\left( \ep^2 |\ln\ep|+\| {\phi}_\ep\|_{L^\infty(B_{Ls_\ep}(x_\ep))}^2+\| \mathcal F\|_{L^\infty(B_{Ls_\ep}(x_\ep))}^2\right).
		\end{split}
	\end{equation}
	If $|y_\ep|<1$, then we find
	\begin{equation}\label{B-11}
		\begin{split}
			0&=U_{\ep, x_\ep, a_{\ep}}(s_\ep y_\ep+x_\ep)-a_\ep\ln\frac{1}{\ep} + \mathcal F(s_\ep y_\ep+x_\ep)+\phi_\ep(s_\ep y_\ep+x_\ep)+O(\ep^2 |\ln\ep| )\\
			&= x_{\ep,1}^{-\frac{2}{p-1}}\left(\frac{\ep }{ s_\ep }\right)^{\frac{2}{p-1}}U(y_\ep)+ \mathcal F(s_\ep y_\ep+x_\ep)+\phi_\ep(s_\ep y_\ep+x_\ep)+O(\ep^2 |\ln\ep| )\\
			&=x_{\ep,1}^{-\frac{2}{p-1}}\left(\frac{\ep }{ s_\ep }\right)^{\frac{2}{p-1}}U'(1)(|y_\ep|-1)+\mathcal F(s_\ep y_\ep+x_\ep)+\phi_\ep(s_\ep y_\ep+x_\ep)\\
			&\quad +O(\ep^2 |\ln\ep|+(|y_\ep|-1)^2),
		\end{split}
	\end{equation}	
 which implies
	\begin{equation}\label{B-12}
		\begin{split}
			|y_\ep|&=1-x_{\ep,1}^{ \frac{2}{p-1}}\left(\frac{s_\ep }{ \ep }\right)^{\frac{2}{p-1}}\frac{1}{U'(1)}\left(\mathcal F(s_\ep y_\ep+x_\ep)+\phi_\ep(s_\ep y_\ep+x_\ep) \right)\\
			&+O\left(\ep^2 |\ln\ep|+\| {\phi}_\ep\|_{L^\infty(B_{Ls_\ep}(x_\ep))}^2+\| \mathcal F\|_{L^\infty(B_{Ls_\ep}(x_\ep))}^2 \right).
		\end{split}
	\end{equation}
	So \eqref{B-6} follows from \eqref{B-10}, \eqref {B-12} and the definition of $s_\ep$ in \eqref{2-21}. The proof is thus finished. 
\end{proof}

\section{Estimates of the location $x_\ep$}\label{appC}
This appendix is devoted to estimate the location of $x_\ep$ by the local Pohozaev identity. Using \eqref{2-3}, \eqref{2-19},  \eqref{D-6} and the odd symmetry, we can calculate each term in \eqref{2-27}. Recall the Pohozaev identity
\begin{align*}
		&\quad-\int_{\partial B_\delta(x_\ep)} \frac{\partial\psi_{1,\varepsilon}}{\partial \nu}\frac{\partial\psi_{1,\varepsilon}}{\partial x_1}ds+ \frac{1}{2}\int_{\partial B_\delta(x_\ep)} |\nabla\psi_{1,\varepsilon}|^2 \nu_1ds\\
		&=-\frac{x_{\ep,1}^2}{\varepsilon^2}\int_{B_\delta(x_\ep)} \left(\psi_\ep(x)-\frac{W}{2}x_1^2\ln\frac{1}{\ep}-\mu_\ep\right)_+^p\partial_1\psi_{2,\varepsilon}(  x)   d x\\
		&\quad+\frac{x_{\ep,1}^2}{\varepsilon^2}\int_{B_\delta(x_\ep)} Wx_1\ln\frac{1}{\varepsilon}\left(\psi_\ep(x)-\frac{W}{2}x_1^2\ln\frac{1}{\ep}-\mu_\ep\right)_+^pd x.
	\end{align*}

By \eqref{B-2}, we have
\begin{align}\label{C-1}
		\psi_\ep(x)-\frac{W}{2}x_1^2\ln\frac{1}{\ep}-\mu_\ep&=\Psi_{\ep, x_\ep,  a_{\ep}}(x) +\phi_\ep(x)-\frac{W}{2}x_1^2\ln\frac{1}{\ep}-\mu_\ep\\
		&= U_{\ep, x_\ep, a_{\ep}}(x) -a_\ep\ln\frac{1}{\ep} + \mathcal F(x)+\phi_\ep(x)+O(\ep^2 \ln\frac{1}{\ep} ),\nonumber
\end{align}
where $\mathcal F(x)$ is defined by \eqref{B-1}.
Note that the function $\mathcal F(x)$ is odd in $x_1-x_{\ep,1}$ in the sense that $$\mathcal F((-x_1,x_2)+x_\ep)=-\mathcal F((x_1,x_2)+x_\ep).$$
Suppose that $\phi_\ep$ takes the following form
\begin{equation}\label{C-2}
	\phi_\ep=\phi_\ep^{o}+\bar \phi_\ep,
\end{equation}
where $\phi_\ep^{o}$ is a function  odd in $x_1-x_{\ep,1}$ in the sense that $$\phi_\ep^{o}((-x_1,x_2)+x_\ep)=-\phi_\ep^{o}((x_1,x_2)+x_\ep).$$
We can use the odd symmetry to eliminate some linear terms, which is the key observation to improve the estimates.
We   proceed a series of lemmas to compute each term in the Pohozaev identity \eqref{2-27}.

Note that all $L^\infty$ norms are taken in the small ball $B_{2s_\ep}(x_\ep)$ since we have $\Om_\ep\subset B_{2s_\ep}(x_\ep)$ by Lemmas \ref{lem2-6} and \ref{lemB-2}. We drop the domain $B_{2s_\ep}(x_\ep)$ in the $\|\cdot\|_{L^\infty(B_{2s_\ep}(x_\ep))}$ and simply denote it as $\|\cdot\|_{L^\infty }$ throughout Appendix \ref{appC}.

To calculate the left hand side of the Pohozaev identity \eqref{2-27}, we need the following estimates for $\psi_{1,\varepsilon}$ far away from $x_\ep$.
\begin{lemma}\label{lemC-1}
	For every $ x\in \mathbb R^2_+\setminus\{  x\ | \ \mathrm{dist}(  x, \Om_\varepsilon)\le L s_\varepsilon\}$, we have
	\begin{equation}\label{C-3}
		\psi_{1,\varepsilon}( x)=\frac{\kappa x_{\ep,1}}{2\pi} \ln \frac{|  x- \overline{ x_\varepsilon}|}{|  x-  x_\varepsilon|}+O\left(\ep(\|\phi_\ep\|_{L^\infty}+\|\mathcal F\|_{L^\infty})+\ep^3 |\ln\ep|+\frac{\ep^2}{|x-x_\ep|}+\frac{\ep^2}{|x-x_\ep|^2}\right),
	\end{equation}
	and
	\begin{equation}\label{C-4}
		\begin{split}
			\nabla \psi_{1,\varepsilon}(  x)=&-\frac{\kappa x_{\ep,1}}{2\pi}  \frac{  x-  x_\varepsilon}{|  x-  x_\varepsilon|^2}+\frac{\kappa x_{\ep,1}}{2\pi} \frac{  x-  \overline{ x_\varepsilon}}{| x- \overline{ x_\varepsilon}|^2}\\
			&+O\left(\ep(\|\phi_\ep\|_{L^\infty}+\|\mathcal F\|_{L^\infty})+\ep^3 |\ln\ep|+\frac{\ep^2}{|x-x_\ep|^2}+\frac{\ep^2}{|x-x_\ep|^3}\right).
		\end{split}
	\end{equation}
\end{lemma}
\begin{proof}
	By the definition of $\psi_{\ep,1}$, it holds
	\begin{align*}
			\psi_{1,\ep}(x)&= \frac{x_{\ep,1}^2}{2\pi }\int_{\mathbb{R}_+^2}\left(\ln \frac{1}{|x-y|}-\ln \frac{1}{|x-\bar{y}|}\right)\zeta_\ep(y)dy \\
			&=\frac{x_{\ep,1}}{2\pi }\int_{\mathbb{R}_+^2}\left(\ln \frac{1}{|x-y|}-\ln \frac{1}{|x-\bar{y}|}\right)y_1\zeta_\ep(y)dy\\
			&+\frac{x_{\ep,1}}{2\pi }\int_{\mathbb{R}_+^2}\left(\ln \frac{1}{|x-y|}-\ln \frac{1}{|x-\bar{y}|}\right)(x_{\ep,1}-y_1)\zeta_\ep(y)dy.
		\end{align*}
For $x\in \mathbb{R}^2_+\setminus \{x: d(x,\Omega_{\ep})\le Ls_{\ep}\}$, using  \eqref{2-3} and  the  expansion
	\begin{equation}\label{C-5}
		|x-y|=|x-x_{\ep}|-\langle\frac{x-x_{\ep}}{|x-x_{\ep}|},y-x_{\ep}\rangle+O\left(\frac{|y-x_{\ep}|^2}{|x-x_{\ep}| }\right), \ \ \forall y\in\Om_{\ep},
	\end{equation}
	we get
	\begin{align*}
			&\int_{\mathbb{R}_+^2} \left(\ln \frac{1}{|x-y|} \right)y_1\zeta_\ep(y)dy
			=  \kappa\ln \frac{1}{|x-x_\ep|}+\int_{\mathbb{R}_+^2} \left(\ln \frac{|x-x_\ep|}{|x-y|} \right)y_1\zeta_\ep(y)dy\\
			=&\kappa\ln \frac{1}{|x-x_\ep|}+\int_{\mathbb{R}_+^2} \left( \frac{(x-x_\ep) \cdot(y-x_\ep) }{|x-x_\ep|^2} +O\left(\frac{\ep^2}{|x-x_\ep|^2}\right)\right)y_1\zeta_\ep(y)dy\\
			=&\kappa\ln \frac{1}{|x-x_\ep|}+x_{\ep,1}\int_{\mathbb{R}_+^2}  \frac{(x-x_\ep) \cdot(y-x_\ep) }{|x-x_\ep|^2}   \zeta_\ep(y)dy+O\left(\frac{\ep^2}{|x-x_\ep|}+\frac{\ep^2}{|x-x_\ep|^2}\right).
		\end{align*}
	By the expansion of $t_+^p$ and the odd symmetry, we find
	\begin{align*}
			&\int_{\mathbb{R}_+^2}  \frac{(x-x_\ep) \cdot(y-x_\ep) }{|x-x_\ep|^2}   \zeta_\ep(y)dy\\
			=&\frac{1}{\ep^2}\int_{\mathbb{R}_+^2}  \frac{(x-x_\ep) \cdot(y-x_\ep) }{|x-x_\ep|^2}   \left(U_{\ep, x_\ep, a_{\ep}}(x)-a_\ep\ln\frac{1}{\ep} + \mathcal F(x)+\phi_\ep(x)+O(\ep^2 |\ln\ep| )\right)_+^pdy\\
			=&\frac{1}{\ep^2}\int_{\mathbb{R}_+^2}  \frac{(x-x_\ep) \cdot(y-x_\ep) }{|x-x_\ep|^2}   \left(U_{\ep, x_\ep,  a_{\ep}}(x)-a_\ep\ln\frac{1}{\ep}\right)_+^pdy+O\left(\ep(\|\phi_\ep\|_{L^\infty}+\|\mathcal F\|_{L^\infty})+\ep^3 |\ln\ep| \right)\\
			=&O\left(\ep(\|\phi_\ep\|_{L^\infty}+\|\mathcal F\|_{L^\infty})+\ep^3 |\ln\ep| \right).
		\end{align*}
	Thus, we obtain
	\begin{align*}
	\int_{\mathbb{R}_+^2} \left(\ln \frac{1}{|x-y|} \right)y_1\zeta_\ep(y)dy	&=\kappa\ln \frac{1}{|x-x_\ep|}\\
		&\quad	+O\left(\ep(\|\phi_\ep\|_{L^\infty}+\|\mathcal F\|_{L^\infty})+\ep^3 |\ln\ep|+\frac{\ep^2}{|x-x_\ep|}+\frac{\ep^2}{|x-x_\ep|^2}\right).
		\end{align*}
	Similarly,
	\begin{align*}
			\int_{\mathbb{R}_+^2} \left(\ln \frac{1}{|x-\bar y|} \right)y_1\zeta_\ep(y)dy&=\kappa\ln \frac{1}{|x-\overline{x_\ep}|}\\
			&\quad+O\left(\ep(\|\phi_\ep\|_{L^\infty}+\|\mathcal F\|_{L^\infty})+\ep^3 |\ln\ep|+\frac{\ep^2}{|x-\overline{x_\ep}|}+\frac{\ep^2}{|x-\overline{x_\ep}|^2}\right).\\
		\end{align*}
	Using the expansion \eqref{C-5} and odd symmetry, one can find
	\begin{align*}
			&\int_{\mathbb{R}_+^2}  \left(\ln \frac{1}{|x-y|}  \right)(x_{\ep,1}-y_1)\zeta_\ep(y)dy\\
			=& \int_{\mathbb{R}_+^2}  \left(\ln \frac{1}{|x-x_\ep|}  \right)(x_{\ep,1}-y_1)\zeta_\ep(y) dy+O\left( \frac{\ep^2}{|x-x_\ep|}+\frac{\ep^2}{|x-x_\ep|^2}\right)\\
			=&O\left(\ep(\|\phi_\ep\|_{L^\infty}+\|\mathcal F\|_{L^\infty})+\ep^3 |\ln\ep|+\frac{\ep^2}{|x-x_\ep|}+\frac{\ep^2}{|x-x_\ep|^2}\right).
		\end{align*}
	We obtain \eqref{C-3} by combining the above calculations.

	The proof of \eqref{C-4} is similar by using the expansion
	\begin{equation*}
		\frac{x -x_{\ep }}{|x-x_{\ep}|^2}-\frac{x -y}{|x-y|^2}=\frac{y -x_{\ep }}{|x-x_{\ep}|^2}+\frac{x -x_{\ep }}{|x-x_{\ep}|^4}(x -x_{\ep })\cdot(y-x_\ep)+O\left(\frac{ \ep^2}{|x-x_{\ep}|^3}\right),
	\end{equation*}
	so we omit it and finish the proof of this lemma.
\end{proof}

Using Lemma \ref{lemC-1}, we can compute the left hand side of \eqref{2-27} as follows.
\begin{lemma}\label{lemC-2}
	It holds
	\begin{equation*}
		-\int_{\partial B_\delta(x_\ep)} \frac{\partial\psi_{1,\varepsilon}}{\partial \nu}\frac{\partial\psi_{1,\varepsilon}}{\partial x_1}ds+ \frac{1}{2}\int_{\partial B_\delta(x_\ep)} |\nabla\psi_{1,\varepsilon}|^2 \nu_1ds=\frac{ x_{\ep,1} \kappa^2}{4\pi }+O\left(\ep(\|\phi_\ep\|_{L^\infty}+\|\mathcal F\|_{L^\infty})+\ep^2\right).
	\end{equation*}
\end{lemma}
\begin{proof}
	We need the following identity for the Green function $G(x,y)=\frac{1}{2\pi}\ln \frac{|x-\bar y|}{|x-y|}$.
	\begin{equation*}
		-\int_{\partial B_\delta(x_\ep)}\frac{G( x_\ep, y)}{\partial \nu}\frac{G(  x_\ep, y)}{\partial x_1}ds+\frac{1}{2}\int_{\partial B_\delta(x_\ep)} |\nabla G(x_\ep, y)|^2 \nu_1 ds
		=\frac{1}{4\pi x_{\ep,1}},
	\end{equation*}
	which can be verified by straightforward calculations. By  the asymptotic estimate in Lemma \ref{lemC-1}, we obtain the conclusion of this lemma  directly.
\end{proof}

Using the circulation constraint \eqref{2-3}, it is obvious that
\begin{equation}\label{C-6}
	\frac{x_{\ep,1}^2}{\varepsilon^2}\int_{B_\delta(x_\ep)} Wx_1\ln\frac{1}{\varepsilon}\left(\psi_\ep(x)-\frac{W}{2}x_1^2\ln\frac{1}{\ep}-\mu_\ep\right)_+^pd x=\kappa Wx_{\ep,1}^2\ln\frac{1}{\varepsilon}.
\end{equation}
Thus we will focus on the first term in the right hand side of \eqref{2-27} relevant to $\partial_1\psi_{2,\varepsilon}$, which is the most difficult term to estimate. We have the following lemma.
\begin{lemma}\label{lemC-3}
	It holds
	\begin{equation*}
		\begin{split}
			&-\frac{x_{\ep,1}^2}{\varepsilon^2}\int_{B_\delta(x_\ep)} \left(\psi_\ep(x)-\frac{W}{2}x_1^2\ln\frac{1}{\ep}-\mu_\ep\right)_+^p\partial_1\psi_{2,\varepsilon}(  x)   d x=- \frac{x_{\ep,1}\kappa^2 }{4\pi}\left(\ln\frac{8x_{\ep,1}}{s_\ep}+\frac{p-5}{4}\right)\\
			+	&O\left(\|\bar{\phi}_\ep\|_{L^\infty}+\ep^2|\ln\ep|+(\|\phi_\ep\|_{L^\infty}+\|\mathcal F\|_{L^\infty})^{2}+\ep(\|\phi_\ep\|_{L^\infty}+\|\mathcal F\|_{L^\infty})  \right).
		\end{split}
	\end{equation*}
\end{lemma}
\begin{proof}
	By the definition of $\partial_1\psi_{2,\varepsilon}$, it holds
	$$\partial_1\psi_{2,\varepsilon}= \int_{\mathbb R^2_+}\partial_{x_1}H(  x, y)\zeta_\ep(y)dy.$$
	Recall the definition of $H(x,y)$ and the following asymptotic behavior in \eqref{2-24} and \eqref{2-25}
	\begin{equation*}
		\begin{split}
			H( x,y)&= \frac{x_1^{1/2}y_1^{3/2}-x_{\ep,1}^2}{2\pi}  \ln\frac{1}{|  x-y|}+\frac{x_{\ep,1}^2}{2\pi}\ln\frac{1}{| x- {\bar y}|}\\
			&\quad+\frac{x_1^{1/2}y_1^{3/2}}{4\pi}\left(\ln(x_1 y_1)+2\ln 8-4+  O(|\rho\ln \rho|)\right).
		\end{split}
	\end{equation*}
	For simplicity, we divide the integral into four parts
	\begin{equation*}
		-\frac{x_{\ep,1}^2}{\varepsilon^2}\int_{B_\delta(x_\ep)} \left(\psi_\ep(x)-\frac{W}{2}x_1^2\ln\frac{1}{\ep}-\mu_\ep\right)_+^p\partial_1\psi_{2,\varepsilon}(  x)   d x=I_1+I_2+I_3+I_4,
	\end{equation*}
	where
	\begin{align*}
		I_1&=-\frac{x_{\ep,1}^2}{4\pi }\ln\left(\frac{1}{s_\ep}\right)\int_{B_\delta(x_\ep)}\int_{\mathbb R^2_+}x_1^{-1/2}y_1^{3/2}\zeta_\ep(x)\zeta_\ep(y)d yd x,\\
		I_2&=-\frac{x_{\ep,1}^2}{4\pi }\int_{B_\delta(x_\ep)}\int_{\mathbb R^2_+}x_1^{-1/2}y_1^{3/2}\zeta_\ep(x)\zeta_\ep(y)\ln\left(\frac{s_\ep}{|  x-y|}\right)dy d  x,\\
		I_3&=\frac{x_{\ep,1}^2}{2\pi }\int_{B_\delta(x_\ep)}\int_{\mathbb R^2_+}\left(x_1^{1/2}y_1^{3/2}-x_{\ep,1}^2\right) \frac{x_1-y_1}{|  x-y|^2} \zeta_\ep(x)\zeta_\ep(y)dy d  x,
	\end{align*}
	and $I_{4}$ is the remaining regular terms.
	
	Let us consider $I_1$ first. Using Taylor's expansion, $I_1$ can be rewritten as
	\begin{align*}
			I_1&=-\frac{x_{\ep,1}^2}{4\pi } \ln\frac{1}{s_\ep} \int_{B_\delta(x_\ep)}x_1 \zeta_\ep(x)\left(x_{\ep,1}^{-3/2}-\frac{3}{2x_{\ep,1}^{5/2}} (x_1-x_{\ep,1})+O(\ep^2)\right)d  x\\
			&\quad \times\int_{\mathbb R^2_+} y_1\zeta_\ep(y)\left(x_{\ep,1}^{1/2}+\frac{1}{2x_{\ep,1}^{1/2}}(y_1-x_{\ep,1})+O(\ep^2)\right)dy\\
			&=-\frac{x_{\ep,1} \kappa^2}{4\pi } \ln\frac{1}{s_\ep}-\frac{ \kappa}{4\pi } \ln\frac{1}{s_\ep}\int_{\mathbb R^2_+} (x_1-x_{\ep,1})x_1 \zeta_\ep(x) dx +O\left(  \varepsilon^2|\ln\varepsilon|\right).
		\end{align*}
	
	By the odd symmetry and the Lipschitz continuity of the function $t_+^p$, it holds
	\begin{align*}
			&\quad\int_{\mathbb R^2_+} (x_1-x_{\ep,1})x_1 \zeta_\ep(x) dx\\
			&=x_{\ep,1}\int_{\mathbb R^2_+} (x_1-x_{\ep,1})  \zeta_\ep(x) dx-\int_{\mathbb R^2_+} (x_1-x_{\ep,1})^2 \zeta_\ep(x) dx\\
			&=\frac{x_{\ep,1}}{\ep^2}\int_{\mathbb R^2_+} (x_1-x_{\ep,1})  \left(U_{\ep, x_\ep, a_{\ep}}(x)-a_\ep\ln\frac{1}{\ep} + \mathcal F(x)+\phi_\ep(x)+O(\ep^2 |\ln\ep| )\right)_+^p dx+O(\varepsilon^2)\\
			&=\frac{x_{\ep,1}}{\ep^2}\int_{\mathbb R^2_+} (x_1-x_{\ep,1})  \left(U_{\ep, x_\ep, a_{\ep}}(x)-a_\ep\ln\frac{1}{\ep}  \right)_+^p dx+O(\ep(\|\phi_\ep\|_{L^\infty}+\|\mathcal F\|_{L^\infty})+\ep^3 |\ln\ep|+\varepsilon^2) \\
			&=O\left(\ep(\|\phi_\ep\|_{L^\infty}+\|\mathcal F\|_{L^\infty})+\ep^2\right).
		\end{align*}
Here we have used $\frac{x_{\ep,1}}{\ep^2}\int_{\mathbb R^2_+} (x_1-x_{\ep,1})  \left(U_{\ep, x_\ep, a_{\ep}}(x)-a_\ep\ln\frac{1}{\ep}  \right)_+^p dx=0$ due to the odd symmetry. Thus we have shown
	\begin{equation}\label{C-7}
		I_1=-\frac{x_{\ep,1} \kappa^2}{4\pi } \ln\frac{1}{s_\ep} +O\left( \ep|\ln\varepsilon|(\|\phi_\ep\|_{L^\infty}+\|\mathcal F\|_{L^\infty})+ \varepsilon^2|\ln\varepsilon|\right).
	\end{equation}

	For the second term $I_2$, we also expand it as
	\begin{align*}
			I_2&=-\frac{x_{\ep,1}^2}{4\pi }\int_{\mathbb R^2_+} \left(x_{\ep,1}^{-1/2}-\frac{1}{2x_{\ep,1}^{3/2}}(x_1-x_{\ep,1})+O(\ep^2)\right)\zeta_\ep(x)\\
			&\quad \times\int_{\mathbb R^2_+}\left(x_{\ep,1}^{3/2}+\frac{3x_{\ep,1}^{1/2}}{2}(y_1-x_{\ep,1})+O(\ep^2)\right)\ln\left(\frac{s_\ep}{|x-y|}\right)\zeta_\ep(y)dyd x\\
			&=-\frac{x_{\ep,1}^3}{4\pi }\int_{\mathbb R^2_+}\int_{\mathbb R^2_+}\ln\left(\frac{s_\ep}{|x-y|}\right)\zeta_\ep(x)\zeta_\ep(y) dy dx\\
			&\quad-\frac{x_{\ep,1}^2}{4\pi }\int_{\mathbb R^2_+}\int_{\mathbb R^2_+}\ln\left(\frac{s_\ep}{|x-y|}\right)(x_1-x_{\ep,1})\zeta_\ep(x)\zeta_\ep(y) dx dy+O(\ep^2).
		\end{align*}
	Recall that $d_\ep=\sup_{\theta\in [0,2\pi)} |t_\ep(\theta)|$ with $t_\ep$, which is defined in \eqref{2-32} and by \eqref{B-6}, $d_\ep=O\left( \|\phi_\ep\|_{L^\infty}+\|\mathcal F\|_{L^\infty}+\ep^2\right)$.
	For $x\in \Om_\ep\setminus B_{(1-d_\ep)s_\ep}(x_\ep)\subset  B_{(1+d_\ep)s_\ep}(x_\ep)\setminus B_{(1-d_\ep)s_\ep}(x_\ep)$, it can be seen $$U_{\ep, x_\ep, a_{\ep}}(x)-a_\ep\ln\frac{1}{\ep}=O(d_\ep), \quad \zeta_\ep(x)=O(d_\ep+\ep^2|\ln\ep|).$$ Using a similar method as we deal with $I_1$ and the calculations in the proof of Proposition \ref{prop2-10}, by \eqref{2-14}, the expansion \eqref{2-31} and the odd symmetry, we compute
	\begin{align}\label{C-8}
			&\frac{1}{2\pi }\int_{\mathbb R^2_+}\int_{\mathbb R^2_+}\ln\left(\frac{s_\ep}{|x-y|}\right)\zeta_\ep(x)\zeta_\ep(y) dy dx\\
			=&\frac{1}{2\pi }\int_{ B_{(1-d_\ep)s_\ep}(x_\ep)}\int_{ B_{(1-d_\ep)s_\ep}(x_\ep)}\ln\left(\frac{s_\ep}{|x-y|}\right)\zeta_\ep(x)\zeta_\ep(y) dy dx\nonumber\\
			&+\frac{1}{ \pi }\int_{ B_{(1-d_\ep)s_\ep}(x_\ep)}\int_{\Om_\ep\setminus B_{(1-d_\ep)s_\ep}(x_\ep)}\ln\left(\frac{s_\ep}{|x-y|}\right)\zeta_\ep(x)\zeta_\ep(y) dy dx\nonumber\\
			&+\frac{1}{2\pi }\int_{ \Om_\ep\setminus B_{(1-d_\ep)s_\ep}(x_\ep)}\int_{\Om_\ep\setminus  B_{(1-d_\ep)s_\ep}(x_\ep)}\ln\left(\frac{s_\ep}{|x-y|}\right)\zeta_\ep(x)\zeta_\ep(y) dy dx\nonumber\\
			=&\frac{1}{2\pi }\int_{ B_{(1-d_\ep)s_\ep}(x_\ep)}\int_{ B_{(1-d_\ep)s_\ep}(x_\ep)}\ln\left(\frac{s_\ep}{|x-y|}\right)\zeta_\ep(x)\zeta_\ep(y) dy dx+O\left(\frac{(d_\ep+\ep^2|\ln\ep|)|\Om_\ep\setminus  B_{(1-d_\ep)s_\ep}(x_\ep)|}{\ep^2}\right)\nonumber\\
			=&\frac{1}{2\pi \ep^4 }\int_{ B_{(1-d_\ep)s_\ep}(x_\ep)}\int_{ B_{(1-d_\ep)s_\ep}(x_\ep)}\ln\left(\frac{s_\ep}{|x-y|}\right)\left(U_{\ep, x_\ep,  a_{\ep}}(x)-a_\ep\ln\frac{1}{\ep} + \mathcal F(x)+\phi_\ep(x)+O(\ep^2 |\ln\ep| )\right)_+^p\nonumber\\
			&\quad \times \left(U_{\ep, x_\ep, a_{\ep}}(y)-a_\ep\ln\frac{1}{\ep} + \mathcal F(y)+\phi_\ep(y)+O(\ep^2 |\ln\ep| )\right)_+^p  dy dx\nonumber\\
			&+O\left( d_\ep(d_\ep+\ep^2|\ln\ep|)  \right)\nonumber\\
			=&\frac{1}{2\pi \ep^4 }\int_{ B_{(1-d_\ep)s_\ep}(x_\ep)}\int_{ B_{(1-d_\ep)s_\ep}(x_\ep)}\ln\left(\frac{s_\ep}{|x-y|}\right)\left(U_{\ep, x_\ep, a_{\ep}}(x)-a_\ep\ln\frac{1}{\ep}  \right)_+^p  \left(U_{\ep, x_\ep,  a_{\ep}}(y)-a_\ep\ln\frac{1}{\ep}  \right)_+^p  dy dx\nonumber\\
			&+\frac{1}{ \pi \ep^4 }\int_{ B_{(1-d_\ep)s_\ep}(x_\ep)}\int_{ B_{(1-d_\ep)s_\ep}(x_\ep)}\ln\left(\frac{s_\ep}{|x-y|}\right)\left(U_{\ep, x_\ep,  a_{\ep}}(x)-a_\ep\ln\frac{1}{\ep}  \right)_+^p \nonumber \\
			&\quad \times\left(\mathcal F(y)+\phi_\ep^{o}(y)+\bar \phi_\ep+O(\ep^2 |\ln\ep| )\right)  dx dy\nonumber\\
			&+O\left((\|\phi_\ep\|_{L^\infty}+\|\mathcal F\|_{L^\infty})^{2}+ d_\ep(d_\ep+\ep^2|\ln\ep|)  \right)\nonumber\\
			=&\frac{1}{2\pi \ep^4 }\int_{ B_{s_\ep}(x_\ep)}\int_{ B_{ s_\ep}(x_\ep)}\ln\left(\frac{s_\ep}{|x-y|}\right)\left(U_{\ep, x_\ep,  a_{\ep}}(x)-a_\ep\ln\frac{1}{\ep}  \right)_+^p  \left(U_{\ep, x_\ep,  a_{\ep}}(y)-a_\ep\ln\frac{1}{\ep}  \right)_+^p  dy dx\nonumber\\
			&+\frac{1}{ \pi \ep^4 }\int_{ B_{(1-d_\ep)s_\ep}(x_\ep)}\int_{ B_{(1-d_\ep)s_\ep}(x_\ep)}\ln\left(\frac{s_\ep}{|x-y|}\right)\left(U_{\ep, x_\ep, a_{\ep}}(x)-a_\ep\ln\frac{1}{\ep}  \right)_+^p   \left(\bar \phi_\ep+O(\ep^2 |\ln\ep| )\right)  dx dy\nonumber\\
			&+O\left((\|\phi_\ep\|_{L^\infty}+\|\mathcal F\|_{L^\infty}+\ep^2|\ln\ep|)^{2}+ d_\ep(d_\ep+\ep^2|\ln\ep|)  \right)\nonumber\\
			=&\frac{s_\ep^4}{2\pi \ep^4 }x_{\ep,1}^{-\frac{4p}{p-1}}\left(\frac{\ep}{s_\ep}\right)^{\frac{4p}{p-1}}\int_{ B_{1}(0)}\int_{ B_{ 1}(0)}\ln\left(\frac{1}{|x-y|}\right)\left(U(x) \right)_+^p  \left(U(y)\right)_+^p  dy dx\nonumber\\
			&+O\left(\|\bar{\phi}_\ep\|_{L^\infty}+\ep^2|\ln\ep|+(\|\phi_\ep\|_{L^\infty}+\|\mathcal F\|_{L^\infty})^{2}+ d_\ep(d_\ep+\ep^2|\ln\ep|)  \right)\nonumber\\
			=&x_{\ep,1}^{-\frac{4p}{p-1}}\left(\frac{\ep}{s_\ep}\right)^{\frac{4}{p-1}}\int_{ B_{1}(0)}  \left(U(x) \right)_+^{p+1} dx\nonumber\\
			&+O\left(\|\bar{\phi}_\ep\|_{L^\infty}+\ep^2|\ln\ep|+(\|\phi_\ep\|_{L^\infty}+\|\mathcal F\|_{L^\infty})^{2}+ d_\ep^2  \right)\nonumber\\
			 =&x_{\ep,1}^{-\frac{4p}{p-1}}\left(\frac{\ep}{s_\ep}\right)^{\frac{4}{p-1}}\frac{(p+1)\Lambda_p^2}{8\pi}+O\left(\|\bar{\phi}_\ep\|_{L^\infty}+\ep^2|\ln\ep|+(\|\phi_\ep\|_{L^\infty}+\|\mathcal F\|_{L^\infty})^{2}   \right)\nonumber
		\end{align}
	Similarly, using the odd symmetry, it is easy to see
	\begin{equation}\label{C-9}
		\begin{split}
			&\frac{x_{\ep,1}^2}{4\pi }\int_{\mathbb R^2_+}\int_{\mathbb R^2_+}\ln\left(\frac{s_\ep}{|x-y|}\right)(x_1-x_{\ep,1})\zeta_\ep(x)\zeta_\ep(y) dx dy\\=&O\left(\ep^3|\ln\ep|+\ep(\|\phi_\ep\|_{L^\infty}+\|\mathcal F\|_{L^\infty})+ \ep d_\ep(d_\ep+\ep^2|\ln\ep|)  \right).
		\end{split}
	\end{equation}
	Thus, we arrive at
	\begin{equation}\label{C-10}
		\begin{split}
			I_2=&-\frac{(p+1)\Lambda_p^2}{16\pi x_{\ep,1}}x_{\ep,1}^{-\frac{4}{p-1}}\left(\frac{\ep}{s_\ep}\right)^{\frac{4}{p-1}}\\
			&+O\left(\|\bar{\phi}_\ep\|_{L^\infty}+\ep^2|\ln\ep|+(\|\phi_\ep\|_{L^\infty}+\|\mathcal F\|_{L^\infty})^{2}+\ep(\|\phi_\ep\|_{L^\infty}+\|\mathcal F\|_{L^\infty})   \right).
		\end{split}
	\end{equation}
	Now we turn to $I_3$. By   virtue of symmetries, we  obtain
	\begin{align}\label{C-11}
			&I_3=\frac{x_{\ep,1}^2}{2\pi }\int_{\mathbb{R}_+^2}\int_{\mathbb{R}_+^2}\bigg(\big(x_{\ep,1}^{1/2}+\frac{1}{2x_{\ep,1}^{1/2}} (x_1-x_{\ep,1})-\frac{1}{4 x_{\ep,1}^{3/2}}(x_1-x_{\ep,1})^2+O(\ep^3)\big)\\
			&\quad\times \big(x_{\ep,1}^{3/2}+\frac{3x_{\ep,1}^{1/2}}{2}(y_1-x_{\ep,1})+\frac{3}{4 x_{\ep,1}^{1/2}}(y_1-x_{\ep,1})^2+O(\ep^3)\big)-x_{\ep,1}^2\bigg) \frac{x_1-y_1}{|x-y|^2}\zeta_\ep(y)\zeta_\ep(x)dy d  x\nonumber\\
			&=\frac{x_{\ep,1}^3}{2\pi }\int_{\mathbb{R}_+^2}\int_{\mathbb{R}_+^2}(y_1-x_{\ep,1})\frac{x_1-y_1}{|x-y|^2}\zeta_\ep(y)\zeta_\ep(x)dy d  x\nonumber\\
			&\quad+\frac{x_{\ep,1}^2}{2\pi }\int_{\mathbb{R}_+^2}\int_{\mathbb{R}_+^2}\left(\frac{3}{4}(x_1-x_{\ep,1})(y_1-x_{\ep,1})+\frac{1}{2}(y_1-x_{\ep,1})^2\right) \frac{x_1-y_1}{|x-y|^2}\zeta_\ep(y)\zeta_\ep(x)dy d  x+O(\ep^2)\nonumber\\
			&=\frac{x_{\ep,1}^3}{4\pi }\int_{\mathbb{R}_+^2}\int_{\mathbb{R}_+^2}(y_1-x_{ 1})\frac{x_1-y_1}{|x-y|^2}\left(U_{\ep, x_\ep,  a_{\ep}}(y)-a_\ep\ln\frac{1}{\ep}  \right)_+^p \left(U_{\ep, x_\ep,  a_{\ep}}(x)-a_\ep\ln\frac{1}{\ep}  \right)_+^p dy d  x\nonumber\\
			&\quad+O\left(\|\bar{\phi}_\ep\|_{L^\infty}+\ep^2|\ln\ep|+(\|\phi_\ep\|_{L^\infty}+\|\mathcal F\|_{L^\infty})^{2}+\ep(\|\phi_\ep\|_{L^\infty}+\|\mathcal F\|_{L^\infty})+ d_\ep^2  \right)\nonumber\\
			&=\frac{x_{\ep,1}^3}{8\pi }\int_{\mathbb{R}_+^2}\int_{\mathbb{R}_+^2} \left(U_{\ep, x_\ep,  a_{\ep}}(y)-a_\ep\ln\frac{1}{\ep}  \right)_+^p \left(U_{\ep, x_\ep,  a_{\ep}}(x)-a_\ep\ln\frac{1}{\ep}  \right)_+^p dy d  x\nonumber\\
			&\quad+O\left(\|\bar{\phi}_\ep\|_{L^\infty}+\ep^2|\ln\ep|+(\|\phi_\ep\|_{L^\infty}+\|\mathcal F\|_{L^\infty})^{2}+\ep(\|\phi_\ep\|_{L^\infty}+\|\mathcal F\|_{L^\infty})+ d_\ep^2  \right)\nonumber\\
			&=\frac{\Lambda_p^2}{8\pi x_{\ep,1}}x_{\ep,1}^{-\frac{4}{p-1}}\left(\frac{\ep}{s_\ep}\right)^{\frac{4}{p-1}}+O\left(\|\bar{\phi}_\ep\|_{L^\infty}+\ep^2|\ln\ep|+(\|\phi_\ep\|_{L^\infty}+\|\mathcal F\|_{L^\infty})^{2}+\ep(\|\phi_\ep\|_{L^\infty}+\|\mathcal F\|_{L^\infty})   \right).\nonumber
	\end{align}
	For the last term $I_{4}$, the computations are easier since it is regular. We have
	\begin{align}\label{C-12}
			I_{4}=&-\frac{x_{\ep,1}^2}{4\pi\varepsilon^4}\int_{\mathbb{R}_+^2} \int_{\mathbb{R}_+^2}\partial_{x_1}\left(\frac{x_{\ep,1}^2}{2\pi}\ln\frac{1}{| x- {\bar y}|} \quad+\frac{x_1^{1/2}y_1^{3/2}}{4\pi}\left(\ln(x_1 y_1)+2\ln 8-4+  O(|\rho\ln \rho|)\right)\right)\nonumber\\
			&\quad\quad\times\zeta_\ep(y)\zeta_\ep(x)dy d  x\nonumber\\
			=&-\frac{\Lambda_p^2}{4\pi x_{\ep,1}}\left(\ln(8x_{\ep,1})-2\right)x_{\ep,1}^{-\frac{4}{p-1}}\left(\frac{\ep}{s_\ep}\right)^{\frac{4}{p-1}}\nonumber\\
			&+O\left(\|\bar{\phi}_\ep\|_{L^\infty}+\ep^2|\ln\ep|+(\|\phi_\ep\|_{L^\infty}+\|\mathcal F\|_{L^\infty})^{2}+\ep(\|\phi_\ep\|_{L^\infty}+\|\mathcal F\|_{L^\infty})   \right).
	\end{align}
	
	Using this odd symmetries, we can provide an identity about $\kappa$.
	\begin{align}\label{D-6}
			\kappa&= \frac{\Lambda_p}{x_{\ep,1}}x_{\ep,1}^{-\frac{2}{p-1}}\left(\frac{\ep}{s_\ep}\right)^{\frac{2}{p-1}} \\
			&+O\left(\|\bar{\phi}_\ep\|_{L^\infty}+\ep^2|\ln\ep|+(\|\phi_\ep\|_{L^\infty}+\|\mathcal F\|_{L^\infty})^{2}+\ep(\|\phi_\ep\|_{L^\infty}+\|\mathcal F\|_{L^\infty})  \right).\nonumber
	\end{align}
	Indeed, by \eqref{2-19}, \eqref{2-30}, \eqref{2-31}, \eqref{2-32},  \eqref{B-2} and the odd symmetry, we have
	\begin{small}
		\begin{align*}
				& \kappa=\int_{\mathbb{R}_+^2} x_1 \zeta_\ep(x)d x\\
				=&\ep^{-2}\int_{\Om_\ep} x_1 \left(U_{\ep, x_\ep,  a_{\ep}}(x)-a_\ep\ln\frac{1}{\ep} + \mathcal F(x)+\phi_\ep(x)+O(\ep^2 |\ln\ep| )\right)_+^p d x\\
				=&\ep^{-2}\left(\int_{B_{(1-d_\ep)s_\ep}(x_\ep)}+\int_{\Om_\ep\setminus B_{(1-d_\ep)s_\ep}(x_\ep)}\right)\left(U_{\ep, x_\ep,  a_{\ep}}(x)-a_\ep\ln\frac{1}{\ep} + \mathcal F(x)+\phi_\ep(x)+O(\ep^2 |\ln\ep| )\right)_+^p d x\\
				=&\ep^{-2}\int_{B_{(1-d_\ep)s_\ep}(x_\ep)} x_{\ep,1} \left(U_{\ep, x_\ep,  a_{\ep}}(x)-a_\ep\ln\frac{1}{\ep} \right)_+^p dx+\ep^{-2}\int_{B_{(1-d_\ep)s_\ep}(x_\ep)} (x_1-x_{\ep,1}) \left(U_{\ep, x_\ep,  a_{\ep}}(x)-a_\ep\ln\frac{1}{\ep} \right)_+^pdx\\
				+&\ep^{-2}\int_{B_{(1-d_\ep)s_\ep}(x_\ep)} x_{\ep,1} \left(\left(U_{\ep, x_\ep,  a_{\ep}}(x)-a_\ep\ln\frac{1}{\ep} + \mathcal F(x)+\phi_\ep(x)+O(\ep^2 |\ln\ep| )\right)_+^p- \left(U_{\ep, x_\ep,  a_{\ep}}(x)-a_\ep\ln\frac{1}{\ep} \right)_+^p \right)dx\\
				+&\ep^{-2}\int_{B_{(1-d_\ep)s_\ep}(x_\ep)} (x_1-x_{\ep,1})\left(\left(U_{\ep, x_\ep,  a_{\ep}}(x)-a_\ep\ln\frac{1}{\ep} + \mathcal F(x)+\phi_\ep(x)+O(\ep^2 |\ln\ep| )\right)_+^p- \left(U_{\ep, x_\ep,  a_{\ep}}(x)-a_\ep\ln\frac{1}{\ep} \right)_+^p\right)dx\\
				+&O(\ep^{-2}d_\ep^p d_\ep s_\ep^2 )\\
				=&\frac{\Lambda_p}{x_{\ep,1}}x_{\ep,1}^{-\frac{2}{p-1}}\left(\frac{\ep}{s_\ep}\right)^{\frac{2}{p-1}}+\frac{p x_{\ep,1}}{\ep^2}\int_{B_{(1-d_\ep)s_\ep}(x_\ep)}\left(U_{\ep, x_\ep,  a_{\ep}}(x)-a_\ep\ln\frac{1}{\ep} \right)_+^{p-1}  \left(\mathcal F(x)+\phi_\ep^{o}(x)+\bar \phi_\ep(x)+O(\ep^2 |\ln\ep| )\right) dx\\
				+&O\left((\|\mathcal F\|_{L^\infty(B_{s_\ep}(x_\ep))}+\|\phi_\ep\|_{L^\infty(B_{s_\ep}(x_\ep))}+O(\ep^2 |\ln\ep|))^2+\ep(\|\mathcal F\|_{L^\infty(B_{s_\ep}(x_\ep))}+\|\phi_\ep\|_{L^\infty(B_{s_\ep}(x_\ep))}+O(\ep^2 |\ln\ep|)) +d_\ep^{p+1}\right)\\
				=&\frac{\Lambda_p}{x_{\ep,1}}x_{\ep,1}^{-\frac{2}{p-1}}\left(\frac{\ep}{s_\ep}\right)^{\frac{2}{p-1}} +O\left(\|\bar{\phi}_\ep\|_{L^\infty}+\ep^2|\ln\ep|+(\|\phi_\ep\|_{L^\infty}+\|\mathcal F\|_{L^\infty})^{2}+\ep(\|\phi_\ep\|_{L^\infty}+\|\mathcal F\|_{L^\infty})  \right).
			\end{align*}
	\end{small}

Then the desired estimate for $-\frac{x_{\ep,1}^2}{\varepsilon^2}\int_{B_\delta(x_\ep)} \left(\psi_\ep(x)-\frac{W}{2}x_1^2\ln\frac{1}{\ep}-\mu_\ep\right)_+^p\partial_1\psi_{2,\varepsilon}(  x)   d x$ follows by a combination of \eqref{C-7} and \eqref{C-10}--\eqref{D-6}. The proof is thus finished.
\end{proof}

From the Pohozaev identity \eqref{2-27}, \eqref{C-6} and Lemmas \ref{lemC-2} and    \ref{lemC-3}, we obtain a relation of $\kappa$, $W$, $s_\ep$ and $x_{\ep,1}$, which has been used to derive Kelvin--Hicks formula in Section \ref{sec2} several times. We summarize this result as follows.
\begin{lemma}\label{lemC-4}
	It holds
	\begin{align*}
			&Wx_{\ep,1}\ln\frac{1}{\varepsilon}-\frac{\kappa}{4\pi}\ln\frac{8x_{\ep,1}}{s_\ep}+\frac{(1-p)\kappa}{16\pi}\\
			=&O\left(\|\bar{\phi}_\ep\|_{L^\infty}+\ep^2|\ln\ep|+(\|\phi_\ep\|_{L^\infty}+\|\mathcal F\|_{L^\infty})^{2}+\ep(\|\phi_\ep\|_{L^\infty}+\|\mathcal F\|_{L^\infty})  \right).
		\end{align*}
\end{lemma}

\section{An existence theorem for the linearized operator}\label{appD}
In this appendix, 	we consider the problem
\begin{equation}\label{D-1}
	\begin{cases}
		-\Delta \phi-pU_+^{p-1}\phi =h- b_hU_+^{p-1}\frac{\partial U}{\partial x_1}\quad \text{in}\ \  \mathbb{R}^2\\
		\int_{\mathbb{R}^2} \phi(x) U_+^{p-1}(x)\frac{\partial U}{\partial x_1}(x) dx =0\\
		|\phi(x)|\to 0,\quad \text{as}\  \  |x|\to+\infty,
	\end{cases}
\end{equation}
where $b_h$ is the constant defined by
\begin{equation}\label{D-2}
	b_h:=\frac{\int_{\mathbb{R}^2} h(x)  \frac{\partial U}{\partial x_1}(x) dx}{\int_{\mathbb{R}^2}  U_+^{p-1}(x)\left(\frac{\partial U}{\partial x_1}(x)\right)^2 dx}.
\end{equation}
For a function $f$ and $\nu\in [0,1]$, we define the weight norm
\begin{equation}\label{D-3}
	\|f\|_\nu:=\sup_{x\in \mathbb{R}^2} (|x|+1)^\nu |f(x)|.
\end{equation}

We have the estimate
\begin{lemma}\label{lemD-1}
	Suppose that $\mathrm{supp}(h)\subset B_2(0)$,  $\|h\|_ {W^{-1,q}(B_2(0))}<+\infty$ for some $q\in(2,+\infty]$ and  $h(x_1,x_2)$ is odd in $x_1$ and even in $x_2$. Let $(\phi, b)$ be a solution of \eqref{D-1} and \eqref{D-2} such that $\phi(x_1,x_2)$ is odd in $x_1$ and even in $x_2$.
	Then, it holds
	$$\|\phi\|_1+\|\nabla \phi\|_{L^q(B_2(0))}+|b|\leq C \|h\|_ {W^{-1,q}(B_4(0))},$$
	for some constant $C>0$.
\end{lemma}
\begin{proof}
	By the definition of $b$, it is easy to see $|b|\leq C \|h\|_ {W^{-1,q}(B_2(0))}$. Suppose that there exist $h_n$ and $(\phi_n, b_n)$ such that the following conditions hold
	\begin{itemize}
		\item[i).] $h_n$ and $(\phi_n, b_n)$ satisfy \eqref{D-1} and \eqref{D-2},
		\item[ii).] $h_n(x_1,x_2)$ and $\phi_n$ are odd in $x_1$ and even in $x_2$, more over $\mathrm{supp}(h_n)\subset B_2(0)$,
		\item[iii).] $\|\phi_n\|_1+\|\nabla \phi_n\|_{L^q(B_2(0))}=1$, \  \   $\|h_n\|_ {W^{-1,q}(B_4(0))}\leq \frac{1}{n}$.
	\end{itemize}
	Since $supp(h_n)\subset B_2(0)$ and  $\|h_n\|_ {W^{-1,q}(B_4(0))}\leq \frac{1}{n}$, one can verify that  $\|h_n\|_ {W^{-1,q}(B_L(0))}\leq \frac{C}{n}$ for some $C$ independent of $L$ and any $L>0$. Thus, we infer from the elliptic estimate that $\|\phi_n\|_ {W^{ 1,q}(B_L(0))}\leq C$ for some $C$ independent of $L$ and hence $\|\phi_n\|_ {C^\alpha(B_L(0))}\leq C$ by Sobolev embedding. Therefore, we may assume that $\phi_n\to \phi$ uniformly in any compact set for some $\phi\in L^\infty(\mathbb{R}^2)\cap C (\mathbb{R}^2)$. Moreover, since $\|h_n\|_ {W^{-1,q}(B_L(0))}\leq \frac{C}{n}$ and $|b_n| \leq C \|h_n\|_ {W^{-1,q}(B_4(0))}$, we find $\phi$ is a solution to
	$$-\Delta \phi-pU_+^{p-1}\phi =0.$$
	So, by the non-degeneracy Lemma \ref{lem2-8}, we find
	$$\phi=c_1 \frac{\partial U}{\partial x_1}+ c_2 \frac{\partial U}{\partial x_2}.$$
	By the even symmetry of $\phi$ in $x_2$, we have $c_2=0$. By the condition $\int_{\mathbb{R}^2} \phi(x) U_+^{p-1}(x)\frac{\partial U}{\partial x_1}(x) dx =\lim_{n\to+\infty}\int_{\mathbb{R}^2} \phi_n(x) U_+^{p-1}(x)\frac{\partial U}{\partial x_1}(x) dx =0$, we obtain $c_1=0$. That is, $\phi\equiv 0$ and therefore $\|\phi_n\|_{L^\infty(B_4(0))}=o_n(1)$. By the integral equation and the odd symmetry, for $|x|>4$, we derive
	\begin{align*}
			\phi_n(x)&=\frac{1}{2\pi} \int_{B_2(0)} \ln \frac{1}{|x-y|} \left(pU_+^{p-1}(y)\phi_n(y) +h_n(y)- b_nU_+^{p-1}(y)\frac{\partial U}{\partial x_1}(y)\right) dy\\
			&=\frac{1}{2\pi} \int_{B_2(0)} \ln \frac{|x|}{|x-y|} \left(pU_+^{p-1}(y)\phi_n(y) +h_n(y)- b_nU_+^{p-1}(y)\frac{\partial U}{\partial x_1}(y)\right) dy\\
			&=\frac{1}{4\pi} \int_{B_2(0)} \left(\frac{2x\cdot y}{|x|^2}-\frac{|y|^2}{|x|^2}\right) \left(pU_+^{p-1}(y)\phi_n(y) +h_n(y)- b_nU_+^{p-1}(y)\frac{\partial U}{\partial x_1}(y)\right) dy\\
			&=o_n\left(\frac{1}{|x|}\right),
		\end{align*}
	which implies $\|\phi_n\|_1=o_n(1)$.  For $\varphi\in W_0^{1,q'}(B_2(0))$,
		\begin{align*}
				\int_{\mathbb{R}^2} \nabla \phi_n \cdot \nabla \varphi =\int_{ B_2(0)} \left(pU_+^{p-1} \phi_n  +h_n - b_nU_+^{p-1} \frac{\partial U}{\partial x_1} \right) \varphi=o_n(\|\varphi\|_{W^{1,q'}(B_2(0))})=o_n(\|\nabla\varphi\|_{L^{ q'}(B_2(0))}),
			\end{align*}
	which shows $\|\nabla \phi_n\|_{L^q(B_2(0))}=o_n(1)$.
	Therefore, we obtain $\|\phi_n\|_1+\|\nabla \phi_n\|_{L^q(B_2(0))}=o_n(1)$. This is a contradiction and hence we complete the proof.
\end{proof}

With the a priori estimate Lemma \ref{lemD-1}, we can prove the following existence theorem by the Fredholm alternative.
\begin{theorem}\label{thmD-2}
	Suppose that $supp(h)\subset B_2(0)$,  $\|h\|_ {W^{-1,q}(B_2(0))}<+\infty$ for some $q\in(2,+\infty]$ and  $h(x_1,x_2)$ is odd in $x_1$ and even in $x_2$. Then, there exists a uniqueness solution $\phi$ to \eqref{D-1}. In addition, the following estimate holds
	$$\|\phi\|_1+\|\nabla \phi\|_{L^q(B_2(0))}\leq C \|h\|_ {W^{-1,q}(B_4(0))},$$
	for some constant $C>0$.
\end{theorem}
\begin{proof}
	Define the weighted space
	\begin{equation*}
		\begin{split}X:=&\left\{\phi\in L^\infty(\mathbb{R}^2) \bigg|  \|\phi \|_{\frac{1}{2}}<+\infty, \int_{\mathbb{R}^2} \phi(x) U_+^{p-1}(x)\frac{\partial U}{\partial x_1}(x) dx =0, \right. \\ &\qquad \qquad\qquad  \ \ \phi(-x_1,x_2)=-\phi(x_1,x_2),  \phi( x_1,-x_2)=\phi(x_1,x_2),\ \forall\ (x_1,x_2)\in \mathbb{R}^2\bigg\}.\end{split}
	\end{equation*}
	One see that $\phi$ has the integral representation by \eqref{D-1}
	\begin{equation}\label{D-4}
		\phi=(-\Delta)^{-1}(pU_+^{p-1}\phi)+(-\Delta)^{-1}\left(h- bU_+^{p-1}\frac{\partial U}{\partial x_1}\right),
	\end{equation}
	where $(-\Delta)^{-1} f:=\frac{1}{2\pi} \int_{\mathbb{R}^2} \ln \frac{1}{|x-y|} f(y)dy$ for a function $f$. Using the odd symmetry, it can be seen that $(-\Delta)^{-1}(pU_+^{p-1}\phi), (-\Delta)^{-1}\left(h- bU_+^{p-1}\frac{\partial U}{\partial x_1}\right)\in X$ and  $\|\phi\|_1<\infty$. In addition, by elliptic estimates, $\|\phi\|_1<\infty$ and Arzela-Ascoli's theorem, we see that $(-\Delta)^{-1}(pU_+^{p-1} (\cdot)):X\to X$ is a compact operator. By Lemma \ref{lemD-1}, the homogeneous equation $phi=(-\Delta)^{-1}(pU_+^{p-1}\phi)$ has only trivial solution in $X$. Therefore, the equation \eqref{D-4} has a unique solution for every $h$ satisfying the hypotheses of Theorem \ref{thmD-2} by the Fredholm alternative. The estimate  $\|\phi\|_1+\|\nabla \phi\|_{L^q(B_2(0))}\leq C \|h\|_ {W^{-1,q}(B_4(0))}$ is a consequence of Lemma \ref{lemD-1}.
\end{proof}

If we take $h_\ep(x):=pU_+^{p-1}(x)\mathcal F(s_\ep x+x_\ep)$ where $\mathcal F$ is defined by \eqref{B-1}, then one can see $\|h_\ep\|_ {W^{-1,q}(B_4s(0))}=O(\ep)$ and hence the constant $b_\ep:=b_{h_\ep}$ defined by \eqref{D-2} is also of   $O(\ep)$, which is not small enough. Fortunately, by careful calculations, the estimate of $b_\ep$ can be improved. The key observation is that the function $\mathcal F(x)$ is odd in $x_1-x_{\ep,1}$ in the sense that $$\mathcal F((-x_1,x_2)+x_\ep)=-\mathcal F((x_1,x_2)+x_\ep),$$ and we can \emph{eliminate this term by such odd symmetry}. We also divide the error function into a form of ``odd function + higher order term". Suppose that $\phi_\ep$ takes the following form
\begin{equation}\label{D-5}
	\phi_\ep=\phi_\ep^{o}+\bar \phi_\ep,
\end{equation}
where $\phi_\ep^{o}$ is a function  odd in $x_1-x_{\ep,1}$ in the sense that $$\phi_\ep^{o}((-x_1,x_2)+x_\ep)=-\phi_\ep^{o}((x_1,x_2)+x_\ep).$$

We   give a better estimate of $b_\ep$.
\begin{lemma}\label{lemD-3}
	Suppose $h_\ep(x)=pU_+^{p-1}(x)\mathcal F(s_\ep x+x_\ep)$. Then, for $$b_\ep:=\frac{\int_{\mathbb{R}^2} h_\ep(x)  \frac{\partial U}{\partial x_1}(x) dx}{\int_{\mathbb{R}^2}  U_+^{p-1}(x)\left(\frac{\partial U}{\partial x_1}(x)\right)^2 dx},$$ the following estimate holds
	\begin{equation}\label{D-7}
		|b_\ep|=O\left(\ep|\ln\ep|(\|\bar{\phi}_\ep\|_{L^\infty}+\ep^2|\ln\ep|+ \|\phi_\ep\|_{L^\infty}^2+\ep \|\phi_\ep\|_{L^\infty} ) \right)=O(\ep^2|\ln\ep|).
	\end{equation}
\end{lemma}
\begin{proof}
	Recall the definition of $\mathcal F$
	\begin{align*}
			\mathcal F(x)&=(x_1-x_{\ep,1})\left[\frac{1}{2 x_{\ep,1}} U_{\ep, x_\ep,  a_{\ep}}(x)-Wx_{\ep,1}\ln\frac{1}{\ep}+\frac{x_{\ep,1}}{4\pi }\left(\ln (8 x_{\ep,1})-1\right) x_{\ep,1}^{-\frac{2p}{p-1}} \left(\frac{\ep}{s_\ep}\right)^{\frac{2}{p-1}} \Lambda_p\right]\\
			&+\frac{3x_{\ep,1}}{4\pi \ep^2}\int_{\mathbb{R}^2_+} (y_1-x_{\ep,1})\ln \frac{s_\ep}{|x-y|}\left(U_{\ep, x_\ep, a_{\ep}}(y)- a_{\ep}\ln \frac{1}{\ep}\right)_+^p dy.
		\end{align*}
	We split the integral into four terms.
		\begin{align*}
				&\int_{\mathbb{R}^2} h_\ep(x)  \frac{\partial U}{\partial x_1}(x) dx\\
				=&\frac{x_{\ep,1}^2 s_\ep p}{\ep^2}\int_{\mathbb{R}^2} \mathcal F(x)\left(U_{\ep, x_\ep, a_{\ep}}(x)-a_\ep\ln\frac{1}{\ep}  \right)_+^{p-1}\partial_{x_1} U_{\ep, x_\ep,  a_{\ep}}(x)dx\\
				=& -\frac{x_{\ep,1}^2 s_\ep  }{\ep^2}\int_{\mathbb{R}^2} \partial_{x_1}\mathcal F(x)\left(U_{\ep, x_\ep, a_{\ep}}(x)-a_\ep\ln\frac{1}{\ep}  \right)_+^{p }  dx\\
				=& -\frac{x_{\ep,1}^2 s_\ep  }{\ep^2}\int_{\mathbb{R}^2}  \frac{1}{2x_{\ep,1}}\left(U_{\ep, x_\ep, a_{\ep}}(x)-a_\ep\ln\frac{1}{\ep}  \right)_+^{p+1}  dx\\
				&-\frac{x_{\ep,1}^2 s_\ep  }{\ep^2}\int_{\mathbb{R}^2}\left(U_{\ep, x_\ep,  a_{\ep}}(x)-a_\ep\ln\frac{1}{\ep}  \right)_+^{p } \left(\frac{a_\ep\ln\frac{1}{\ep} }{2 x_{\ep,1}}  -Wx_{\ep,1}\ln\frac{1}{\ep}+\frac{x_{\ep,1}}{4\pi }\left(\ln (8 x_{\ep,1})-1\right) x_{\ep,1}^{-\frac{2p}{p-1}} \left(\frac{\ep}{s_\ep}\right)^{\frac{2}{p-1}} \Lambda_p\right)\\
				&-\frac{x_{\ep,1}^2 s_\ep  }{\ep^2}\int_{\mathbb{R}^2}\frac{(x_1-x_{\ep,1}) }{2 x_{\ep,1}}\left(U_{\ep, x_\ep, a_{\ep}}(x)-a_\ep\ln\frac{1}{\ep}  \right)_+^{p }\partial_{x_1} U_{\ep, x_\ep, a_{\ep}}(x)dx\\
				&-\frac{x_{\ep,1}^2 s_\ep  }{\ep^2}\frac{3x_{\ep,1}}{4\pi \ep^2}\int_{\mathbb{R}^2}\int_{\mathbb{R}^2_+} (y_1-x_{\ep,1})  \frac{y_1-x_1}{|x-y|^2}\left(U_{\ep, x_\ep, a_{\ep}}(x)-a_\ep\ln\frac{1}{\ep}  \right)_+^{p }\left(U_{\ep, x_\ep,a_{\ep}}(y)- a_{\ep}\ln \frac{1}{\ep}\right)_+^p dy dx\\
				=:&- x_{\ep,1}^2 s_\ep  (J_1+J_2+J_3+J_4).
			\end{align*}
	We estimate $J_1-J_4$ separately. For $J_1$, by the identity \eqref{2-14}, the definition \eqref{2-19} and \eqref{D-6}, we have	
	\begin{align*}
			J_1&=\frac{1}{\ep^2}\int_{\mathbb{R}^2}  \frac{1}{2x_{\ep,1}}\left(U_{\ep, x_\ep,  a_{\ep}}(x)-a_\ep\ln\frac{1}{\ep}  \right)_+^{p+1}  dx=\frac{1}{2x_{\ep,1}} x_{\ep,1}^{-\frac{2(p+1)}{p-1}}\left(\frac{\ep}{s_\ep}\right)^{ \frac{2(p+1)}{p-1}-2}\int_{\mathbb{R}^2}  U_+^{p+1} dx\\
			&=\frac{(p+1)x_{\ep,1}^{-1}}{16\pi}\kappa^2+O\left(\|\bar{\phi}_\ep\|_{L^\infty}+\ep^2|\ln\ep|+(\|\phi_\ep\|_{L^\infty}+\|\mathcal F\|_{L^\infty})^{2}+\ep(\|\phi_\ep\|_{L^\infty}+\|\mathcal F\|_{L^\infty})   \right).
		\end{align*}
	By integration by part, \eqref{2-14} and \eqref{D-6}, we calculate
		\begin{align*}
				J_2&=\frac{1}{\ep^2}\int_{\mathbb{R}^2}\left(U_{\ep, x_\ep,  a_{\ep}}(x)-a_\ep\ln\frac{1}{\ep}  \right)_+^{p } \left(\frac{a_\ep\ln\frac{1}{\ep} }{2 x_{\ep,1}}  -Wx_{\ep,1}\ln\frac{1}{\ep}+\frac{x_{\ep,1}}{4\pi }\left(\ln (8 x_{\ep,1})-1\right) x_{\ep,1}^{-\frac{2p}{p-1}} \left(\frac{\ep}{s_\ep}\right)^{\frac{2}{p-1}} \Lambda_p\right)\\
				&=  \left(\frac{a_\ep\ln\frac{1}{\ep} }{2 x_{\ep,1}}  -Wx_{\ep,1}\ln\frac{1}{\ep}+\frac{x_{\ep,1}}{4\pi }\left(\ln (8 x_{\ep,1})-1\right) x_{\ep,1}^{-\frac{2p}{p-1}} \left(\frac{\ep}{s_\ep}\right)^{\frac{2}{p-1}} \Lambda_p\right)x_{\ep,1}^{-\frac{2p}{p-1}} \left(\frac{\ep}{s_\ep}\right)^{\frac{2}{p-1}} \Lambda_p\\
				&=\left(\frac{\kappa }{4 \pi} \ln\frac{1}{s_\ep} -Wx_{\ep,1}\ln\frac{1}{\ep}+\frac{\kappa}{4\pi }\left(\ln (8 x_{\ep,1})-1\right) x_{\ep,1}^{-\frac{2p}{p-1}} \left(\frac{\ep}{s_\ep}\right)^{\frac{2}{p-1}} \Lambda_p\right)x_{\ep,1}^{-1}\kappa\\
				&\quad+O\left(|\ln\ep|(\|\bar{\phi}_\ep\|_{L^\infty}+\ep^2|\ln\ep|+(\|\phi_\ep\|_{L^\infty}+\|\mathcal F\|_{L^\infty})^{2}+\ep(\|\phi_\ep\|_{L^\infty}+\|\mathcal F\|_{L^\infty}))   \right).
			\end{align*}
	Similarly, for $J_3$, we derive by integration by part that 	
	\begin{align*}
			J_3&=\frac{1}{\ep^2}\int_{\mathbb{R}^2}\frac{(x_1-x_{\ep,1}) }{2 x_{\ep,1}}\left(U_{\ep, x_\ep,  a_{\ep}}(x)-a_\ep\ln\frac{1}{\ep}  \right)_+^{p }\partial_{x_1} U_{\ep, x_\ep, a_{\ep}}(x)dx\\
			&=-\frac{1}{\ep^2}\frac{1 }{2(p+1) x_{\ep,1}}\int_{\mathbb{R}^2}\left(U_{\ep, x_\ep, a_{\ep}}(x)-a_\ep\ln\frac{1}{\ep}  \right)_+^{p +1}dx\\
			&=-\frac{x_{\ep,1}^{-1}}{16\pi} \kappa^2+O\left(\|\bar{\phi}_\ep\|_{L^\infty}+\ep^2|\ln\ep|+(\|\phi_\ep\|_{L^\infty}+\|\mathcal F\|_{L^\infty})^{2}+\ep(\|\phi_\ep\|_{L^\infty}+\|\mathcal F\|_{L^\infty})   \right).
		\end{align*}
	As for $J_4$, we apply the even symmetries to deduce that
	\begin{align*}
			J_4&=\frac{3x_{\ep,1}}{4\pi \ep^4}\int_{\mathbb{R}^2}\int_{\mathbb{R}^2_+} (y_1-x_{\ep,1})  \frac{y_1-x_1}{|x-y|^2}\left(U_{\ep, x_\ep,  a_{\ep}}(x)-a_\ep\ln\frac{1}{\ep}  \right)_+^{p }\left(U_{\ep, x_\ep, a_\ep}(y)- a_\ep\ln \frac{1}{\ep}\right)_+^p dy dx\\
			&=\frac{3x_{\ep,1}}{8\pi \ep^4}\int_{\mathbb{R}^2}\int_{\mathbb{R}^2_+}   \frac{(y_1-x_1)^2}{|x-y|^2}\left(U_{\ep, x_\ep,  a_{\ep}}(x)-a_\ep\ln\frac{1}{\ep}  \right)_+^{p }\left(U_{\ep, x_\ep, a_\ep}(y)- a_\ep\ln \frac{1}{\ep}\right)_+^p dy dx\\
			&=\frac{3x_{\ep,1}}{16\pi \ep^4}\int_{\mathbb{R}^2}\int_{\mathbb{R}^2_+}  \left(U_{\ep, x_\ep, a_{\ep}}(x)-a_\ep\ln\frac{1}{\ep}  \right)_+^{p }\left(U_{\ep, x_\ep, a_\ep}(y)- a_\ep\ln \frac{1}{\ep}\right)_+^p dy dx\\
			&=\frac{3x_{\ep,1}^{-1}}{16\pi} \kappa^2+O\left(\|\bar{\phi}_\ep\|_{L^\infty}+\ep^2|\ln\ep|+(\|\phi_\ep\|_{L^\infty}+\|\mathcal F\|_{L^\infty})^{2}+\ep(\|\phi_\ep\|_{L^\infty}+\|\mathcal F\|_{L^\infty})   \right).
		\end{align*}
	Summarizing the above calculation and noticing that $ \|\mathcal F\|_{L^\infty(B_{2s_\ep}(x_\ep))}=O(\ep)$ by \eqref{B-3}, we obtain the desired estimate \eqref{D-7} and finish the proof of this lemma.
\end{proof}


\begin{thebibliography}{10}	
	
\bibitem{Abe22}
K. Abe, Existence of vortex rings in Beltrami flows.
\textit{ Comm. Math. Phys.}, \textbf{391} (2022), no. 2, 873--899.
https://doi.org/10.1007/s00220-022-04331-y

\bibitem{Abe}
K. Abe and K. Choi,
 Stability of Lamb dipoles,
\textit{  Arch. Ration. Mech. Anal.}, $\boldsymbol{244}$ (2022), 877--917. https://doi.org/10.1007/s00205-022-01782-4

\bibitem{Akh}
D. G. Akhmetov, Vortex Rings, Springer-Verlag, Berlin, Heidelberg, 2009.
	
\bibitem{AS}
A. Ambrosetti and M. Struwe, Existence of steady vortex rings in an ideal fluid, \textit{Arch. Ration. Mech. Anal}, $\boldsymbol{108}$ (2) (1989), 97--109. https://doi.org/10.1007/BF01053458	



\bibitem{AF}
C. J. Amick and L. E. Fraenkel, The uniqueness of Hill's spherical vortex,
\textit{Arch. Ration. Mech. Anal}, $\boldsymbol{92}$ (2) (1986), 91--119. https://doi.org/10.1007/BF00251252


\bibitem{AF88}
C. J. Amick and L. E. Fraenkel, The uniqueness of a family of steady vortex rings, \textit{Arch. Ration. Mech. Anal.}, $\boldsymbol{100}$ (1988), 207--241. https://doi.org/10.1007/BF00251515

\bibitem{AW}
W. Ao,  Y. Liu, J. Wei,   Clustered travelling vortex rings to the axisymmetric three-dimensional incompressible Euler flows, \textit{Phys. D}, $\boldsymbol{434}$ (2022), Paper No. 133258, 26 pp. https://doi.org/10.1016/j.physd.2022.133258


\bibitem{Ar2}
 V. I. Arnol'd, Sur la g\'eom\'etrie diff\'erentielle des groupes de Lie de dimension infinie et ses applications \`{a} l'hydrodynamique des fluides parfaits, \textit{ Ann. Inst. Fourier (Grenoble)}, $\boldsymbol{16}$ (1966),  319--361.   https://doi.org/10.5802/aif.233


\bibitem{Ar4}
V. I. Arnol'd and  B. A.  Khesin,  Topological Methods in Hydrodynamics,    \textit{Applied Mathematical Sciences, vol. 125},  Springer-Verlag, New York, 1998. xvi+374 pp.

\bibitem{Bad}
T. V. Badiani and  G. R.  Burton, Vortex rings in $\mathbb{R}^3$ and rearrangements, \textit{R. Soc. Lond. Proc. Ser. A Math. Phys. Eng. Sci.}, $\boldsymbol{457}$ (2001), no. 2009, 1115--1135. https://doi.org/10.1098/rspa.2000.0710

\bibitem{BCV}
J. Bedrossian,  M. Coti Zelati, V. Vicol,   Vortex axisymmetrization, inviscid damping, and vorticity depletion in the linearized 2D Euler equations, \textit{Ann. PDE}, \textbf{ 5} (2019), no. 1, Paper No. 4, 192 pp. https://doi.org/10.1007/s40818-019-0061-8

\bibitem{BM}
J. Bedrossian, N. Masmoudi, 
Inviscid damping and the asymptotic stability of planar shear flows in the 2D Euler equations, \textit{Publ. Math. Inst. Hautes ?tudes Sci.},  \textbf{122} (2015), 195--300.  https://doi.org/10.1007/s10240-015-0070-4


\bibitem{Ben}
T. B. Benjamin, The alliance of practical and analytical insights into the nonlinear problems of fluid mechanics, Applications of methods of functional analysis to problems in mechanics,   \textit{Lecture Notes in Math., vol. 503}, Springer, Berlin, 1976,
pp. 8--29.


\bibitem{Bu}
G. R. Burton, Rearrangements of functions, maximization of convex functionals  and vortex rings, \textit{Math. Ann.,} $\boldsymbol{276}$ (1987), no. 2, 225--253. https://doi.org/10.1007/BF01450739


\bibitem{Bu1}
G. R. Burton,  Variational problems on classes of rearrangements and multiple configurations for steady vortices,
\textit{Ann. Inst. H. Poincar\'e Anal. Non Lineair\'e}, $\boldsymbol{6}$ (1989), no. 4, 295--319. https://doi.org/10.1016/S0294-1449(16)30320-1


\bibitem{Bu5} G. R. Burton,  Global nonlinear stability for steady ideal fluid flow in bounded planar domains,  \textit{Arch. Ration. Mech. Anal.}, $\boldsymbol{176}$ (2005), no. 1, 149--163. https://doi.org/10.1007/s00205-004-0339-0

\bibitem{Bu6}
G. R. Burton,  Compactness and stability for planar vortex-pairs with prescribed impulse, \textit{J. Differential Equations}, $\boldsymbol{270}$ (2021), 547--572. https://doi.org/10.1016/j.jde.2020.08.009

\bibitem{BNL13}
G. R. Burton, H. J. Nussenzveig Lopes and M. C. Lopes Filho, Nonlinear stability for steady vortex pairs, \textit{Comm. Math. Phys.}, $\boldsymbol{324}$ (2013), 445--463. https://doi.org/10.1007/s00220-013-1806-y

\bibitem{CGPY}
D. Cao, Y. Guo, S. Peng and S. Yan, Local uniqueness for vortex patch problem in incompressible planar
steady flow, \textit{J. Math. Pures Appl.}, $\boldsymbol{131}$ (2009), 251--289. https://doi.org/10.1016/j.matpur.2019.
05.011.

\bibitem{CPYb}
D. Cao, S. Peng and S. Yan, Singularly Perturbed Methods for Nonlinear Elliptic Problems, Cambridge University Press, 2021.
https://doi.org/10.1017/9781108872638

\bibitem{CQYZZ}
D. Cao, G. Qin, W. Yu, W. Zhan and C. Zou, Existence, uniqueness and stability of steady vortex rings of small cross-section, preprint, arXiv:2201.08232.

\bibitem{CQZZ1}
D. Cao, G. Qin, W. Zhan and C. Zou, Existence and stability of smooth traveling circular pairs for the generalized surface quasi-geostrophic equation,   \textit{Int. Math. Res. Not.}. https://doi.org/10.1093/imrn/rnab371


\bibitem{CQZZ2}
D. Cao, G. Qin, W. Zhan and C. Zou, Uniqueness and stability of traveling vortex patch-pairs  for the incompressible Euler equation, preprint, submitted.


\bibitem{CQZZ}
D. Cao, G. Qin, W. Zhan and C. Zou, Remarks on orbital stability of steady vortex rings, \textit{Trans. Amer. Math. Soc.}, \textbf{ 376} (2023), no. 5, 3377--3395. https://doi.org/10.1090/tran/8888

\bibitem{CWZ}
D. Cao, J. Wan and W. Zhan, Desingularization of vortex rings in 3 dimensional Euler flows, \textit{J. Diff. Equat.}, $\boldsymbol{270}$ (2021), 1258--1297. https://doi.org/10.1016/j.jde.2020.09.014




\bibitem{CWWZ}
D. Cao, J. Wan, G. Wang and W. Zhan, Asymptotic behavior of global vortex rings, \textit{Nonlinearity}, \textbf{35} (2022), no. 7, 3680--3705. https://doi.org/10.1088/1361-6544/ac7497

\bibitem{CW}
D. Cao, G. Wang,  Nonlinear stability of planar vortex patches in an ideal fluid, \textit{J. Math. Fluid Mech.}, $\boldsymbol{23}$ (2021), no. 3, Paper No. 58, 16 pp. https://doi.org/10.1007/s00021-021-00588-w



\bibitem{CLO}
W. Chen, C. Li and B. Ou,  Classification of solutions for an integral equation, \textit{Comm. Pure Appl. Math.}, \textbf{59} (2006), 330--343.  https://doi.org/10.1002/cpa.20116

\bibitem{Choi20}
K. Choi, Stability of Hill's spherical vortex, \textit{Comm. Pure Appl. Math.}, \textbf{77} (2024), no. 1, 52--138. https://doi.org/10.1002/cpa.22134

\bibitem{CJ2}
K. Choi and  I. J. Jeong,  Stability and instability of Kelvin waves, \textit{Calc. Var. Partial Differential Equations}, \textbf{ 61} (2022), no. 6, Paper No. 221, 38 pp.
https://doi.org/10.1007/s00526-022-02334-0


\bibitem{CL}  K. Choi   and D. Lim,  Stability of radially symmetric, monotone vorticities of 2D Euler equations, \textit{Calc. Var. Partial Differential Equations}, \textbf{61} (2022), no. 4, Paper No. 120, 27 pp. https://doi.org/10.1007/s00526-022-02231-6




\bibitem{DR}
L. S. Da Rios, Sul moto d'un liquido indefinito con un filetto vorticoso di forma qualunque, \textit{Rendiconti del Circolo Matematico di Palermo (1884-1940)}, \textbf{22} (1906), no. 1, 117--135.

\bibitem{DL}
R. J. DiPerna,  P.-L. Lions,
Ordinary differential equations, transport theory and Sobolev spaces, \textit{Invent. Math.} \textbf{ 98} (1989), no. 3, 511--547. https://doi.org/10.1007/BF01393835

\bibitem{Dou}
R. J. Douglas,  Rearrangements of functions on unbounded domains, \textit{Proc. R. Soc. Edinb., Sect. A}, $\boldsymbol{124}$ (1994), 621--644. https://doi.org/10.1017/S0308210500028572



\bibitem{FW}
F. Flucher and J. Wei, Asymptotic shape and location of small cores in elliptic free-boundary problems,
\textit{Math. Z.}, \textbf{228} (1998), 683--703. https://doi.org/10.1007/PL00004636

\bibitem{Fra1}
L. E. Fraenkel, On steady vortex rings of small cross-section in an ideal fluid,
\textit{Proc. R. Soc. Lond. A.}, \textbf{316} (1970), 29--62. https://doi.org/10.1098/rspa.1970.0065


\bibitem{Fra2}
L. E. Fraenkel, Examples of steady vortex rings of small cross-section in an ideal fluid,
\textit{J. Fluid Mech.}, \textbf{51} (1972), 119--135. https://doi.org/10.1017/S0022112072001107

\bibitem{BF1}
L. E. Fraenkel and M. S. Berger, A global theory of steady vortex rings in an ideal fluid, \textit{Acta Math.},
\textbf{132} (1974), 13--51. https://doi.org/10.1007/BF02392107
	
\bibitem{FT}
A. Friedman and B. Turkington, Vortex rings: existence and asymptotic estimates, \textit{Trans. Amer. Math. Soc.}, \textbf{268}(1) (1981), 1--37. https://doi.org/10.1090/S0002-9947-1981-0628444-6

\bibitem{GSm}
T. Gallay, D. Smets,  Spectral stability of inviscid columnar vortices, \textit{Anal. PDE} \textbf{ 13} (2020), no. 6, 1777--1832. https://doi.org/10.2140/apde.2020.13.1777

 \bibitem{GS}
T. Gallay, V. S\v{v}er\'ak, Arnold's variational principle and its application to the stability of planar vortices, \textit{Anal. PDE}, to appear.

\bibitem{GT}
 D. Gilbarg and  N.  Trudinger, Elliptic Partial Differential Equations of Second Order, Reprint of the 1998 edition,  \textit{Classics in Mathematics}. Springer-Verlag, Berlin, 2001. xiv+517 pp.

\bibitem{Han}
Q. Han and F. Lin, Elliptic Partial Differential Equations, Second edition, Courant Lecture Notes in Mathematics, Courant Institute of Mathematical Sciences, New York; American Mathematical Society, Providence, RI, 2011.
	
\bibitem{Hel}
H. Helmholtz, On integrals of the hydrodynamics equations which express vortex
motion, \textit{J. Reine Angew. Math.}, \textbf{55} (1858), 25--55. https://doi.org/10.1515/crll.1858.55.25
(English translation by P. G. Tait: \textit{Phil. Mag.}, Ser. 4, \textbf{33} (1867), 485--512. https://doi.org/10.1080/14786446708639824)


\bibitem{Hil}
M. J. M. Hill, On a spherical vortex,
\textit{Philos. Trans. R. Soc. Lond. A}, \textbf{185} (1894), 213--245. https://doi.org/10.1098/rspl.1894.0032


\bibitem{IJ}
A. D. Ionescu, H. Jia,  
Axi-symmetrization near point vortex solutions for the 2D Euler equation, \textit{Comm. Pure Appl. Math.}, \textbf{ 75} (2022), no. 4, 818--891.  https://doi.org/10.1002/cpa.21974


\bibitem{JS}
R.L. Jerrard and C. Seis, On the vortex filament conjecture for Euler flows, \textit{Arch. Ration. Mech. Anal.}, \textbf{224} (2017), no. 1, 135--172. https://doi.org/10.1007/s00205-016-1070-3

\bibitem{Lam}
H. Lamb, Hydrodynamics, Cambridge Mathematical Library, 6th edn. Cambridge University Press, Cambridge (1932).

 \bibitem{LC}
T. Levi-Civita, Sull'attrazione esercitata da una linea materiale in punti prossimi alla linea stessa, \textit{Rend. R. Acc. Lincei}, \textbf{17} (1908), 3--15.	

\bibitem{Lieb}
E. H. Lieb and M. Loss,   Analysis.
\textit{Graduate Studies in Mathematics, Vol. 14.} American Mathematical Society, Providence, Second edition, RI (2001).

\bibitem{Lions}
P. L. Lions, The concentration-compactness principle in the calculus of variations, The locally compact case I, \textit{Ann. Inst. H. Poincar\'e Anal. Non Lin\'eaire.}, \textbf{1} (1984), no. 2, 109--145. https://doi.org/10.1016/S0294-1449(16)30428-0


\bibitem{MB}
A. Majda and A. Bertozzi, Vorticity and Incompressible Flow, Cambridge University Press, Cambridge, 2002.

\bibitem{MGT}
V. V. Meleshko, A. A. Gourjii and T. S. Krasnopolskaya,  Vortex ring: history and state of the art. \textit{J. of	Math. Sciences}, \textbf{187} (2012), 772--806. https://doi.org/10.1007/s10958-012-1100-0

\bibitem{Ni}
W. M. Ni, On the existence of global vortex rings,
\textit{J. Anal. Math.}, \textbf{37} (1980), 208--247. https://doi.org/10.1007/BF02797686

\bibitem{Nobi}
C. Nobili and C. Seis, Renormalization and energy conservation for axisymmetric fluid flows,   \textit{Math. Ann.}, \textbf{382} (2022),   1--36. https://doi.org/10.1007/s00208-020-02050-0.

\bibitem{Nor72}
J. Norbury, A steady vortex ring close to Hill's spherical vortex,
\textit{Proc. Camb. Philos. Soc.}, \textbf{72} (1972), 253--284. https://doi.org/10.1017/S0305004100047083



\bibitem{Tho}
W. Thomson (Lord Kelvin), Mathematical and Physical Papers, IV. Cambridge (1910).

\bibitem{T83}
B. Turkington,   On steady vortex flow in two dimensions. I, II.
\textit{Comm. Partial Differential Equations} $\boldsymbol{8}$ (1983), 999--1030, 1031--1071.  https://doi.org/10.1080/03605308308820293, https://doi.org/10.1080/03605308308820294


\bibitem{Sve}
V. S\v{v}er\'ak, Selected topics in fluid mechanics (an introductory graduate course taught in 2011/2012), available at the
following URL : http://www-users.math.umn.edu/\%7Esverak/course-notes2011.pdf.

\bibitem{VS}
S. de Valeriola and J. V. Schaftingen, Desingularization of vortex rings and shallow water vortices by a semilinear elliptic problem,
\textit{Arch. Ration. Mech. Anal.}, \textbf{210} (2013), 409--450. https://doi.org/10.1007/s00205-013-0647-3


\bibitem{Wan1}
Y.-H. Wan, Variational principles for Hill's spherical vortex and nearly spherical vortices, \textit{Trans. Amer. Math. Soc.}, \textbf{308} (1988), no. 1, 299--312. https://doi.org/10.1090/S0002-9947-1988-0946444-X

\bibitem{Wan}
Y. H. Wan and   M.   Pulvirenti, Nonlinear stability of circular vortex patches,
\textit{Comm. Math. Phys.}, \textbf{99} (1985), no. 3, 435--450. https://doi.org/10.1007/BF01240356


\bibitem{WZZ}
D. Wei, Z. Zhang, W. Zhao,  
Linear inviscid damping for a class of monotone shear flow in Sobolev spaces, \textit{Comm. Pure Appl. Math.}, \textbf{71} (2018), no. 4, 617--687. https://doi.org/10.1002/cpa.21672

\bibitem{YJ}
J. Yang, Global vortex rings and asymptotic behaviour,
\textit{Nonlinear Anal.}, \textbf{25} (1995), no. 5, 531--546. https://doi.org/10.1016/0362-546X(93)E0018-X

\end{thebibliography}
\end{document}